\documentclass[11pt]{amsart}
\usepackage[top=1.5in, bottom=1.5in, left=1.4in, right=1.4in]{geometry}
\RequirePackage{etex}

\usepackage{color}
\usepackage{graphicx}
\usepackage{amssymb}
\usepackage{epstopdf}
\usepackage{amsthm}
\usepackage{hyperref}
\usepackage[table,dvipsnames]{xcolor}

\usepackage{array}
\usepackage{mathtools}
\usepackage{enumerate}
\usepackage{lscape}
\usepackage{verbatim}
\usepackage{color}

\allowdisplaybreaks

\usepackage[all]{xy}
\usepackage{graphicx} 

\usepackage{tikz}
\usetikzlibrary{decorations.markings}
\usetikzlibrary{arrows}
\usetikzlibrary{arrows.meta}
\usetikzlibrary{bending}
\usetikzlibrary{decorations.pathreplacing,calc}
\usetikzlibrary{decorations.pathmorphing,backgrounds,positioning,fit,matrix}
\usepackage{tkz-graph}

\usepackage{cellspace} %
\def\scl{0.3}

\newtheorem{thm}{Theorem}[section]
\newtheorem{lemma}[thm]{Lemma}
\newtheorem{cor}[thm]{Corollary}
\newtheorem{prop}[thm]{Proposition}
\newtheorem{conj}[thm]{Conjecture}
\newtheorem{question}[thm]{Question}

\newtheorem{claim}{Claim}

\newtheorem{Definition}[thm]{Definition}
\newenvironment{definition}
  {\begin{Definition}\rm}{\end{Definition}}

\newtheorem{Example}[thm]{Example}
\newenvironment{example}
  {\begin{Example}\rm}{\end{Example}}

\newtheorem{Remark}[thm]{Remark}
\newenvironment{remark}
  {\begin{Remark}\rm}{\end{Remark}}

\numberwithin{equation}{section}

\usepackage{etoolbox}
\apptocmd{\sloppy}{\hbadness 10000\relax}{}{}
 
\newcommand{\oldpaperpermcontainment}{Proposition~2.2}
\newcommand{\oldpapersymprop}{Theorem~5.1}
\newcommand{\oldpaperconfluence}{Corollaries~9.2 and~11.5}
\newcommand{\oldpapernonescaping}{Lemma~8.2}
\newcommand{\oldpapersinks}{Lemmas~6.6 and~11.1}
\newcommand{\oldpaperconjecture}{Conjectures~1 and~2}
\newcommand{\oldpapertruncatedsymremark}{Remark~16.3}

\newcommand{\emailhref}[1]{\email{\href{#1}{#1}}}

\title[A positive formula for Ehrhart-like polynomials]{A positive formula for the Ehrhart-like polynomials from root system chip-firing}
\author[S. Hopkins]{Sam Hopkins}\emailhref{samuelfhopkins@gmail.com}
\author[A. Postnikov]{Alexander Postnikov}\emailhref{apost@math.mit.edu}
\address{Department of Mathematics, Massachusetts Institute of Technology, Cambridge, MA 02139, USA}

\date{\today}

\begin{document}

\begin{abstract}
In earlier work in collaboration with Pavel Galashin and Thomas McConville we introduced a version of chip-firing for root systems. Our investigation of root system chip-firing led us to define certain polynomials analogous to Ehrhart polynomials of lattice polytopes, which we termed the \emph{symmetric} and \emph{truncated Ehrhart-like polynomials}. We conjectured that these polynomials have nonnegative integer coefficients. Here we affirm ``half'' of this positivity conjecture by providing a positive, combinatorial formula for the coefficients of the symmetric Ehrhart-like polynomials. This formula depends on a subtle integrality property of slices of permutohedra, and in turn a lemma concerning dilations of projections of root polytopes, which both may be of independent interest. We also discuss how our formula very naturally suggests a conjecture for the coefficients of the truncated Ehrhart-like polynomials that turns out to be false in general, but which may hold in some cases.
\end{abstract}

\maketitle

\section{Introduction and statement of results} \label{sec:intro}

In~\cite{galashin2017rootfiring1} and~\cite{galashin2017rootfiring2}, together with Pavel Galashin and Thomas McConville, we introduced an analog of chip-firing for root systems. More specifically, in these papers we studied certain discrete dynamical processes whose states are the weights of a root system and whose transition moves consist of adding roots under certain conditions. We referred to these processes as \emph{root-firing} processes. Our investigation of root-firing was originally motivated by Jim Propp's labeled chip-firing game~\cite{hopkins2017sorting}: indeed, \emph{central root-firing}, which is the main subject of~\cite{galashin2017rootfiring2}, is exactly the same as Propp's labeled chip-firing when the root system is of Type A. But in~\cite{galashin2017rootfiring1} we instead focused on some remarkable deformations of central root-firing, which we called \emph{interval root-firing}, or just \emph{interval-firing} for short. It is these interval-firing processes which concern us in this present paper. So let us briefly review the definitions and key properties of these interval-firing processes.

Let $\Phi$ be an irreducible, crystallographic root system in an $n$-dimensional Euclidean vector space~$(V,\langle\cdot,\cdot\rangle)$, with weight lattice $P \coloneqq \{v \in V\colon \langle v,\alpha^\vee\rangle \in \mathbb{Z} \textrm{ for all $\alpha\in \Phi$}\}$ and set of positive roots~$\Phi^+$. Let $k \in \mathbb{Z}_{\geq 0}$ be any nonnegative integer. The \emph{symmetric interval-firing process} is the binary relation on~$P$ defined by
\[ \lambda \xrightarrow[\mathrm{sym},\,k]{} \lambda + \alpha \textrm{ whenever $\langle \lambda,\alpha^\vee\rangle+1 \in \{-k,-k+1,\ldots,k\}$ for $\lambda \in P$, $\alpha \in \Phi^+$};\]
and the \emph{truncated interval-firing process} is the binary relation on $P$ defined by
\[ \lambda \xrightarrow[\mathrm{tr},\,k]{} \lambda + \alpha \textrm{ whenever $\langle \lambda,\alpha^\vee\rangle + 1 \in \{-k+1,-k+2,\ldots,k\}$ for $\lambda \in P$, $\alpha \in \Phi^+$}.\]
(The intervals defining these processes are the same as those defining the extended $\Phi^\vee$-Catalan and extended $\Phi^\vee$-Shi hyperplane arrangements, and, although we have no precise statement to this effect, empirically it seems that the remarkable properties of these families of hyperplane arrangements~\cite{edelman1996free, postnikov2000deformations, athanasiadis2000deformations,terao2002multiderivations, yoshinaga2004characterization, abe2011freeness} are reflected in the interval-firing processes.) One should think of the ``$k$'' here as a deformation parameter: we are interested in understanding how these processes change as $k$ varies.

One of the main results of~\cite{galashin2017rootfiring1} is that for any $\Phi$ and any $k \in \mathbb{Z}_{\geq 0}$ these two interval-firing processes are both confluent (and terminating), meaning that there is always a unique stabilization starting from any initial weight~$\lambda \in P$; or in other words, in the directed graphs corresponding to these relations, each connected component contains a unique sink. It was also shown that these sinks are (a subset of) $\{\eta^k(\lambda)\colon \lambda \in P\}$ where~$\eta\colon P\to P$ is the piecewise-linear ``dilation'' operator on the weight lattice defined by $\eta(\lambda) \coloneqq \lambda + w_{\lambda}(\rho)$. Here $w_{\lambda} \in W$ is the minimal length element of the Weyl group~$W$ of $\Phi$ such that $w_{\lambda}^{-1}(\lambda)$ is dominant, and~$\rho \coloneqq \frac{1}{2}\sum_{\alpha \in \Phi^+} \alpha$ is the \emph{Weyl vector} of $\Phi$. Hence, it makes sense to define the stabilization maps $s^{\mathrm{sym}}_k,s^{\mathrm{tr}}_k\colon P \to P$ by
\begin{align*}
s^{\mathrm{sym}}_k(\mu) = \lambda &\Leftrightarrow \textrm{ the $\xrightarrow[\mathrm{sym},\,k]{}$-stabilization of $\mu$ is $\eta^k(\lambda)$}; \\
s^{\mathrm{tr}}_k(\mu) = \lambda &\Leftrightarrow \textrm{ the $\xrightarrow[\mathrm{tr},\,k]{}$-stabilization of $\mu$ is $\eta^k(\lambda)$}.
\end{align*}

These interval-firing processes turn out to be closely related to convex polytopes. For instance, it was observed in~\cite{galashin2017rootfiring1} that the set $(s^{\mathrm{sym}}_k)^{-1}(\lambda)$ (or $(s^{\mathrm{tr}}_k)^{-1}(\lambda)$) of weights with stabilization $\eta^k(\lambda)$ looks ``the same'' across all values of $k$ except that it gets ``dilated'' as $k$ is scaled. In analogy with the \emph{Ehrhart polynomial}~\cite{ehrhart1977polynomes} $L_{\mathcal{P}}(k)$ of a convex lattice polytope $\mathcal{P}$, which counts the number of lattice points in the $k$th dilate $k\mathcal{P}$ of the polytope, in~\cite{galashin2017rootfiring1} we began to investigate for all $\lambda \in P$ the quantities
\begin{align*}
L^{\mathrm{sym}}_{\lambda}(k) &\coloneqq \#(s^{\mathrm{sym}}_k)^{-1}(\lambda); \\
L^{\mathrm{tr}}_{\lambda}(k) &\coloneqq \#(s^{\mathrm{tr}}_k)^{-1}(\lambda),
\end{align*}
as functions of $k \in \mathbb{Z}_{\geq 0}$. It was shown in~\cite{galashin2017rootfiring1} that $L^{\mathrm{sym}}_{\lambda}(k)$ is a polynomial in $k$ for all~$\lambda \in P$ for any $\Phi$, and it was shown that $L^{\mathrm{tr}}_{\lambda}(k)$ is a polynomial in $k$ for all~$\lambda \in P$ assuming that $\Phi$ is simply-laced. Hence we refer to $L^{\mathrm{sym}}_{\lambda}(k)$ and $L^{\mathrm{tr}}_{\lambda}(k)$ as the \emph{symmetric} and \emph{truncated Ehrhart-like polynomials}, respectively. It was conjectured in~\cite{galashin2017rootfiring1} that for any $\Phi$ both $L^{\mathrm{sym}}_{\lambda}(k)$ and $L^{\mathrm{tr}}_{\lambda}(k)$ are polynomials in $k$ for all~$\lambda \in P$, and was moreover conjectured that these polynomials always have nonnegative integer coefficients. This positivity conjecture about Ehrhart-like polynomials connects root system chip-firing to the broader study of \emph{Ehrhart positivity} (see for instance the recent survey of Liu~\cite{liu2017positivity}). In fact, although the set of weights with a fixed interval-firing stabilization is in general not the set of lattice points in a convex polytope, these Ehrhart-like polynomials turn out to be closely related to Ehrhart polynomials of zonotopes, as discussed below.  %

In this paper we affirm ``half'' of the positivity conjecture: we show that the symmetric Ehrhart-like polynomials $L^{\mathrm{sym}}_{\lambda}(k)$ have nonnegative integer coefficients by providing an explicit, positive formula for these polynomials. We need just a bit more notation to write down this formula. Suppose the simple roots of $\Phi$ are $\alpha_1,\ldots,\alpha_{n}$. Recall that a weight $\lambda \in P$ is \emph{dominant} if $\langle\lambda,\alpha_i^\vee\rangle \geq 0$ for all $i \in [n] \coloneqq \{1,\ldots,n\}$. We denote the dominant weights by $P_{\geq 0}$. For $I \subseteq [n]$, let~$W_I$ denote the parabolic subgroup of~$W$ generated by the simple reflections $s_{\alpha_i}$ for~$i \in I$. For any weight $\lambda \in P$, let $\lambda_{\mathrm{dom}}$ denote the unique dominant element of~$W(\lambda)$. For a dominant weight $\lambda \in P_{\geq 0}$ we set~$I^{0,1}_{\lambda} \coloneqq \{i\in [n]\colon \langle \lambda,\alpha_i^\vee\rangle \in \{0,1\}\}$, and for an arbitrary weight $\lambda \in P$ we set $I^{0,1}_{\lambda} \coloneqq I^{0,1}_{\lambda_{\mathrm{dom}}}$. Let~$Q \coloneqq  \mathrm{Span}_{\mathbb{Z}}(\Phi)$ denote the \emph{root lattice} of $\Phi$.

Finally, for any linearly independent subset $X \subseteq \Lambda$ of a lattice $\Lambda$, we use $\mathrm{rVol}_{\Lambda}(X)$ to denote the \emph{relative volume} (with respect to~$\Lambda$) of $X$, which is the greatest common divisor of the maximal minors of the matrix whose columns are the coefficients expressing the elements of $X$ in some basis of $\Lambda$.\footnote{This is the relative volume, with respect to  $\Lambda$, of the paralellepiped $\sum_{v \in X} [0,v]$; it is also equal to the number of $\Lambda$-points in the half-open parallelepied $\sum_{\alpha \in X} [0,v)$ (see e.g.~\cite[Lemma 9.8]{beck2015computing}).}

Then we have the following:

\begin{thm} \label{thm:main_intro}
Let $\lambda \in P$ be any weight. Then $L^{\mathrm{sym}}_{\lambda}(k)=0$ if $\langle \lambda,\alpha^\vee\rangle = -1$ for some positive root~$\alpha \in \Phi^+$, and otherwise
\[L^{\mathrm{sym}}_{\lambda}(k) = \hspace{-0.5cm} \sum_{\substack{X\subseteq \Phi^+, \\ \textrm{$X$ is linearly} \\ \textrm{independent}}} \hspace{-0.5cm} \#\left\{\mu \in w_{\lambda}W_{I^{0,1}_{\lambda}} (\lambda_{\mathrm{dom}})\colon \parbox{1.6in}{\begin{center}$\langle\mu,\alpha^\vee\rangle \in \{0,1\}$ for \\ all $\alpha \in \Phi^+\cap \mathrm{Span}_{\mathbb{R}}(X)$\end{center}}\right\} \cdot \mathrm{rVol}_Q(X) \,\, k^{\#X}.\]
\end{thm}

Note that the condition $\langle \mu, \alpha^\vee\rangle \in \{0,1\}$ for all $\alpha \in \Phi^+\cap\mathrm{Span}_{\mathbb{R}}(X)$ is equivalent to saying that the orthogonal projection of $\mu$ onto $\mathrm{Span}_{\mathbb{R}}(X)$ is zero or a minuscule weight of the sub-root system $\Phi\cap\mathrm{Span}_{\mathbb{R}}(X)$. Thus, as one might expect, the combinatorics of minuscule weights and more generally the combinatorics of the partial order of dominant weights (which was first extensively investigated by Stembridge~\cite{stembridge1998partial}) features prominently in the proof of Theorem~\ref{thm:main_intro}.

The other major ingredient in the proof of Theorem~\ref{thm:main_intro} is a kind of extension of the Ehrhart theory of zonotopes. In 1980, Stanley~\cite{stanley1980decompositions} (see also~\cite{stanley1991zonotope}) proved that the Ehrhart polynomial~$L_{\mathcal{Z}}(k)$ of a lattice zonotope~$\mathcal{Z}$ has nonnegative integer coefficients; in fact, he gave the following explicit formula:
\begin{equation} \label{eqn:zonotope_formula}
L_{\mathcal{Z}}(k) = \sum_{\substack{X \subseteq \{v_1,\ldots,v_m\}, \\ X \textrm{ is linearly} \\ \textrm{independent}}}\mathrm{rVol}_{\mathbb{Z}^n}(X) \, k^{\#X},
\end{equation}
where $\mathcal{Z} \coloneqq \sum_{i=1}^{m}[0,v_i]$ is the Minkowski sum of the lattice vectors $v_1,\ldots,v_m \in \mathbb{Z}^n$. (See~\cite[\S9]{beck2015computing} for another presentation of this result.) In~\cite{galashin2017rootfiring1} we proved a slight extension of Stanley's result: we showed that for any convex lattice polytope $\mathcal{P}$ and any lattice zonotope $\mathcal{Z}$, the number of lattice points in $\mathcal{P} + k\mathcal{Z}$ is given by a polynomial with nonnegative integer coefficients in $k$. The case where $\mathcal{P}$ is a point recaptures Stanley's result. However, in~\cite{galashin2017rootfiring1} we did not give any explicit formula for the coefficients of the polynomial analogous to the formula~\eqref{eqn:zonotope_formula} for zonotopes. The first thing we need to do in the present paper is provide such a formula, whose simple proof we also go over now. In fact, this result is stated most naturally in its ``multi-parameter'' formulation:

\begin{thm} \label{thm:poly_plus_zono}
Let $\mathcal{P}$ be a convex lattice polytope in $\mathbb{R}^n$, and $v_1,\ldots,v_m \in \mathbb{Z}^n$ lattice vectors. Set $\mathcal{Z} \coloneqq \sum_{i=1}^{m} [0,v_i]$, and for $\mathbf{k} = (k_1,\ldots,k_m)\in\mathbb{Z}_{\geq 0}^m$ define $\mathbf{k}\mathcal{Z} \coloneqq\sum_{i=1}^{m}k_i[0,v_i]$. Then for any~$\mathbf{k} \in \mathbb{Z}^m_{\geq 0}$ we have
\[\#\left(\mathcal{P}+\mathbf{k}\mathcal{Z}\right)\cap\mathbb{Z}^n = \sum_{\substack{X \subseteq \{v_1,\ldots,v_m\}, \\ X \textrm{ is linearly} \\ \textrm{independent}}} \# \left(\mathrm{quot}_X(\mathcal{P}) \cap \mathrm{quot}_X(\mathbb{Z}^n)\right) \cdot \mathrm{rVol}_{\mathbb{Z}^n}(X) \, \mathbf{k}^{X}\]
where $\mathbf{k}^X \coloneqq \prod_{v_i \in X} k_i$ and $\mathrm{quot}_X\colon \mathbb{R}^n \twoheadrightarrow \mathbb{R}^n/\mathrm{Span}_{\mathbb{R}}(X)$ is the canonical quotient map.
\end{thm}
\begin{proof}
The standard proof of Stanley's formula for the Ehrhart polynomial of a lattice zonotope (and indeed the proof originally given by Stanley~\cite{stanley1980decompositions, stanley1991zonotope}) is via ``paving'' the zonotope, i.e., decomposing it into disjoint half-open parallelepipeds (see also~\cite[\S9]{beck2015computing}). This decomposition goes back to Shephard~\cite{shephard1974combinatorial}. In~\cite{galashin2017rootfiring1}, we explained how the technique of paving can be adapted to apply to $\mathcal{P}+\mathbf{k}\mathcal{Z}$ as well. But we can actually establish the claimed formula for $\#(\mathcal{P}+\mathbf{k}\mathcal{Z})\cap\mathbb{Z}^n$ just from some general properties of ``multi-parameter'' Ehrhart polynomials. We need only the following result:

\begin{lemma}[{McMullen~\cite[Theorem 6]{mcmullen1977valuations}}] \label{lem:multiparam_ehrhart}
Let $\mathcal{Q}_0,\mathcal{Q}_1,\ldots,\mathcal{Q}_m$ be convex lattice polytopes in $\mathbb{R}^n$. Then for nonnegative integers $k_0,\ldots,k_m \in \mathbb{Z}_{\geq 0}$, the number of lattice points in $k_0\mathcal{Q}_0 + \cdots +k_m\mathcal{Q}_m$ is a polynomial (with real coefficients) in $k_0,\ldots,k_m$ of total degree at most the dimension of the smallest affine subspace containing all of $\mathcal{Q}_0,\ldots,\mathcal{Q}_m$.
\end{lemma}

First of all, Lemma~\ref{lem:multiparam_ehrhart} immediately gives that $\#(\mathcal{P}+\mathbf{k}\mathcal{Z})\cap\mathbb{Z}^n$ is a polynomial in~$\mathbf{k}$: we can just take $\mathcal{Q}_0 \coloneqq \mathcal{P}$, $\mathcal{Q}_i \coloneqq [0,v_i]$ for $i \in [m]$, and set~$k_0 \coloneqq 1$. We use~$f(\mathbf{k})$ to denote this polynomial.

Now we check that each coefficient of $f(\mathbf{k})$ agrees with the claimed formula. So fix some~$\mathbf{a} = (a_1,\ldots,a_m) \in \mathbb{Z}_{\geq 0}^m$ and set $\mathbf{k}^{\mathbf{a}} \coloneqq \prod_{i=1}^{m}x_i^{a_i}$. We will check that the coefficient of $\mathbf{k}^{\mathbf{a}}$ is as claimed. By substituting $k_i \coloneqq 0$ for any~$i$ for which $a_i = 0$, we can assume that $a_i \neq 0$ for all $i \in [m]$. Set~$X \coloneqq \{v_1,\ldots,v_m\}$; the goal will be to show that the coefficient of $\mathbf{k}^{\mathbf{a}}$ is zero if $\mathbf{k}^{\mathbf{a}}\neq k^{X}$ or $X$ is not linearly independent. We can count the number of lattice points in $\mathcal{P}+\mathbf{k}\mathcal{Z}$ by dividing them into ``slices'' which lie in affine translates of~$\mathrm{Span}_{\mathbb{R}}(X)$. Accordingly, let $u_1,u_2,\ldots,u_{\ell} \in \mathbb{Z}^n$ be such that:
\begin{itemize}
\item $(u_i + \mathrm{Span}_{\mathbb{R}}(X) ) \cap \mathcal{P} \neq \varnothing$ for all $i \in [\ell]$;
\item if $(u + \mathrm{Span}_{\mathbb{R}}(X) )\cap \mathcal{P} \neq \varnothing$ for some $u \in \mathbb{Z}^n$, then $ u + \mathrm{Span}_{\mathbb{R}}(X) = u_i +\mathrm{Span}_{\mathbb{R}}(X)$ for some $i \in [\ell]$;
\item $u_i +\mathrm{Span}_{\mathbb{R}}(X) \neq u_j +\mathrm{Span}_{\mathbb{R}}(X)$ for $i\neq j \in [\ell]$.
\end{itemize}
Set $\mathcal{P}_i \coloneqq \mathcal{P}\cap (u_i+\mathrm{Span}_{\mathbb{R}}(X))$ for $i = 1,\ldots,\ell$ and observe that
\[(\mathcal{P}+\mathbf{k}\mathcal{Z})\cap \mathbb{Z}^n = \bigcup_{i=1}^{\ell} (\mathcal{P}_i+\mathbf{k}\mathcal{Z})\cap \mathbb{Z}^n.\]
Because $\mathcal{Z}$ is full-dimensional inside of $\mathrm{Span}_{\mathbb{R}}(X)$, for each $i \in [\ell]$ there is $c_i \in \mathbb{Z}_{\geq 0}$ such that $\mathcal{P}_i$ is contained, up to lattice translation, in $c_i\mathcal{Z}$. Hence we obtain the inequalities
\begin{equation} \label{eqn:slice_ineqs}
\sum_{i=1}^{\ell}\#\left(\mathbf{k}\mathcal{Z}\cap\mathbb{Z}^n\right) \leq f(\mathbf{k}) \leq \sum_{i=1}^{\ell}\#\left((\mathbf{k}+c_i)\mathcal{Z}\cap \mathbb{Z}^n\right),
\end{equation}
for all $\mathbf{k} \in \mathbb{Z}_{\geq 0}^{m}$. First consider the case where either $\mathbf{k}^{\mathbf{a}} \neq \mathbf{k}^{X}$, or $X$ is not linearly independent. Then~$\sum_{i=1}^{m}a_i$ is strictly greater than the dimension of $\mathrm{Span}_{\mathbb{R}}(X)$. But by Lemma~\ref{lem:multiparam_ehrhart} the left- and right-hand sides of~\eqref{eqn:slice_ineqs} are polynomials in $\mathbf{k}$ of degree at most the dimension of $\mathrm{Span}_{\mathbb{R}}(X)$. So in this case it must be that the coefficient of~$\mathbf{k}^{\mathbf{a}}$ in~$f(\mathbf{k})$ is zero. Now assume that $\mathbf{k}^{\mathbf{a}} = \mathbf{k}^{X}$ and $X$ is linearly independent. In this case, $\mathcal{Z}$ is a parallelepiped whose relative volume is in fact $\mathrm{rVol}_{\mathbb{Z}^n}(X)$: see for example~\cite[Lemma~9.8]{beck2015computing}. Hence, the leading coefficient of $\#(k\mathcal{Z}\cap \mathbb{Z}^n)$ as a polynomial in $k$ is $\mathrm{rVol}_{\mathbb{Z}^n}(X)$. By substituting $k \coloneqq k+c$ into this polynomial for some constant $c\in\mathbb{Z}_{\geq 0}$, we see that the leading coefficient of $\#((k+c)Z\cap \mathbb{Z}^n)$ as a polynomial in $k$ is also $\mathrm{rVol}_{\mathbb{Z}^n}(X)$. Hence when we make the substitution $k_i \coloneqq k$ for all $i \in [m]$, the leading coefficient of both the left- and right-hand sides of~\eqref{eqn:slice_ineqs} is~$\ell\cdot\mathrm{rVol}_{\mathbb{Z}^n}(X)$. Furthermore, the degree of both of these polynomials is the dimension of $\mathrm{Span}_{\mathbb{R}}(X)$, which is the same as $\#X=m$. By induction on $m$, we know that the coefficient of $\mathbf{k}^{\mathbf{b}}$ in~$f(\mathbf{k})$ is zero for any $\mathbf{b}=(b_1,\ldots,b_m) \in \mathbb{Z}^{m}_{\geq 0}$ with $\mathbf{b}\neq (1,1,\ldots,1)$ and $\sum_{i=1}^{m}b_i=m$. Thus we conclude that~$\ell\cdot\mathrm{rVol}_{\mathbb{Z}^n}(X)$ is also the coefficient of $\mathbf{k}^{\mathbf{a}}$ in~$f(\mathbf{k})$. But then note that $\ell = \# \left(\mathrm{quot}_X(\mathcal{P}) \cap \mathrm{quot}_X(\mathbb{Z}^n)\right)$, finishing the proof of the theorem.
\end{proof}

\begin{remark}
Although we will not need this, we believe that Theorem~\ref{thm:poly_plus_zono} holds verbatim in the case where $\mathcal{P}$ is a rational convex polytope (as in~\cite{mcmullen1978lattice}), or even in the case where $\mathcal{P}$ is an arbitrary convex polytope in $\mathbb{R}^n$.
\end{remark}

In general, the formula in Theorem~\ref{thm:poly_plus_zono} is not ideal from a combinatorial perspective because in order to compute the quantity $\# \left(\mathrm{quot}_X(\mathcal{P}) \cap \mathrm{quot}_X(\mathbb{Z}^n)\right)$ we have to consider every rational point in $\mathcal{P}$. But in particularly nice situations we may actually have that~$\mathrm{quot}_X(\mathcal{P}) \cap \mathrm{quot}_X(\mathbb{Z}^n) = \mathrm{quot}_X(\mathcal{P}\cap\mathbb{Z}^n)$ for all $X \subseteq \{v_1,\ldots,v_m\}$. In fact, this is exactly what happens in the case of Theorem~\ref{thm:poly_plus_zono} that is relevant to interval-firing: the Minkowski sum of a permutohedron and a dilating regular permutohedron.

For $\lambda \in P$, the \emph{permutohedron} corresponding to~$\lambda$ is $\Pi(\lambda) \coloneqq \mathrm{ConvexHull} \; W(\lambda)$. We use $\Pi^Q(\lambda) \coloneqq \Pi(\lambda)\cap(Q+\lambda)$ to denote the (root) lattice points in $\Pi(\lambda)$. The \emph{regular permutohedron} of $\Phi$ is $\Pi(\rho)$. Note that the regular permutohedron is in fact a zonotope: $\Pi(\rho) = \sum_{\alpha \in \Phi^+}[-\alpha/2,\alpha/2]$. Also note that if $\lambda \in P_{\geq 0}$ is a dominant weight, then $\Pi(\lambda + k\rho) = \Pi(\lambda)+k\Pi(\rho)$, so $\Pi(\lambda + k\rho)$ really is a polytope of the form~$\mathcal{P}+k\mathcal{Z}$. It is this polytope $\Pi(\lambda + k\rho)$ which is relevant to interval-firing. 

We will show that permutohedra satisfy the following subtle integrality property:

\begin{lemma} \label{lem:slice_intro}
Let $\lambda\in P_{\geq 0}$ and $X\subseteq \Phi^+$. Then
\[\mathrm{quot}_X(\Pi(\lambda)) \cap \mathrm{quot}_X(Q+\lambda) = \mathrm{quot}_X(\Pi^Q(\lambda)),\]
where $\mathrm{quot}_X$ is the canonical quotient map $\mathrm{quot}_X\colon V \twoheadrightarrow V/\mathrm{Span}_{\mathbb{R}}(X)$.
\end{lemma}

We prove Lemma~\ref{lem:slice_intro} in the second section of the paper. The proof turns out to be quite involved. In particular, we show that this lemma follows from a certain property of dilations of projections of root polytopes. Here the \emph{root polytope} $\mathcal{P}_{\Phi}$ of $\Phi$ is the convex hull of the roots: $\mathcal{P}_{\Phi} \coloneqq \mathrm{ConvexHull}(\Phi)$. Lemma~\ref{lem:slice_intro} follows (in a non-obvious way) from the following lemma about root polytopes:

\begin{lemma} \label{lem:root_polytope_projection_intro}
Let $\{0\}\neq U\subseteq V$ be a nonzero subspace of $V$ spanned by a subset of~$\Phi$. Set $\Phi_U\coloneqq\Phi\cap U$, a sub-root system of $\Phi$. Let $\pi_U\colon V\to U$ denote the orthogonal projection of $V$ onto $U$. Then there exists some $1\leq \kappa < 2$ such that $\pi_U(\mathcal{P}_{\Phi}) \subseteq \kappa \cdot \mathcal{P}_{\Phi_U}$.
\end{lemma}

Note that the constant $2$ in Lemma~\ref{lem:root_polytope_projection_intro} cannot be replaced with any smaller constant; and conversely, if we replaced $2$ with any larger constant then Lemma~\ref{lem:root_polytope_projection_intro} would no longer imply Lemma~\ref{lem:slice_intro}. Observe how Lemmas~\ref{lem:slice_intro} and~\ref{lem:root_polytope_projection_intro} are both formulated in a uniform way across all root systems $\Phi$. Moreover, our proof that Lemma~\ref{lem:root_polytope_projection_intro} implies Lemma~\ref{lem:slice_intro} is uniform (and indeed, we show that these two lemmas are ``almost'' equivalent; see Remark~\ref{rem:lem_equivalence}). However, we were unfortunately unable to find a uniform proof of Lemma~\ref{lem:root_polytope_projection_intro} and instead had to rely on the classification of root systems and a case-by-case analysis. We relegated the case-by-case check of Lemma~\ref{lem:root_polytope_projection_intro} to the appendix of the paper. We believe that both of these lemmas may be of independent interest. We leave it as an open problem to find uniform proofs of them.

Lemma~\ref{lem:slice_intro} together with Theorem~\ref{thm:poly_plus_zono} immediately implies the following:

\begin{thm} \label{thm:perm_intro}
Let $\lambda \in P_{\geq 0}$ be a dominant weight and $k \in \mathbb{Z}_{\geq 0}$. Then
\[ \#\Pi^Q(\lambda+k\rho) = \sum_{\substack{X\subseteq \Phi^+ \\ \textrm{ $X$ is linearly} \\ \textrm{independent}}} \# \mathrm{quot}_X(\Pi^Q(\lambda))\cdot \mathrm{rVol}_Q(X) \, k^{\#X}.\]
\end{thm}

The relevance of Theorem~\ref{thm:perm_intro} to interval-firing is that for any~$\lambda \in P_{\geq 0}$, the discrete permutohedron $\Pi^Q(\lambda+k\rho)$ is a disjoint union of connected components of the directed graph corresponding to $\xrightarrow[\mathrm{sym},\,k]{}$: this follows from the ``permutohedron non-escaping lemma,'' a key technical result in~\cite{galashin2017rootfiring1}. In fact, as we will see later, if $\lambda \in P_{\geq 0}$ is a dominant weight with $I^{0,1}_{\lambda} = [n]$, then
\begin{equation} \label{eqn:sym_components_intro}
(s^{\mathrm{sym}}_k)^{-1}(\lambda) = \Pi^Q(\lambda+k\rho) \setminus \bigcup_{\substack{\mu \neq \lambda \in P_{\geq 0}, \\ \mu \leq \lambda}}\Pi^Q(\mu+k\rho)
\end{equation}
where $\mu \leq \lambda$ means $\lambda - \mu = \sum_{i=1}^{n}a_i\alpha_i$ with all $a_i \in \mathbb{Z}_{\geq 0}$ (this is the partial order of dominant weights sometimes referred to as \emph{root order} or \emph{dominance order}). Moreover, for any $\lambda \in P$ which satisfies $\langle \lambda,\alpha^\vee \rangle \neq -1$ for all $\alpha \in \Phi^+$, we can express the fiber~$(s^\mathrm{sym}_k)^{-1}(\lambda)$ as a difference of permutohedra as in~\eqref{eqn:sym_components_intro} except that we may need to first project to a smaller-dimensional sub-root system of $\Phi$. Therefore, Theorem~\ref{thm:main_intro} follows from Theorem~\ref{thm:perm_intro} and Lemma~\ref{lem:slice_intro} via inclusion-exclusion, together with some fundamental facts about root order established by Stembridge~\cite{stembridge1998partial}.

At the end of the paper we also discuss the truncated Ehrhart-like polynomials. It was shown in~\cite{galashin2017rootfiring1} that for any $\lambda \in P_{\geq 0}$ with $I^{0,1}_{\lambda} = [n]$, we have
\[ (s^{\mathrm{sym}}_k)^{-1}(\lambda) = \bigcup_{\mu \in W (\lambda)} (s^{\mathrm{tr}}_k)^{-1}(\mu),\]
or at the level of Ehrhart-like polynomials,
\[ L^{\mathrm{sym}}_{\lambda}(k) = \sum_{\mu \in W(\lambda)} L^{\mathrm{tr}}_{\mu}(k).\]
Hence, the formula in Theorem~\ref{thm:main_intro} very naturally suggests the following conjecture:
\begin{conj} \label{conj:main_intro}
Let $\lambda \in P$ be any weight. Then
\[L^{\mathrm{tr}}_{\lambda}(k) = \sum_{X} \mathrm{rVol}_Q(X) \, k^{\#X},\]
where the sum is over all $X\subseteq \Phi^+$ such that:
\begin{itemize}
\item $X$ is linearly independent;
\item $\langle \lambda,\alpha^\vee\rangle \in \{0,1\}$ for all $\alpha \in \Phi^+\cap\mathrm{Span}_{\mathbb{R}}(X)$.
\end{itemize}
\end{conj}
However, in fact Conjecture~\ref{conj:main_intro} is {\bf false in general}! We discuss examples where Conjecture~\ref{conj:main_intro} fails, as well as some cases where it may possibly hold (such as Type~A and Type~B), in the last section. As it is, the truncated Ehrhart-like polynomials remain largely a mystery.

Here is a brief outline of the rest of the paper. In Section~\ref{sec:lattice_pts} we establish the subtle integrality property of permutohedra (Lemma~\ref{lem:slice_intro}), conditional on the root polytope projection-dilation lemma (Lemma~\ref{lem:root_polytope_projection_intro}). We go on in this section to prove the formula for the number of points in a permutohedron plus dilating regular permutohedron (Theorem~\ref{thm:perm_intro}). We also briefly discuss the specifics of what this formula looks like in Type~A. In Section~\ref{sec:sym_formula} we use our formula for the number of points in a permutohedron plus dilating regular permutohedron to establish the formula for the symmetric Ehrhart-like polynomials (Theorem~\ref{thm:main_intro}). In Section~\ref{sec:future} we discuss related questions and possible future directions, including the truncated Ehrhrat-like polynomials (specifically, Conjecture~\ref{conj:main_intro}). In Appendix~\ref{sec:root_polytopes} we prove the root polytope projection-dilation lemma in a case-by-case manner.

\medskip

\noindent {\bf Acknowledgments}: We thank Pavel Galashin and Thomas McConville, with whom we have had countless conversations about root system chip-firing in the past year, which invariably aided us in the present research. We also thank Federico Castillo and Fu Liu for some enlightening discussions concerning Ehrhart positivity. We thank Matthew Dyer for introducing us to ``Oshima's lemma''~\cite{oshima2006classification, dyer2018parabolic}, which describes orbit representatives for the roots under the action of a parabolic subgroup. And we thank Christian Stump for directing us to the work of Cellini and Marietti~\cite{cellini2015root}, which gives a uniform facet description of root polytopes. Finally, we thank the anonymous referees for a careful reading of our manuscript and many helpful comments. This material is based upon work supported by the National Science Foundation under Grant No.~1440140, while the authors were in residence at the Mathematical Sciences Research Institute in Berkeley, California, during the fall semester of 2017. The first author was also partially supported by NSF Grant No.~1122374. We used the Sage mathematical software system~\cite{sagemath, Sage-Combinat} for many computations during the course of this research. 

\section{Lattice points in the Minkowski sum of a permutohedron and a dilating regular permutohedron} \label{sec:lattice_pts}

\subsection{Background on root systems, sub-root sytstems, permutohedra, and root order} \label{subsec:root_background}

In this subsection we briefly review basics on root systems and collect some results about sub-root systems, permutohedra, and root order that we will need going forward. For a more detailed treatment with complete proofs for all the results mentioned, consult~\cite{bourbaki2002lie},~\cite{humphreys1972lie}, or~\cite{bjorner2005coxeter}. Generally speaking, we follow the notation from~\cite{galashin2017rootfiring1}.

Fix $V$, a $n$-dimensional Euclidean vector space with inner product~$\langle \cdot,\cdot\rangle$. For a nonzero vector $v \in V$, we define the \emph{covector} of $v$ to be $v^\vee\coloneqq\frac{2}{\langle v,v\rangle}\,v$ and then define the orthogonal \emph{reflection} across the hyperplane with normal vector $v$ to be the linear map $s_v\colon V \to V$ given by $s_v(u) \coloneqq u - \langle u,v^\vee\rangle v$ for all $u \in V$. A \emph{(crystallographic, reduced) root system} in $V$ is a finite collection $\Phi\subseteq V \setminus \{0\}$ of nonzero vectors satisfying:
\begin{itemize}
\item $\mathrm{Span}_{\mathbb{R}}(\Phi) = V$;
\item $s_\alpha(\Phi) = \Phi$ for all $\alpha \in \Phi$;
\item $\mathrm{Span}_{\mathbb{R}}(\{\alpha\})\cap\Phi = \{\alpha,-\alpha\}$ for all $\alpha\in \Phi$;
\item $\langle \beta,\alpha^\vee\rangle \in \mathbb{Z}$ for all $\alpha,\beta \in \Phi$.
\end{itemize}
From now on, fix such a root system $\Phi$ in $V$. The elements of $\Phi$ are called \emph{roots}. The dimension $n$ of $V$ is called the \emph{rank} of $\Phi$. We use $W$ to denote the \emph{Weyl group} of $\Phi$, which is the subgroup of $GL(V)$ generated by the reflections $s_{\alpha}$ for all roots $\alpha \in \Phi$. 

It is well-known that we can choose a set $\Phi^+$ of \emph{positive roots} with the properties that: if $\alpha,\beta \in \Phi^+$ and $\alpha+\beta\in\Phi$ then $\alpha+\beta\in \Phi^+$; and $\{\Phi^+,-\Phi^+\}$ is a partition of $\Phi$. The choice of set of positive roots~$\Phi^+$ is equivalent to a choice of \emph{simple roots}~$\alpha_1,\ldots,\alpha_{n}$, which have the properties that: the $\alpha_i$ form a basis of $V$; and every root is either a nonnegative or a nonpositive integral combination of the $\alpha_i$. Of course, $\Phi^+$ consists exactly of those roots which are nonnegative integral combinations of the $\alpha_i$. The Weyl group acts freely and transitively on the set of possible choices of $\Phi^+$. Therefore, since all choices are equivalent in this sense, let us fix a choice $\Phi^+$ of positive roots, and thus also a collection $\alpha_1,\ldots,\alpha_{n}$ of simple roots. The simple roots are pairwise non-acute: i.e., $\langle \alpha_i,\alpha_j^\vee \rangle \leq 0$ for $i \neq j \in [n]$.

The \emph{coroots} $\alpha^\vee$ for $\alpha \in \Phi$ themselves form a root system which we call the \emph{dual root system} of $\Phi$ and which we denote $\Phi^\vee$. We always consider $\Phi^\vee$ with its positive roots being $\alpha^\vee$ for $\alpha \in \Phi^+$; hence $\alpha_i^\vee$ for $i=1,\ldots,n$ are the \emph{simple coroots}.

 The \emph{simple reflections} $s_{\alpha_i}$ for $i\in[n]$ generate $W$. The \emph{length} of a Weyl group element~$w \in W$ is the minimum length of a word expressing $w$ as a product of simple reflections. It is known that the length of $w\in W$ coincides with the number of inversions of $w$, where an \emph{inversion} of $w$ is a positive root $\alpha \in \Phi^+$ with $w(\alpha)\notin \Phi^+$.

There are two important lattices associated to $\Phi$: the \emph{root lattice} $Q \coloneqq \mathrm{Span}_{\mathbb{Z}}(\Phi)$ and the \emph{weight lattice} $P \coloneqq \{v\in V\colon\langle v,\alpha^\vee \rangle \in \mathbb{Z}\}$. By assumption of crystallography, we have that $Q\subseteq P$. The elements of $P$ are called \emph{weights}. We use $\omega_1,\ldots,\omega_{n}$ to denote the set of \emph{fundamental weights}, which form a dual basis to $\alpha^\vee_1,\ldots,\alpha^\vee_{n}$, i.e., they are defined by $\langle \omega_i,\alpha_j^\vee\rangle = \delta_{ij}$ for $i,j \in [n]$, where $\delta_{ij}$ is the Kronecker delta. Observe that $Q = \mathrm{Span}_{\mathbb{Z}}(\{\alpha_1,\ldots,\alpha_{n}\})$ and $P=\mathrm{Span}_{\mathbb{Z}}(\{\omega_1,\ldots,\omega_{n}\})$. We use the following notation for the ``positive parts'' of these lattices: $Q_{\geq 0} \coloneqq\mathrm{Span}_{\mathbb{Z}_{\geq 0}}(\{\alpha_1,\ldots,\alpha_{n}\})$; $P_{\geq 0} \coloneqq\mathrm{Span}_{\mathbb{Z}_{\geq 0}}(\{\omega_1,\ldots,\omega_{n}\})$. We also use the following notation for the two associated cones:
$Q^{\mathbb{R}}_{\geq 0} \coloneqq\mathrm{Span}_{\mathbb{R}_{\geq 0}}(\{\alpha_1,\ldots,\alpha_{n}\});$ $P^{\mathbb{R}}_{\geq 0} \coloneqq\mathrm{Span}_{\mathbb{R}_{\geq 0}}(\{\omega_1,\ldots,\omega_{n}\})$. Note that~$P^{\mathbb{R}}_{\geq 0}$ and~$Q^{\mathbb{R}}_{\geq 0}$ are \emph{dual cones}, meaning that $P^{\mathbb{R}}_{\geq 0} = \{v\in V\colon \langle v, u \rangle \geq 0 \textrm{ for all $u \in Q^{\mathbb{R}}_{\geq 0}$}\}$. A weight $\lambda \in P$ is called \emph{dominant} if~$\langle \lambda,\alpha_i^\vee\rangle \geq 0$ for all $i \in [n]$. Hence $P_{\geq 0}$ is the set of dominant weights. An important dominant weight is the \emph{Weyl vector} $\rho \coloneqq \sum_{i=1}^{n}\omega_i$. As mentioned earlier, we also have that $\rho = \frac{1}{2}\sum_{\alpha\in\Phi^+}\alpha$. That these two descriptions of $\rho$ agree implies that $P^{\mathbb{R}}_{\geq 0}\subseteq Q^{\mathbb{R}}_{\geq 0}$.

Let $X \subseteq \Phi$. Then $\Phi \cap \mathrm{Span}_{\mathbb{R}}(X)$ is a root system in $\mathrm{Span}_{\mathbb{R}}(X)$, and we call this root system $\Phi \cap \mathrm{Span}_{\mathbb{R}}(X)$ a \emph{sub-root system} of $\Phi$. (Note that this terminology may be slightly nonstandard insofar as we do not consider every subset of $\Phi$ which forms a root system to be a sub-root system.) It is always the case that $\Phi^+ \cap \mathrm{Span}_{\mathbb{R}}(X)$ is a choice of positive roots for $\Phi \cap \mathrm{Span}_{\mathbb{R}}(X)$ and we always consider sub-root systems with this choice of positive roots. However, note that $\{\alpha_1,\ldots,\alpha_{n}\}\cap\mathrm{Span}_{\mathbb{R}}(X)$ need not be a collection of simple roots for $\Phi \cap \mathrm{Span}_{\mathbb{R}}(X)$. The case where this intersection~$\{\alpha_1,\ldots,\alpha_{n}\}\cap\mathrm{Span}_{\mathbb{R}}(X)$ does form a collection of simple roots is nonetheless an important special case of sub-root system which we call a \emph{parabolic sub-root system}: for~$I\subseteq [n]$ we use the notation~$\Phi_I \coloneqq \Phi \cap \mathrm{Span}_{\mathbb{R}}(\{\alpha_i\colon i \in I\})$. We also use $W_I$ to  denote the corresponding \emph{parabolic subgroup} of $W$, which is the subgroup generated by~$s_{\alpha}$ for~$\alpha \in \Phi_I$.

We will often need to consider projections of weights onto sub-root systems. For a subspace $U\subseteq V$ we use $\pi_U\colon V \to U$ to denote the orthogonal (with respect to~$\langle \cdot,\cdot\rangle$) projection of $V$ onto $U$. And for~$X \subseteq \Phi$ we use the shorthand $\pi_{X} \coloneqq \pi_{\mathrm{Span}_{\mathbb{R}}(X)}$. Note that for $\lambda \in P$ we always have that $\pi_{X}(\lambda)$ is a weight of $\Phi\cap\mathrm{Span}_{\mathbb{R}}(X)$, although $\pi_X(\lambda)$ need not be a weight of $\Phi$. Similarly, if $\lambda \in P_{\geq 0}$ is dominant then $\pi_{X}(\lambda)$ is a dominant weight of $\Phi\cap\mathrm{Span}_{\mathbb{R}}(X)$. For $I\subseteq [n]$ we use the notation $\pi_I \coloneqq \pi_{\{\alpha_i\colon i \in I\}}$. Observe that~$\pi_I(\sum_{i=1}^{n}c_i\omega_i)=\sum_{i\in I}c_i\omega'_i$, where $\{\omega'_i\colon i\in I\}$ is the set of fundamental weights of~$\Phi_I$.

It would be helpful to have a ``standard form'' for sub-root systems. As mentioned, sub-root systems need not be parabolic. Nevertheless, we can always act by a Weyl group element to make them parabolic. In fact, as the following proposition demonstrates, slightly more than this is true: we can also make any given vector which projects to zero in the sub-root system dominant at the same time.

\begin{prop} \label{prop:parabolic_subspace}
Let $X \subseteq \Phi^+$. Let $v \in V$ be such that $\pi_{X}(v)=0$. Then there exists some~$w \in W$ and $I \subseteq [n]$ such that $w\,\mathrm{Span}_{\mathbb{R}}(X)=\mathrm{Span}_{\mathbb{R}}\{\alpha_i\colon i \in I\}$ and $wv \in P_{\geq 0}^{\mathbb{R}}$.
\end{prop}
\begin{proof}
This result is a slight extension of a result of Bourbaki~\cite[Chapter IV, \S1.7, Proposition 24]{bourbaki2002lie}, which is equivalent to the present proposition but without the requirement~$wv \in P_{\geq 0}^{\mathbb{R}}$. If $v=0$, then~$wv \in P_{\geq 0}^{\mathbb{R}}$ is automatically satisfied for any $w\in W$, so let us assume that $v\neq 0$.

Following Bourbaki, let us explain one way to choose a set of positive roots. Namely, suppose that $\preccurlyeq$ is a total order on $V$ compatible with the real vector space structure in the sense that if $u \preccurlyeq v$ then $u+u' \preccurlyeq v+u'$ and $\kappa u \preccurlyeq \kappa v$ for all~$u,u',v \in V$ and~$\kappa \in \mathbb{R}_{\geq 0}$. Then the set $\{\alpha\in \Phi\colon 0 \preccurlyeq\alpha\}$ will be a valid choice of positive roots for $\Phi$.

We proceed to define an appropriate total order $\preccurlyeq$. Let $\beta_1,\ldots,\beta_{\ell}$ be a choice of simple roots for $\Phi\cap\mathrm{Span}_{\mathbb{R}}(X)$. Then let $v_1,\ldots,v_{n}$ be an ordered basis of $V$ such that: $v_1=v$; $v_{(n-\ell)+i}=\beta_i$ for all $i=1,\ldots,\ell$; $v_1$ is orthogonal to all of $v_2,\ldots,v_n$. (Such a basis exists because $\pi_{X}(v)=0$ implies $v$ is orthogonal to all of $\beta_1,\ldots,\beta_{\ell}$.) Then let~$\preccurlyeq$ be the lexicographic order on~$V$ with respect to the ordered basis $v_1,\ldots,v_{n}$; that is to say, $\sum_{i=1}^{n}a_iv_i \preccurlyeq \sum_{i=1}^{n}a'_iv_i$ means that either $\sum_{i=1}^{n}a_iv_i= \sum_{i=1}^{n}a'_iv_i$ or there is some~$i\in[n]$ such that $a_j = a'_j$ for all $1\leq j < i$ and $a_{i}<a'_{i}$.

It is clear that $\beta_1,\ldots,\beta_{\ell}$ are minimal (with respect to $\preccurlyeq$) in $\{\alpha\in \Phi\colon 0 \preccurlyeq\alpha\}$, which implies that they are simple roots of $\Phi$ for the choice $\{\alpha\in \Phi\colon 0 \preccurlyeq\alpha\}$ of positive roots.

Moreover, for any $u = \sum_{i=1}^{n}a_iv_i \in V$ we have that $\langle v,u\rangle=a_1$ because~$v$ is orthogonal to $v_2,\ldots,v_{n}$. Hence for any $u \in V$ with $0 \preccurlyeq u$ we have $\langle v,u\rangle\geq 0$. This means in particular that $\langle v,\alpha^\vee\rangle \geq 0$ for any $\alpha \in \Phi$ with $0 \preccurlyeq\alpha$.

Since all choices of positive roots are equivalent up to the action of the Weyl group, there exists $w\in W$ such that $w\{\alpha\in \Phi\colon 0 \preccurlyeq\alpha\} = \Phi^+$. This $w$ transports $\{\beta_1,\ldots,\beta_{\ell}\}$ to a subset of simple roots, so we get $w\,\mathrm{Span}_{\mathbb{R}}(X)=\mathrm{Span}_{\mathbb{R}}\{\alpha_i\colon i \in I\}$. That $wv \in P_{\geq 0}^{\mathbb{R}}$ follows from the previous paragraph and the fact that $w$ is an orthogonal transformation.
\end{proof}

Now let us return to our discussion of ($W$-)permutohedra. We can define the \emph{permutohedron} $\Pi(v)$ for any $v \in V$ to be $\Pi(v) \coloneqq \mathrm{ConvexHull}\, W(v)$. And for a weight $\lambda \in P$ we also define the \emph{discrete permutohedron} to be  $\Pi^Q(\lambda) \coloneqq \Pi(\lambda)\cap(Q+\lambda)$. (This discrete permutohedron only really makes sense for weights $\lambda \in P$ and not arbitrary vectors $v\in V$.) A simple but important consequence of the description of permutohedra containment given in Proposition~\ref{prop:perm_containment} below, which we will often use, is that $\Pi(u+v)=\Pi(u)+\Pi(v)$ for vectors $u,v \in P^{\mathbb{R}}_{\geq 0}$. 

Certain very special permutohedra are zonotopes. As mentioned, the \emph{regular permutohedron} $\Pi(\rho)$ is a zonotope: $\Pi(\rho)= \sum_{\alpha\in\Phi^+}[-\alpha/2,\alpha/2]$ (this can be seen, for instance, by taking the Newton polytope of both sides of the \emph{Weyl denominator formula}~\cite[\S24.3]{humphreys1972lie}). Moreover, for any $k \in \mathbb{Z}_{\geq0}$ we have that $\Pi(k\rho)$ is also a zonotope since $\Pi(k\rho)=k\Pi(\rho)$. In fact, we can obtain a slightly more general family of permutohedra which are zonotopes by scaling each Weyl group orbit of roots separately. We use the notation $\mathbf{k} \in \mathbb{Z}[\Phi]^W$ to mean that~$\mathbf{k}$ is a function~$\mathbf{k}\colon \Phi\to\mathbb{Z}$ which is invariant under the action of the Weyl group. For~$\mathbf{a},\mathbf{b}\in\mathbb{Z}[\Phi]^W$ and $a,b\in\mathbb{Z}$ we ascribe the obvious meanings to $a\mathbf{a}+b\mathbf{b}$, $\mathbf{a}=a$, and~$\mathbf{a}\geq\mathbf{b}$. We use~$\mathbb{N}[\Phi]^W$ to denote the set of $\mathbf{k}\in\mathbb{Z}[\Phi]^W$ with $\mathbf{k}\geq 0$. For any~$\mathbf{k} \in \mathbb{N}[\Phi]^W$ we define~$\rho_{\mathbf{k}} \coloneqq \sum_{i=1}^{n} \mathbf{k}(\alpha_i)\omega_i$ (so that~$\rho=\rho_1$ and $k\rho=\rho_k$). Then for any~$\mathbf{k} \in \mathbb{N}[\Phi]^W$ we have that~$\Pi(\rho_{\mathbf{k}}) = \sum_{\alpha \in \Phi^+} \mathbf{k}(\alpha)[-\alpha/2,\alpha/2]$ (this is an easy exercise given that~ $\Pi(\rho)= \sum_{\alpha\in\Phi^+}[-\alpha/2,\alpha/2]$).

We want to understand containment of permutohedra. As mentioned earlier, for a weight $\lambda \in P$ we use $\lambda_{\mathrm{dom}}$ to denote the unique dominant element of $W(\lambda)$. For any $\lambda,\mu \in P$, some immediate consequences of the $W$-invariance of permutohedra are: $\Pi(\lambda)=\Pi(\lambda_{\mathrm{dom}})$; $\Pi(\mu)\subseteq\Pi(\lambda)$ if and only if $\mu \in \Pi(\lambda)$; and $\mu\in \Pi(\lambda)$ if and only if~$\mu_{\mathrm{dom}}\in\Pi(\lambda)$. Thus to understand containment of permutohedra we can restrict to dominant weights. Recall that \emph{root order} is the partial order $\leq$ on $P$ for which we have~$\mu \leq \lambda$ if and only if $\lambda - \mu \in Q_{\geq 0}$. The following proposition says that for dominant weights, containment of discrete permutohedra is equivalent to root order; it also says that we can describe containment of (real) permutohedra in an exactly analogous way.

\begin{prop}[{See~\cite[Theorem 1.9]{stembridge1998partial} or~\cite[\oldpaperpermcontainment]{galashin2017rootfiring1}}] \label{prop:perm_containment}
Let $u,v \in P_{\geq 0}^{\mathbb{R}}$. Then $u \in \Pi(v)$ if and only if~$v-u \in Q_{\geq 0}^{\mathbb{R}}$. Consequently, for $\mu, \lambda \in P_{\geq 0}$ we have that~$\mu \in \Pi^Q(\lambda)$ if and only if~$\mu \leq \lambda$.
\end{prop}

In light of Proposition~\ref{prop:perm_containment}, let us review some basic facts about root order which appear in the seminal paper of Stembridge~\cite{stembridge1998partial} (but may have been known in some form earlier). First of all, we have that dominant weights are always maximal in root order in their Weyl group orbits.

\begin{prop}[{\cite[Lemma 1.7]{stembridge1998partial}}] \label{prop:dominant_maximal}
For any $\lambda \in P$ we have that $\lambda \leq \lambda_{\mathrm{dom}}$.
\end{prop}

Now let us consider root order restricted to the set of dominant weights. Root order on all of $P$ is trivially a disjoint union of lattices\footnote{Here we mean the poset-theoretic concept of lattice: i.e. a poset with joins and meets.}: it is isomorphic to $f$ copies of $\mathbb{Z}^n$ where $f$ is the index of $Q$ in $P$ (this index $f$ is called the \emph{index of connection} of the root system $\Phi$). Stembridge proved, what is much less trivial, that the root order on~$P_{\geq 0}$ is also a disjoint union of $f$ lattices. Let us explain how he did this. For~$\lambda = \sum_{i=1}^{n} a_i\alpha_i, \mu = \sum_{i=1}^{n} a'_i\alpha_i \in P$ with $\lambda - \mu \in Q$, we define their \emph{meet} to be
\[\lambda \wedge \mu \coloneqq \sum_{i=1}^{n} \mathrm{min}(a_i,a'_i) \alpha_i.\]
This is obviously the meet of $\lambda$ and $\mu$ in $P$ with respect to the partial order $\leq$. Stembridge proved the following about this meet operation:

\begin{prop}[{\cite[Lemma 1.2]{stembridge1998partial}}] \label{prop:meet_weights}
Let $\lambda, \mu \in P$ with $\lambda - \mu \in Q$. Let $i \in [n]$ and suppose that $\langle \lambda,\alpha_i^\vee\rangle \geq 0$ and $\langle \mu,\alpha_i^\vee\rangle \geq 0$. Then $\langle \lambda \wedge \mu,\alpha_i^\vee \rangle \geq 0$. Hence, in particular, if~$\lambda, \mu \in P_{\geq 0}$ then $\lambda \wedge \mu \in P_{\geq 0}$ as well.
\end{prop}

Strictly speaking, Proposition~\ref{prop:meet_weights} only implies that $(P_{\geq 0}, \geq)$ is a disjoint union of meet-semilattices; a little more is needed to show that it is a disjoint union of lattices. At any rate, Proposition~\ref{prop:meet_weights} compels us to ask what the minimal elements of $(P_{\geq 0}, \geq)$ are; there will again be $f$ of these, one for every coset of $Q$ in $P$ (because $P_{\geq 0}^{\mathbb{R}}\subseteq Q_{\geq 0}^{\mathbb{R}}$, every element of $(P_{\geq 0}, \geq)$ has to be greater than or equal to a minimal element).

Recall that a dominant, nonzero weight $\lambda \in P_{\geq 0}\setminus \{0\}$ is called \emph{minuscule} if we have that~$\langle \lambda,\alpha^\vee \rangle \in \{-1,0,1\}$ for all $\alpha \in \Phi$.

\begin{prop}[{\cite[Lemma 1.12]{stembridge1998partial}}]  \label{prop:minuscule}
The minimal elements of $(P_{\geq 0}, \geq)$ are precisely the minuscule weights of $\Phi$ together with zero.
\end{prop}

Hence there are $f-1$ minuscule weights. Observe that Proposition~\ref{prop:minuscule} together with Proposition~\ref{prop:perm_containment} gives another characterization of minuscule weights:

\begin{prop} \label{prop:minuscule_alternate}
For~$\lambda \in P_{\geq 0}$ we have that~$\Pi^Q(\lambda)=W(\lambda)$ if and only if $\lambda$ is zero or minuscule. 
\end{prop}

Another simple property of minuscule weights that we will use repeatedly is: if $\lambda \in P$ is zero or a minuscule weight of $\Phi$, then $\pi_{X}(\lambda)$ is a zero or a minuscule weight of~$\Phi\cap\mathrm{Span}_{\mathbb{R}}(X)$ for any $X \subseteq \Phi$.

If $\mu\in P$ is a weight with $\langle \mu, \alpha^\vee \rangle \in \{-1,0,1\}$ for all $\alpha \in \Phi^+$, then $\mu_{\mathrm{dom}}$ is either zero or minuscule, and hence in particular by Proposition~\ref{prop:minuscule} we have that $\mu_{\mathrm{dom}}$ is the minimal dominant weight greater than or equal to $\mu$ in root order. Let us now show that this conclusion (that  $\mu_{\mathrm{dom}}$ is the minimal dominant weight greater than or equal to $\mu$ ) follows from the weaker assumption that  $\langle \mu, \alpha^\vee\rangle\geq -1$ for all $\alpha \in \Phi^+$.

\begin{prop} \label{prop:unique_dominant_greater}
Let $\mu \in P$ be a weight with $\langle \mu,\alpha^\vee \rangle \geq -1$ for all~$\alpha \in \Phi^+$. Then if~$\lambda \in P_{\geq 0}$ is a dominant weight with $\mu \leq \lambda$, it must be that $\mu_{\mathrm{dom}}\leq \lambda$.
\end{prop}

\begin{proof}
If $\mu$ is dominant, the conclusion is clear. So suppose $\mu$ is not dominant. Hence, there is a simple root~$\alpha_i$ with $\langle \mu,\alpha_i^\vee \rangle < 0$. By supposition, this means $\langle \mu,\alpha_i^\vee \rangle = -1$. We claim that then $\langle s_{\alpha_i}(\mu),\alpha^\vee\rangle \geq -1$ for all positive $\alpha \in \Phi^+$, i.e., that $s_{\alpha_i}(\mu)$ satisfies the hypothesis of the proposition. Indeed, for a positive root $\alpha \in \Phi^+$ we have that $\langle s_{\alpha_i}(\mu),\alpha^\vee\rangle = \langle \mu,s_{\alpha_i}(\alpha^\vee)\rangle=\langle \mu,s_{\alpha_i}(\alpha)^\vee\rangle$. Then note that, since it has length one, the simple reflection $s_{\alpha_i}$ permutes the positive roots other than $\alpha_i$, and sends $\alpha_i$ to~$-\alpha_i$ (here we are using the fact that the length of a Weyl group element is equal to its number of inversions). So if $\alpha\in\Phi^+$ and $\alpha\neq \alpha_i$ we have by supposition that $\langle s_{\alpha_i}(\mu),\alpha^\vee\rangle \geq -1$; on the other hand, if  $\alpha=\alpha_i$ then $\langle s_{\alpha_i}(\mu),\alpha^\vee\rangle = -\langle \mu,\alpha^\vee \rangle=1$. So indeed $s_{\alpha_i}(\mu)=\mu+\alpha_i$ satisfies the  hypothesis of the proposition. Thus by induction on the minimum length of a $w\in W$ with $w^{-1}(\mu)=\mu_{\mathrm{dom}}$ we may assume that $s_{\alpha_i}(\mu)=\mu+\alpha_i$ satisfies the conclusion of the proposition. That is, if~$\lambda \in P_{\geq 0}$ is a dominant weight with $\mu+\alpha_i \leq \lambda$, it must be that $\mu_{\mathrm{dom}}=(s_{\alpha_i}(\mu))_{\mathrm{dom}} \leq \lambda$. But note that if $\lambda \in P_{\geq 0}$ is a dominant weight with $\mu \leq \lambda$ then necessarily $\mu+\alpha_i \leq \lambda$: otherwise we would have $\langle \lambda,\alpha_i^\vee\rangle \leq \langle \mu,\alpha_i^\vee \rangle< 0$ by the pairwise non-acuteness of the simple roots. Hence we conclude that for any dominant weight~$\lambda \in P_{\geq 0}$ with $\mu\leq \lambda$ we have $ \mu_{\mathrm{dom}}$, as claimed.
\end{proof}

If there exist $\varnothing \subsetneq \Phi',\Phi'' \subsetneq \Phi$ for which $\Phi = \Phi'\sqcup\Phi''$ and such that $\langle\alpha,\beta^\vee\rangle = 0$ for all~$\alpha \in \Phi',\beta\in\Phi''$ we say that $\Phi$ is \emph{reducible} and write $\Phi= \Phi'\oplus\Phi''$; otherwise we say that $\Phi$ is \emph{irreducible}. (By fiat let us also declare that the empty set, although it is a root system, is not irreducible.) Any root system is the orthogonal direct sum of its irreducible components and so all constructions related to root systems decompose in a simple way into irreducible factors. So from now on we will assume that $\Phi$ is irreducible. The irreducible root systems have been classified into the \emph{Cartan-Killing types} (e.g., Type~$A_n$, Type~$B_n$, et cetera), but since we will not need to use the classification until the appendix of this paper, we will not go over that classification now.

\subsection{Formula for lattice points and an integrality property of permutohedra}

In this subsection we establish the formula for the number of lattice points in a Minkowski sum of a permutohedron and a scaling regular permutohedron (Theorem~\ref{thm:perm_intro} in Section~\ref{sec:intro}). To do this we need to prove the subtle integrality property of permutohedra we mentioned earlier (Lemma~\ref{lem:slice_intro} in Section~\ref{sec:intro}). Recall that this integrality property asserts that for certain lattice polytopes $\mathcal{P}$ and lattice zonotopes $\mathcal{Z} = \sum_{i=1}^{m}[0,v_i]$ in $\mathbb{R}^n$ we have $\mathrm{quot}_{X}(\mathcal{P})\cap\mathrm{quot}_{X}(\mathbb{Z}^n) = \mathrm{quot}_X(\mathcal{P}\cap\mathbb{Z}^n)$ for all $X\subseteq \{v_1,\ldots,v_m\}$. Before we prove this integrality property in the case relevant to us, let us show how it can fail in the more general situation of arbitrary lattice polytopes and lattice zonotopes.

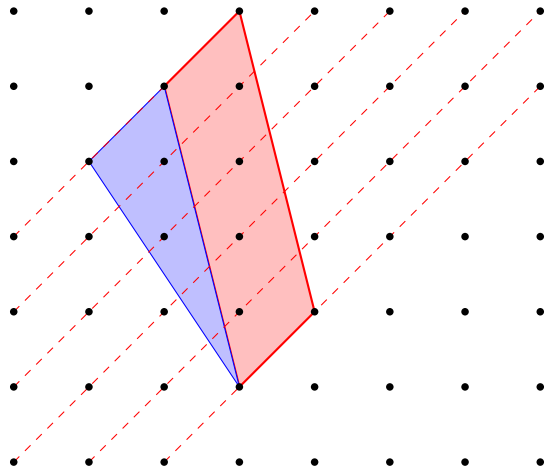
\begin{figure}
\begin{tikzpicture}
\draw[thick, color=red, fill=red!25] (2,6) -- (3,7) -- (4,3) -- (3,2) -- cycle;
\draw[color=blue, fill=blue!25] (2,6) -- (1,5) -- (3,2) -- cycle;
\draw[thin,color=red,dashed] (0,4) -- (3,7);
\draw[thin,color=red,dashed] (0,3) -- (4,7);
\draw[thin,color=red,dashed] (0,2) -- (5,7);
\draw[thin,color=red,dashed] (0,1) -- (6,7);
\draw[thin,color=red,dashed] (1,1) -- (7,7);
\draw[thin,color=red,dashed] (2,1) -- (7,6);
\foreach \i in {0,...,7}
{
        \foreach \j in {1,...,7}
        {
            \fill (\i,\j) circle (0.05);
        }
}
\end{tikzpicture}
\caption{Example~\ref{ex:thin_triangle} of a ``thin'' triangle plus line segment which does not satisfy $\mathrm{quot}_X(\mathcal{P})\cap\mathrm{quot}_X(\mathbb{Z}^n)=\mathrm{quot}_X(\mathcal{P}\cap\mathbb{Z}^n)$.}\label{fig:thin_triangle} 
\end{figure}

\begin{example} \label{ex:thin_triangle}
Let $\mathcal{P}$ be the lattice triangle in $\mathbb{R}^2$ with vertices $(0,3)$, $(1,4)$, and $(2,0)$. Let $v \coloneqq (1,1) \in \mathbb{Z}^2$ and set $\mathcal{Z} \coloneqq [0,v]$, a zonotope (in fact, a line segment). Figure~\ref{fig:thin_triangle} depicts $\mathcal{P}$ as the region shaded in blue, and $\mathcal{P}+\mathcal{Z}$ as the region shaded in blue together with the region shaded in red. The dashed red lines are all the affine subspaces of the form $u+\mathrm{Span}_{\mathbb{R}}(v)$ for which $u+\mathrm{Span}_{\mathbb{R}}(v) \cap \mathcal{P} \neq \varnothing$. There are six such subspaces. However, only four of these subspaces satisfy $u+\mathrm{Span}_{\mathbb{R}}(v) \cap (\mathcal{P}\cap\mathbb{Z}^2) \neq \varnothing$. In other words,
we have~$\#\mathrm{quot}_X(\mathcal{P}\cap\mathbb{Z}^2) = 4 < 6 = \#\mathrm{quot}_X(\mathcal{P})\cap\mathrm{quot}_X(\mathbb{Z}^2)$ when $X \coloneqq\{v\}$. We can verify that $\#(\mathcal{P}+k\mathcal{Z})\cap\mathbb{Z}^2=6k+5$, in agreement with Theorem~\ref{thm:poly_plus_zono}.
\end{example}

The reason that Example~\ref{ex:thin_triangle} fails to satisfy~$\mathrm{quot}_X(\mathcal{P})\cap\mathrm{quot}_X(\mathbb{Z}^n)=\mathrm{quot}_X(\mathcal{P}\cap\mathbb{Z}^n)$ is that the polytope $\mathcal{P}$ is too ``thin'' in the direction of $X$. So in order to show that permutohedra do satisfy this integrality property, we need, roughly speaking, to show that they cannot be too ``thin'' in any direction spanned by roots. Intuitively, the $W$-invariance of permutohedra prevents them from being ``thin'' in any given root direction (because otherwise they would be ``thin'' in \emph{every} root direction). But this is a just a rough intuition for why permutohedra might satisfy the requisite integrality property. The actual argument, which we give now, is rather involved and eventually requires us to invoke the classification of root systems. 

First let us restate the integrality property of permutohedra in a slightly different language, which uses ``slices'' rather than quotients:

\begin{lemma} \label{lem:slice}
Let $\lambda\in P_{\geq 0}$ be a dominant weight, let $\mu \in Q+\lambda$, and let $X\subseteq \Phi^+$. Suppose that $\Pi(\lambda)\cap (\mu+\mathrm{Span}_{\mathbb{R}}(X)) \neq \varnothing$. Then $\Pi^Q(\lambda)\cap (\mu+\mathrm{Span}_{\mathbb{R}}(X))\neq \varnothing$.
\end{lemma}

Recall that the \emph{root polytope} $\mathcal{P}_{\Phi}$ of the root system $\Phi$ is simply the convex hull of the roots: $\mathcal{P}_{\Phi} \coloneqq \mathrm{ConvexHull}(\Phi)$.\footnote{Sometimes, as in~\cite{meszaros2011root1, meszaros2011root2}, the term \emph{root polytope} is used to refer to the convex hull of the positive roots together with the origin. We will always use it to mean the convex hull of all the roots, following the terminology in~\cite{cellini2015root}.} It turns out that the integrality property of slices of permutohedra follows from the following lemma concerning dilations of projections of the root polytope for the \emph{dual} root system.

\begin{lemma} \label{lem:root_projection}
Let $\{0\}\neq U\subseteq V$ be a nonzero subspace of $V$ spanned by a subset of~$\Phi^\vee$. Set $\Phi^\vee_U\coloneqq\Phi^\vee\cap U$, a sub-root system of $\Phi^\vee$. Then there exists some $1\leq \kappa < 2$ such that $\pi_U(\mathcal{P}_{\Phi^\vee}) \subseteq \kappa \cdot \mathcal{P}_{\Phi^\vee_U}$.
\end{lemma}

Since Lemma~\ref{lem:root_projection} can be hard to understand at first sight, let's give an example.

\begin{figure}
\begin{center}
\includegraphics[width=5cm]{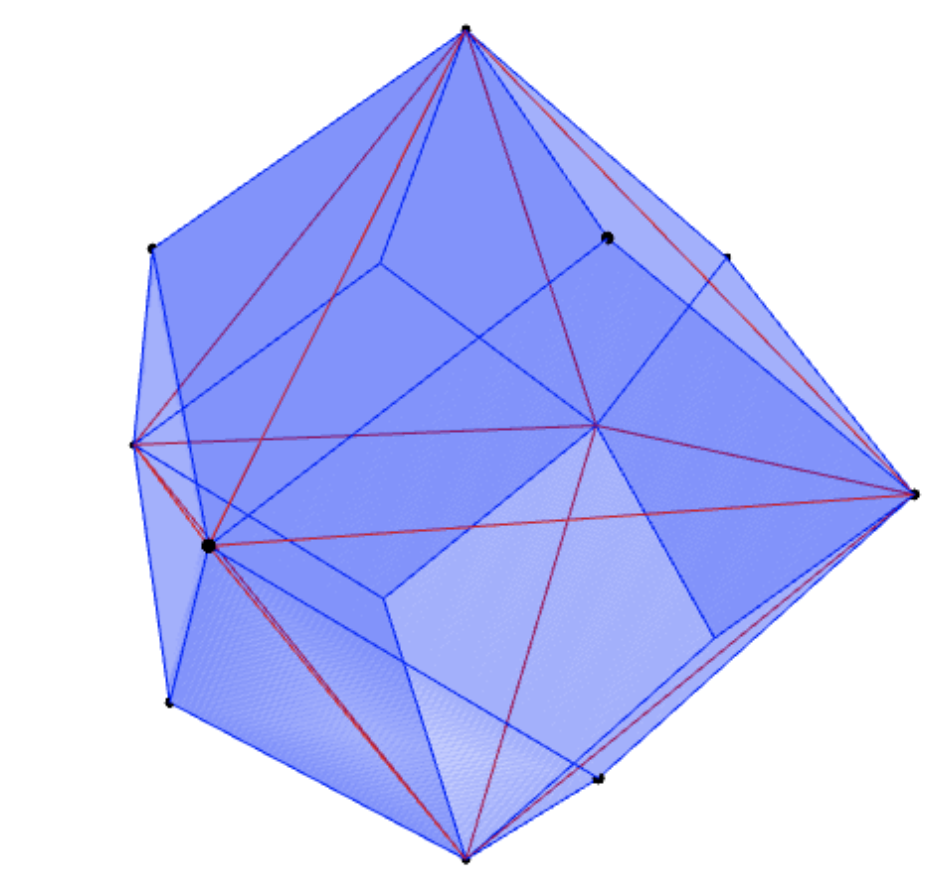}
\end{center}
\caption{The projection of the root polytope of $D_4$ onto the maximal parabolic subspace corresponding to the trivalent node of the Dynkin diagram.} \label{fig:d4}
\end{figure}

\begin{example} \label{ex:d4}
Let $\Phi$ be the root system of Type~$D_4$. Thus $V=\mathbb{R}^4$ with the standard basis $e_1, e_2, e_3, e_4$ and inner product $\langle e_i,e_j\rangle = \delta_{i,j}$, and 
\[\Phi = \{\pm(e_i-e_j), \pm(e_i+e_j)\colon 1 \leq i < j \leq 4\}.\] 
Note that $\Phi^\vee =\Phi$ (i.e., $D_4$ is ``simply laced'') so we will ignore the distinction between $\Phi$ and $\Phi^\vee$ in this example. We choose simple roots $\alpha_1 = e_1-e_2$, $\alpha_2 = e_2-e_3$, $\alpha_3 = e_3-e_4$ and $\alpha_4 = e_3+e_4$. Let $U \coloneqq \mathrm{Span}_{\mathbb{R}}\{\alpha_1,\alpha_3,\alpha_4\}$. Note that $U$ is the subspace of $V$ orthogonal to $\omega_2=e_1+e_2$. Thus for instance we can compute
\[ \pi_U(\alpha_2) = \alpha_2 - \frac{\langle \alpha_2,\omega_2 \rangle}{\langle \omega_2, \omega_2 \rangle} \omega_2 = -\frac{1}{2}e_1+\frac{1}{2}e_2-e_3= -\frac{1}{2}\alpha_1 - \frac{1}{2}\alpha_3 - \frac{1}{2}\alpha_4.\]
In fact, we have that $\pi_U(\Phi)$ consists of $14$ points:
\[\pi_U(\Phi)= \{\pm\frac{1}{2}\alpha_1 \pm\frac{1}{2}\alpha_3 \pm \frac{1}{2}\alpha_4,\pm\alpha_1,\pm\alpha_3,\pm\alpha_4\}.\]
On the other hand, it is also easy to see that $\Phi_U \coloneqq \Phi \cap U$ consists of $6$ points:
\[ \Phi_U =\{\pm\alpha_1,\pm\alpha_3,\pm\alpha_4\}.\]
This means that $\pi_U(\mathcal{P}_{\Phi})$ is a rhombic dodecahedron, and~$\mathcal{P}_{\Phi_U}$ is an octahedron inscribed inside this rhombic dodecahedron. Figure~\ref{fig:d4} depicts $\pi_U(\mathcal{P}_{\Phi})$ (in blue) and~$\mathcal{P}_{\Phi_U}$ (in red wireframe). In this case, it turns out that the minimum~$\kappa \geq 1$ for which $\pi_U(\mathcal{P}_{\Phi}) \subseteq \kappa\cdot \mathcal{P}_{\Phi_U}$ is $\kappa = \frac{3}{2}$. Since $\kappa < 2$, this example agrees with Lemma~\ref{lem:root_projection}.
\end{example}

We now show that Lemma~\ref{lem:root_projection} implies Lemma~\ref{lem:slice}.

\begin{proof}[Proof that Lemma~\ref{lem:root_projection} implies Lemma~\ref{lem:slice}] Let $\lambda$, $\mu$, and~$X$ be as in the statement of Lemma~\ref{lem:slice}. Define $\mu^0_X \in V$ to be the unique vector in the affine subspace $\mu+\mathrm{Span}_{\mathbb{R}}(X)$ for which $\pi_{X}(\mu^0_X) = 0$. We claim that~$\mu^0_X \in \Pi(\lambda)$. Indeed, $\mu^0_X$ is the ``inner-most'' vector in $\mu+\mathrm{Span}_{\mathbb{R}}(X)$, so if any vector of $\mu+\mathrm{Span}_{\mathbb{R}}(X)$ lies in~$\Pi(\lambda)$ then $\mu^0_X$ must as well. To explain this more formally, let $W'\subseteq W$ denote the Weyl group of the sub-root system~$\Phi \cap \mathrm{Span}_{\mathbb{R}}(X)$. There is of course the natural inclusion $W' \subseteq W$. For any $u \in \mathrm{Span}_{\mathbb{R}}(X)$ we have~$0 \in  \mathrm{ConvexHull}\, W'(u)$ (for instance, by Proposition~\ref{prop:perm_containment}). But since $\mathrm{ConvexHull}\, W'(u) = \mu^0_X + \mathrm{ConvexHull}\, W'(\pi_{X}(u))$ for any $u \in \mu+\mathrm{Span}_{\mathbb{R}}(X)$, we conclude~$\mu^0_X \in \mathrm{ConvexHull}\, W'(u)$ for any $u \in \mu+\mathrm{Span}_{\mathbb{R}}(X)$. Hence in particular we have that~$\mu^0_X \in \Pi(u)$ for any $u \in \mu+\mathrm{Span}_{\mathbb{R}}(X)$. By supposition there exists some~$u\in\Pi(\lambda)\cap (\mu+\mathrm{Span}_{\mathbb{R}}(X))$, so $\mu^0_X \in \Pi(u) \subseteq \Pi(\lambda)$ as claimed.

Because of the $W$-invariance of $\Pi(\lambda)$, if the statement of Lemma~\ref{lem:slice} is true for $X$ and $\mu$, then it is true for $wX$ and $w\mu$ as well. Hence, by Proposition~\ref{prop:parabolic_subspace} we may assume that $\mathrm{Span}_{\mathbb{R}}(X) = \mathrm{Span}_{\mathbb{R}}\{\alpha_i\colon i \in I\}$ for some $I \subseteq [n]$ and that $\mu^0_X \in P_{\geq 0}^{\mathbb{R}}$. Note importantly that $\mu^0_X$ need not be a weight of $\Phi$: in general it is just a vector in $V$, and even if $\mu^0_X$ is a weight of $\Phi$ it need not belong to the coset $Q+\lambda$.

Having made some assumptions about $X$ and $\mu^0_X$, let us now show that we can also make some assumptions about $\mu$. First of all, note that $\sum_{\alpha\in \Phi_I} \alpha$ has positive inner product with every $\alpha_i^\vee$ for $i \in I$ (because it is equal to twice the Weyl vector of $\Phi_I$). Thus by repeatedly adding the vector $\sum_{\alpha\in \Phi_I} \alpha$ to $\mu$, we can assume that~$\langle \mu, \alpha_i^\vee \rangle \geq 0$ for all $i \in I$. Furthermore, we claim that $\lambda \wedge \mu \in \mu  + \mathrm{Span}_{\mathbb{R}}(X)$. Indeed, the $\alpha_i$ coordinates (in the basis of simple roots) of $\mu$ and of $\mu^0_X$ are the same for any $i \notin I$.  But since $\mu^0_X \in \Pi(\lambda)$, we have by Proposition~\ref{prop:perm_containment} that the $\alpha_i$ coordinate of $\mu^0_X$ is less than or equal to that of $\lambda$ for all $i \in [n]$. Hence the $\alpha_i$ coordinate of $\mu$ is less than or equal to that of $\lambda$ for any $i \notin I$, which implies that $\mu - (\lambda \wedge \mu)$ belongs to $\mathrm{Span}_{\mathbb{R}}\{\alpha_i\colon i \in I\}$. Because we have assumed that $\langle \mu, \alpha_i^\vee \rangle \geq 0$ for all $i \in I$, by Proposition~\ref{prop:meet_weights} we conclude that $\langle \lambda \wedge \mu, \alpha_i^\vee \rangle \geq 0$ for all~$i \in I$ as well. In other words, we know that $\pi_{I}(\lambda \wedge\mu)$ is dominant in $\Phi_I$. But by Proposition~\ref{prop:minuscule} the minimal, in root order, dominant weights are either zero or minuscule. Hence there exists some weight~$\nu \in \mu+\mathrm{Span}_{\mathbb{R}}(X)$ for which $\pi_{I}(\nu)$ is either zero or a minuscule weight of~$\Phi_I$, and such that $\nu \leq \lambda \wedge \mu$. Of course we also have $\lambda \wedge \mu \leq \lambda$, so in fact $\nu \leq \lambda$. By replacing $\mu$ with $\nu$, we can thus assume that $\pi_{I}(\mu)$ is zero or a minuscule weight of~$\Phi_I$, and that $\mu \leq \lambda$.

To summarize the above, without loss of generality {\bf we from now on in the proof of this lemma assume the following list of additional conditions}:
\begin{enumerate}[(a)]
\item \label{cond:integrality_1} $\mathrm{Span}_{\mathbb{R}}(X) = \mathrm{Span}_{\mathbb{R}}\{\alpha_i\colon i \in I\}$ for some $I\subseteq [n]$;
\item \label{cond:integrality_2} the unique vector $\mu^0_X \in \mu+\mathrm{Span}_{\mathbb{R}}(X)$ with $\pi_{X}(\mu^0_X)=0$ satisfies $\mu^0_X \in P_{\geq 0}^{\mathbb{R}}$;
\item \label{cond:integrality_3} $\pi_{I}(\mu)$ is zero or a minuscule weight of $\Phi_I$;
\item \label{cond:integrality_4} $\mu \leq \lambda$.
\end{enumerate}

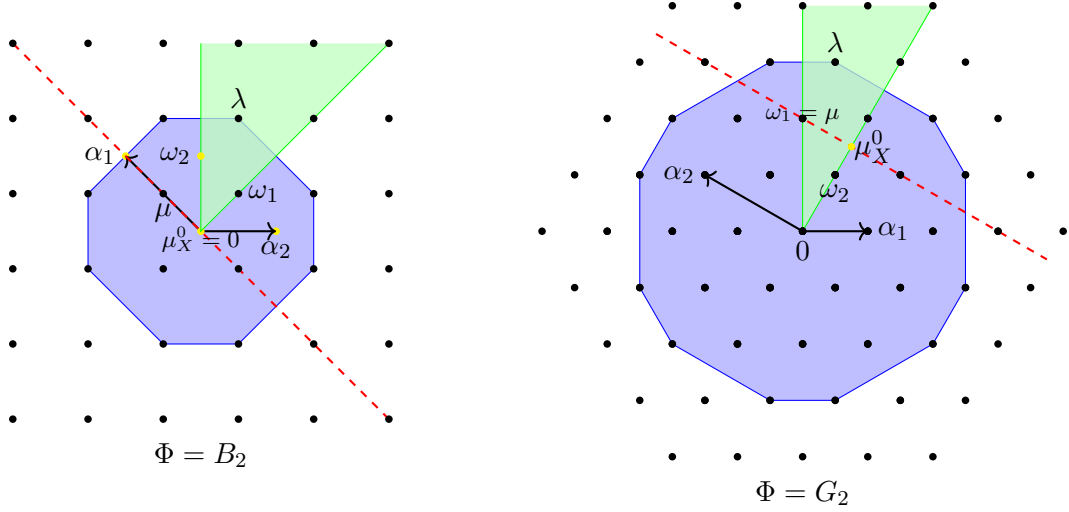
\begin{figure} 
\begin{tikzpicture}
\node (B2) at (0,0) {\begin{tikzpicture}
\draw[blue,fill=blue!25] (1,3)--(2,4)--(3,4)--(4,3)-- (4,2)--(3,1) -- (2,1) -- (1,2) -- cycle;
\fill[yellow] (1.5,3.5) circle (0.05);
\draw[->,thick] (2.5,2.5) -- (1.5,3.5);
\node[anchor=east] at (1.5,3.5) {$\alpha_1$};
\fill[yellow] (3.5,2.5) circle (0.05);
\draw[->,thick] (2.5,2.5) -- (3.5,2.5);
\node[anchor=north] at (3.5,2.5) {$\alpha_2$};
\draw[dashed,red,thick] (5,0) -- (0,5);
\fill[yellow] (2.5,2.5) circle (0.05);
\node[anchor=north,yshift=2mm] at (2.5,2.5) {\footnotesize $\mu^0_X = 0$};
\node[anchor=north] at (2,3) {$\mu$};
\draw[green,fill=green!25,fill opacity=0.75] (5,5) -- (2.5,2.5) -- (2.5,5);
\fill[yellow] (2.5,3.5) circle (0.05);
\node[anchor=east] at (2.5,3.5) {$\omega_2$};
\node[anchor=west] at (3,3) {$\omega_1$};
\node[anchor=south] at (3,4) {$\lambda$};
\foreach \i in {0,...,5}
{
        \foreach \j in {0,...,5}
        {
            \fill (\i,\j) circle (0.05);
        }
}
\end{tikzpicture}};
\node (G2) at (8,0) {\begin{tikzpicture}
\coordinate (L1) at (90:1.5cm);
\coordinate (L2) at (210:1.5cm);
\coordinate (L3) at (-30:1.5cm);
\draw[blue,fill=blue!25] (barycentric cs:L1=4,L2=-1,L3=0) -- (barycentric cs:L1=3,L2=-2,L3=2) -- (barycentric cs:L1=2,L2=-2,L3=3) -- (barycentric cs:L1=0,L2=-1,L3=4) -- (barycentric cs:L1=-1,L2=0,L3=4) --  (barycentric cs:L1=-2,L2=2,L3=3) -- (barycentric cs:L1=-2,L2=3,L3=2) -- (barycentric cs:L1=-1,L2=4,L3=0) -- (barycentric cs:L1=0,L2=4,L3=-1) -- (barycentric cs:L1=2,L2=3,L3=-2) -- (barycentric cs:L1=3,L2=2,L3=-2) -- (barycentric cs:L1=4,L2=0,L3=-1) -- cycle;
\draw[green,fill=green!25,fill opacity=0.75] (barycentric cs:L1=5,L2=-1,L3=-1) -- (barycentric cs:L1=1,L2=1,L3=1) -- (barycentric cs:L1=5,L2=-3,L3=1);
\draw[dashed,red,thick] (barycentric cs:L1=4.5,L2=1.5,L3=-3)--(barycentric cs:L1=0.5,L2=-2.5,L3=5);
\fill[yellow] (barycentric cs:L1=2.5,L2=-0.5,L3=1) circle (0.05);
\node[anchor=west,xshift=-1mm] at (barycentric cs:L1=2.5,L2=-0.5,L3=1) {$\mu^0_X$};
\draw[->,thick] (barycentric cs:L1=1,L2=1,L3=1) -- (barycentric cs:L1=2,L2=2,L3=-1);
\node[anchor=east] at (barycentric cs:L1=2,L2=2,L3=-1) {$\alpha_2$};
\draw[->,thick] (barycentric cs:L1=1,L2=1,L3=1) -- (barycentric cs:L1=1,L2=0,L3=2);
\node[anchor=west] at (barycentric cs:L1=1,L2=0,L3=2) {$\alpha_1$};
\node[anchor=south,yshift=-2.0mm] at (barycentric cs:L1=3,L2=0,L3=0) {\footnotesize $\omega_1=\mu$};
\node[anchor=north] at (barycentric cs:L1=2,L2=0,L3=1) {$\omega_2$};
\node[anchor=north] at (barycentric cs:L1=1,L2=1,L3=1) {$0$};
\node[anchor=south] at (barycentric cs:L1=4,L2=-1,L3=0) {$\lambda$};
\foreach \x in {-2,-1,...,2}{
    \foreach \y in {-2,-1,...,2}{
        \foreach \z in {-2,-1,...,2}{
        \fill (barycentric cs:L1=1.0+\x+\z,L2=1.0-\x+\y,L3=1.0-\y-\z) circle (0.05);
        }
    }
}
\end{tikzpicture}};
\node[anchor=north] at (B2.south) {$\Phi=B_2$};
\node[anchor=north] at (G2.south) {$\Phi=G_2$};
\end{tikzpicture}

\caption{Examples~\ref{ex:slice_b2} and~\ref{ex:slice_g2} of what the setting of Lemma~\ref{lem:slice} might look like for the rank $2$ root systems $B_2$ and $G_2$.} \label{fig:slice}
\end{figure}

We will show that $\mu \in \Pi(\lambda)$ and thus that $\mu \in \Pi^Q(\lambda) \cap (\mu + \mathrm{Span}_{\mathbb{R}}(X))$ to complete the proof of the lemma. But before we do that, let us give two rank $2$ examples of what the setting of this lemma might look like after we have reduced to a case satisfying conditions~\eqref{cond:integrality_1}-\eqref{cond:integrality_4} above. In these examples we follow the numbering of the simple roots from Figure~\ref{fig:dynkin_classification} in the appendix.

\begin{example} \label{ex:slice_b2}
Suppose $\Phi=B_2$, $\lambda = \omega_1+\omega_2$, $\mu = -\omega_1+\omega_2$ and $X=\{\alpha_1\}$. This is depicted on the left of Figure~\ref{fig:slice}. In this figure, the permutohedron $\Pi(\lambda)$ is the region shaded in blue. The dominant cone $P_{\geq 0}^{\mathbb{R}}$ is the region shaded in green. The affine subspace $\mu+\mathrm{Span}_{\mathbb{R}}(X)$ is the dashed red line (in fact in this case it is a linear subspace). Points in the coset $Q+\lambda$ are represented by black circles; other points of interest are marked by yellow circles circles. It is easy to verify that conditions~\eqref{cond:integrality_1}-\eqref{cond:integrality_4} hold in this case: for example, $\pi_{I}(\mu)=\mu=\frac{1}{2}\alpha_1$ is a minuscule weight of $\Phi_{\{1\}}$. Observe that $\mu^0_X=0$ is a weight of $\Phi$, but that it does not belong to the coset $Q+\lambda$.
\end{example}

\begin{example} \label{ex:slice_g2}
Suppose $\Phi=G_2$, $\lambda = \omega_1+\omega_2$, $\mu = \omega_1$ and $X=\{\alpha_2\}$. This is depicted on the right of Figure~\ref{fig:slice}. In this figure, the permutohedron $\Pi(\lambda)$ is the region shaded in blue. The dominant cone $P_{\geq 0}^{\mathbb{R}}$ is the region shaded in green. The affine subspace~$\mu+\mathrm{Span}_{\mathbb{R}}(X)$ is the dashed red line. Points in the coset $Q+\lambda$ are represented by black circles; other points of interest are marked by yellow circles circles. It again is easy to verify that conditions~\eqref{cond:integrality_1}-\eqref{cond:integrality_4} hold in this case: for example, $\pi_{I}(\mu)=\frac{1}{2}\alpha_2$ is a minuscule weight of $\Phi_{\{2\}}$. Observe that $\mu^0_X$ is not even a weight of $\Phi$ here.
\end{example}

Note crucially that even though $\mu^0_X \in  P_{\geq 0}^{\mathbb{R}}$ and $\pi_{I}(\mu)$ is zero or a minuscule weight of~$\Phi_I$, it is not necessarily the case that $\mu \in P_{\geq 0}$. Indeed, we can see already in Example~\ref{ex:slice_b2} that $\mu$ is not dominant, and unavoidably so. If we could assume $\mu \in P_{\geq 0}$, we would be done, because $\mu \leq \lambda$ (by condition~\eqref{cond:integrality_4}) and so by Proposition~\ref{prop:perm_containment} we would have $\mu \in \Pi(\lambda)$. The fact that $\mu$ is not dominant in general presents us with some difficulties. It means that we have to consider the dominant representative $\mu_{\mathrm{dom}}$ of $\mu$ and have to analyze how ``different'' $\mu$ and $\mu_{\mathrm{dom}}$ can be. As it turns out $\mu$ and~$\mu_{\mathrm{dom}}$ cannot be ``too different.'' This is made precise by the following:

\begin{claim} \label{claim:one}
If $\nu \in P_{\geq 0}$ is a dominant weight with $\mu \leq \nu$, then $\mu_{\mathrm{dom}}\leq \nu$.
\end{claim}

Before proving this claim, let us explain why it is enough to finish the proof of the lemma. We assert that if Claim~\ref{claim:one} is true, then $\mu \in \Pi(\lambda)$. Indeed, by condition~\eqref{cond:integrality_4} above we have that~$\mu \leq \lambda$. Thus Claim~\ref{claim:one} says~$\mu_{\mathrm{dom}} \leq \lambda$. So by Proposition~\ref{prop:perm_containment} we get that~$\mu_{\mathrm{dom}} \in \Pi(\lambda)$ and hence that~$\mu \in \Pi(\lambda)$, finishing the proof of the lemma.

We now proceed to prove Claim~\ref{claim:one}. This is where we invoke Lemma~\ref{lem:root_projection}. Recall that if $\mathcal{P}$ is a convex polytope in $V$ containing the origin then \emph{polar dual} $\mathcal{P}^*$ of $\mathcal{P}$ is the polytope $\mathcal{P}^* \coloneqq \{v \in V\colon \langle v,u\rangle \leq 1 \textrm{ for all $u\in\mathcal{P}$}\}$. Hence the polar dual of the root polytope $\mathcal{P}_{\Phi^\vee}$ is $\mathcal{P}^*_{\Phi^\vee} = \{v\in V\colon \langle v,\alpha^\vee\rangle \leq 1  \textrm{ for all $\alpha\in\Phi$}\}$. By Lemma~\ref{lem:root_projection} we have that $\pi_I(\mathcal{P}_{\Phi^\vee}) \subseteq \kappa \cdot \mathcal{P}_{\Phi_I^\vee}$ for some $1 \leq \kappa < 2$. By basic properties of polar duality, this implies that $\mathcal{P}^*_{\Phi_I^\vee} \cap \mathrm{Span}_{\mathbb{R}}(\Phi_I) \subseteq \kappa \cdot \mathcal{P}^*_{\Phi^\vee}$. Note that $\pi_I(\mu)\in \mathcal{P}^*_{\Phi_I^\vee} \cap \mathrm{Span}_{\mathbb{R}}(\Phi_I) $ because $\pi_I(\mu)$ is zero or a minuscule weight of $\Phi_I$.  Thus $\langle \pi_I(\mu),\alpha^\vee\rangle > -2$ for all $\alpha \in \Phi^+$. But since $\mu = \pi_I(\mu) + \mu^0_X$ with $\mu^0_X \in P_{\geq 0}^{\mathbb{R}}$, this means $\langle \mu,\alpha^\vee\rangle > -2$ for all $\alpha \in \Phi^+$. Since $\mu$ is a weight of $\Phi$ these inner products must be integers; hence in fact $\langle \mu,\alpha^\vee\rangle \geq -1$ for all $\alpha \in \Phi^+$. Therefore by Proposition~\ref{prop:unique_dominant_greater} we conclude that $\mu_{\mathrm{dom}}$ is the minimal dominant weight greater than or equal to $\mu$ in root order. That is to say, we conclude that Claim~\ref{claim:one} holds.
\end{proof}

\begin{remark} \label{rem:lem_equivalence}
Suppose that, in contradiction to Lemma~\ref{lem:root_projection}, there exists $I\subseteq [n]$ and $u \in \mathrm{Span}_{\mathbb{R}}(\{\alpha_i\colon i \in I\})$ with $\langle u,\alpha^\vee \rangle \in \{0,1\}$ for all $\alpha \in \Phi_I^+$, and for which $\langle u,\alpha_j^\vee \rangle \leq -2$ for some simple root $\alpha_j$ with $j\notin I$. Then it would be easy to construct a counterexample to Lemma~\ref{lem:slice}: we could take $X \coloneqq \{\alpha_i\colon i \in I\}$; $\mu$ to be such that $\pi_I(\mu)=u$, $\mu-\pi_I(\mu) \in P_{\geq 0}^{\mathbb{R}}$, and $\langle \mu,\alpha_j^\vee\rangle = -2$; and $\lambda\neq \mu_{\mathrm{dom}}$ to be the minimal dominant weight greater than $\mu$ in root order. In this sense, Lemmas~\ref{lem:slice} and~\ref{lem:root_projection} are ``almost'' equivalent to one another.
\end{remark}

Unfortunately, the only proof of Lemma~\ref{lem:root_projection} we could find requires us to invoke the classification of root systems and do a case-by-case check. Thus we have relegated the verification of Lemma~\ref{lem:root_projection} to Appendix~\ref{sec:root_polytopes}. It is worth noting that having reduced Lemma~\ref{lem:slice} to Lemma~\ref{lem:root_projection} is progress at least in the sense for each fixed root system $\Phi$, verifying Lemma~\ref{lem:root_projection} amounts to a finite computation, whereas it is not a priori evident that Lemma~\ref{lem:slice} is a finite statement even for fixed~$\Phi$.

Having established the requisite integrality property of permutohedra, modulo the case-by-case verification of the root polytope projection-dilation property provided in Appendix~\ref{sec:root_polytopes}, we can now complete the proof of the formula for the number of lattice points in a permutohedron plus dilating regular permutohedron. In the language of quotients from Section~\ref{sec:intro}, Lemma~\ref{lem:slice} becomes the following (which is stated as Lemma~\ref{lem:slice_intro} in Section~\ref{sec:intro}):

\begin{cor} \label{cor:quotient}
Let $\lambda \in P_{\geq 0}$  and $X\subseteq \Phi^+$. Then,
\[\mathrm{quot}_X(\Pi(\lambda))\cap \mathrm{quot}_X(Q+\lambda) = \mathrm{quot}_X(\Pi^Q(\lambda)).\]
\end{cor}
\begin{proof}
A point in $\mathrm{quot}_X(\Pi(\lambda))\cap \mathrm{quot}_X(Q+\lambda)$ is an affine subspace of $V$ of the form $\mu+\mathrm{Span}_{\mathbb{R}}(X)$ for some $\mu \in Q+\lambda$ satisfying $(\mu+\mathrm{Span}_{\mathbb{R}}(X))\cap \Pi(\lambda)\neq\varnothing$, while a point in~$ \mathrm{quot}_X(\Pi^Q(\lambda))$ is an affine subspace of~$V$ of the form $\mu+\mathrm{Span}_{\mathbb{R}}(X)$ for some~$\mu \in \Pi^Q(\lambda)$. These two kinds of affine subspaces coincide thanks to Lemma~\ref{lem:slice}.
\end{proof}

Putting it all together:

\begin{thm} \label{thm:perm}
Let $\lambda \in P_{\geq 0}$ and $\mathbf{k} \in \mathbb{N}[\Phi]^W$. Then
\[ \#\Pi^Q(\lambda+\rho_{\mathbf{k}}) = \sum_{\substack{X\subseteq \Phi^+ \\ \textrm{ $X$ is linearly} \\ \textrm{independent}}} \# \mathrm{quot}_X(\Pi^Q(\lambda))\cdot \mathrm{rVol}_Q(X) \, \mathbf{k}^{X},\]
where $\mathbf{k}^{X} \coloneqq \prod_{\alpha \in X} \mathbf{k}(\alpha)$.
\end{thm}
\begin{proof}
This follows immediately from Theorem~\ref{thm:poly_plus_zono} together with Corollary~\ref{cor:quotient}.
\end{proof}

\subsection{The lattice point formula in Type A} \label{subsec:type_a}

In this subsection we briefly discuss what the formula for $\#\Pi^Q(\lambda+\rho_{\mathbf{k}})$ looks like in Type~A. So suppose $\Phi=A_n$. Recall that, using the standard realization of $A_n$, the roots of $\Phi$ are $e_i-e_j$ for $1 \leq i, j \leq n+1$, where $e_i\in\mathbb{R}^{n+1}$ is the $i$th standard basis vector. The positive roots are $e_i-e_j$ for~$1\leq i<j\leq n+1$. The vector space in which $\Phi$ lives is $V=\mathbb{R}^{n+1}/(1,1,\ldots,1)$, where we mod out by the ``all ones'' vector. The Weyl group $W$ is the symmetric group $S_{n+1}$ acting on $\mathbb{R}^{n+1}$ by permuting entries. The weight lattice is $P=\mathbb{Z}^{n+1}/(1,1,\ldots,1)$. But in fact we can ``extend'' in the obvious way the vector space~$V$ to be all of $\mathbb{R}^{n+1}$ and the lattice $P$ to be all of $\mathbb{Z}^{n+1}$, and if we do so, the notion of $W$-permutohedra coincides with the usual notion of permutohedra. That is, for a vector $\mathbf{a}=(a_1,\ldots,a_{n+1})\in\mathbb{Z}^{n+1}$, we define the \emph{permutohedra} of $\mathbf{a}$ to be $\Pi(\mathbf{a}) \coloneqq \mathrm{ConvexHull} \, \{(a_{\sigma_1},\ldots,a_{\sigma_{n+1}})\colon \sigma\in S_{n+1} \}$. The vector $\mathbf{a}$ is dominant if $a_1 \geq a_2 \geq \cdots \geq a_{n+1}$. Finally, note that we may take the Weyl vector in this setting to be $\rho=(n,n-1,n-2,\ldots,0)\in\mathbb{Z}^{n+1}$.

So Theorem~\ref{thm:perm} gives a formula for~$\#\Pi(a_1+kn,a_2+k(n-1),\ldots,a_{n+1})\cap\mathbb{Z}^{n+1}$ where~$a_1\geq a_2\geq \cdots \geq a_{n+1} \in \mathbb{Z}$. How can we understand this formula more concretely? 

First note that because the collection of vectors $e_i-e_j$ is totally unimodular, we will have $\mathrm{rVol}_{\mathbb{Z}^n}(X)=1$ for all linearly independent $X\subseteq\Phi^+$.

Next, note that via the bijection which sends a positive root $e_i-e_j$ to an edge~$\{i,j\}$, a subset $X\subseteq \Phi^+$ which is linearly independent is the same thing as a forest~$F_X$ on the vertex set~$[n+1]$. Moreover, the subspace $\mathrm{Span}_{\mathbb{R}}(X)$ only depends on the connected components of this forest $F_X$: suppose $F_X$ has components~$I_1,\ldots,I_m\subseteq[n+1]$; then we can explicitly realize the quotient map $\mathrm{quot}_X\colon \mathbb{R}^{n+1} \twoheadrightarrow \mathbb{R}^{m}$ by
\[\mathrm{quot}_{X}(b_1,b_2,\ldots,b_n) \coloneqq \left( \sum_{i\in I_1}b_i,\sum_{i\in I_2}b_i,\ldots, \sum_{i\in I_m}b_i\right).\]
Observe that this construction of the quotient satisfies $\mathrm{quot}_X(\mathbb{Z}^{n+1})=\mathbb{Z}^m$.

In fact, up to the action of the Weyl group (i.e., the symmetric group), a subspace of the form $\mathrm{Span}_{\mathbb{R}}(X)$ only depends on the partition $\lambda=(\lambda_1,\lambda_2,\ldots)$ of $n+1$ which records in decreasing order the sizes of the connected components of $F_X$. (By \emph{partition of $n+1$} we just mean a sequence $\lambda_1\geq \lambda_2\geq \cdots \in \mathbb{Z}_{\geq 0}$ of nonnegative integers which sum to $n+1$; do not confuse the $\lambda$ used only in this subsection to denote integer partitions with a weight~$\lambda\in P$.) Let $\lambda \vdash n+1$ be a partition of $n+1$. We use $\ell(\lambda)$ to denote the \emph{length} of $\lambda$, i.e., the minimum $i \in \mathbb{Z}_{\geq 0}$ for which $\lambda_j=0$ for all~$j> i$. We use $f_{\lambda}$ to denote the number of labeled forests on~$n+1$ vertices whose connected component sizes are~$\lambda$ in decreasing order. Note that Cayley's formula gives~$f_{\lambda} = \prod_{i=1}^{\ell(\lambda)} \lambda_i^{\lambda_i-2}$. Finally, we define the corresponding quotient map $\mathrm{quot}_{\lambda}\colon \mathbb{R}^{n+1} \twoheadrightarrow \mathbb{R}^{\ell(\lambda)}$ by
\[\mathrm{quot}_{\lambda}(b_1,b_2,\ldots,b_n) \coloneqq \left( \sum_{i=1}^{\lambda_1}b_i,\sum_{i=\lambda_1+1}^{\lambda_1+\lambda_2}b_i,\ldots, \sum_{i=\lambda_1+\cdots+\lambda_{\ell(\lambda)-1}+1}^{n+1}b_i\right).\]
Then Theorem~\ref{thm:perm} becomes:

\begin{thm} \label{thm:type_a}
For any $\mathbf{a}=(a_1,\ldots,a_{n+1})\in\mathbb{Z}^{n+1}$ with $a_1\geq a_2\geq \cdots \geq a_{n+1}$ and any~$k\in\mathbb{Z}_{\geq 0}$, we have
\[\#\Pi(a_1+kn,a_2+k(n-1),\ldots,a_{n+1})\cap\mathbb{Z}^{n+1} = \hspace{-0.2cm} \sum_{\lambda \vdash n+1}  \hspace{-0.2cm} \#\mathrm{quot}_{\lambda}(\Pi(\mathbf{a})\cap\mathbb{Z}^{n+1}) \cdot f_{\lambda} \, k^{n+1-\ell(\lambda)}. \]
\end{thm}

The quantities $\#\mathrm{quot}_{\lambda}(\Pi(\mathbf{a})\cap\mathbb{Z}^{n+1})$ appearing in Theorem~\ref{thm:type_a} are in general not so easy to understand. For instance, the quotient $\mathrm{quot}_{\lambda}(\Pi(\mathbf{a}))$ need not be a permutohedron in $\mathbb{R}^{\ell(\lambda)}$. However, we can at least say that it is a \emph{generalized permutohedron} in the sense of~\cite{postnikov2009permutohedra}. Indeed, we can explicitly give its facet description: $\mathrm{quot}_{\lambda}(\Pi(\mathbf{a}))$ consists of all points~$(x_1,\ldots,x_{\ell(\lambda)})\in\mathbb{R}^{\ell(\lambda)}$ for which $x_1+\cdots+x_{\ell(\lambda)} = a_1+\ldots+a_{n+1}$ and
\[\sum_{i\in I} x_i \leq a_1+\cdots+a_{\lambda(I)} \textrm{ for any subset $I\subseteq [\ell(\lambda)]$, where $\lambda(I)\coloneqq\sum_{i\in I}\lambda_ i$}.\]
There are formulas for the number of lattice points of such a polytope (see~\cite{postnikov2009permutohedra}), but none are so explicit.

Nevertheless, for some special choices of $\mathbf{a}\in\mathbb{Z}^{n+1}$ corresponding to minuscule weights we can give a more combinatorial description of $\#\mathrm{quot}_{\lambda}(\Pi(\mathbf{a})\cap\mathbb{Z}^{n+1})$. Namely, suppose that $\mathbf{a}=(\overbrace{1,1,\ldots,1}^{i},0,\ldots,0)$ for some $i\in[n+1]$. Then, since $\Pi(\mathbf{a})\cap\mathbb{Z}^{n+1}$ just consists of permutations of $\mathbf{a}$, we get that 
\[\#\mathrm{quot}_{\lambda}(\Pi(\mathbf{a})\cap\mathbb{Z}^{n+1}) = \#\left\{(\mu_1,\ldots,\mu_{\ell(\lambda)})\in\mathbb{Z}_{\geq 0}^{\ell(\lambda)}\colon\mu_j\leq \lambda_j \textrm{ for all $j$, and } \sum_{j=1}^{i}\mu_j=i\right\}.\]
In other words $\#\mathrm{quot}_{\lambda}(\Pi(\mathbf{a})\cap\mathbb{Z}^{n+1})$ is the number of weak compositions of $i$  into $\ell(\lambda)$ parts whose corresponding diagram fits inside the Young diagram of $\lambda$. Observe that this number would be the same if we replaced $i$ by $n+1-i$, which reflects the symmetry of the Dynkin diagram of Type~A. Let us end with a couple of examples of what the formula in Theorem~\ref{thm:type_a} reduces to for small $i$:

\begin{example}
Suppose $i=1$. For any $\lambda \vdash n+1$, choosing a composition of~$1$ whose diagram fits inside $\lambda$ is the same as choosing a row of~$\lambda$, so the number of such compositions is $\ell(\lambda)$. Thus for $k\in\mathbb{Z}_{\geq 0}$ we have
\[\#\Pi(1+kn,k(n-1),k(n-2),\ldots,k,0)\cap\mathbb{Z}^{n+1}= \sum_{\lambda \vdash n+1}  \ell(\lambda) \, f_{\lambda} \, k^{n+1-\ell(\lambda)}. \]
\end{example}

\begin{example}
Suppose $i=2$. Let $\lambda \vdash n+1$. To make a composition of $2$ which fits inside $\lambda$, we can either choose two different rows of $\lambda$ in which to place one box each, or we can choose one row of $\lambda$ to put two boxes in, but we can only do that if the size of that row is at least two. Therefore the number of such compositions is $\binom{\lambda'_1}{2} + \lambda'_2$, where $\lambda' = (\lambda'_1,\lambda'_2,\ldots)$ is the conjugate partition to $\lambda$. Thus for $k\in\mathbb{Z}_{\geq 0}$ we have
\[\#\Pi(1+kn,1+k(n-1),k(n-2),\ldots,k,0)\cap\mathbb{Z}^{n+1}= \sum_{\lambda \vdash n+1} \left(\binom{\lambda'_1}{2}+ \lambda'_2\right) \, f_{\lambda} \, k^{n+1-\ell(\lambda)} .\]
\end{example}

\section{Symmetric Ehrhart-like polynomials} \label{sec:sym_formula}

\begin{figure}
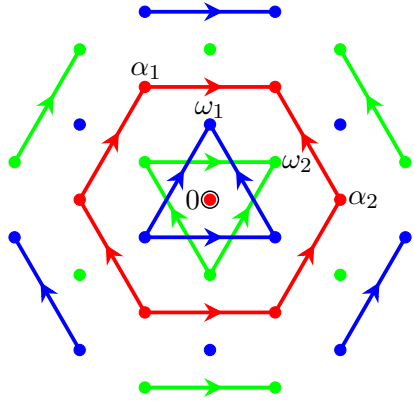


\caption{The $k=0$ symmetric interval-firing process for $\Phi=A_2$.} \label{fig:a2_k0}
\end{figure}

\begin{figure}
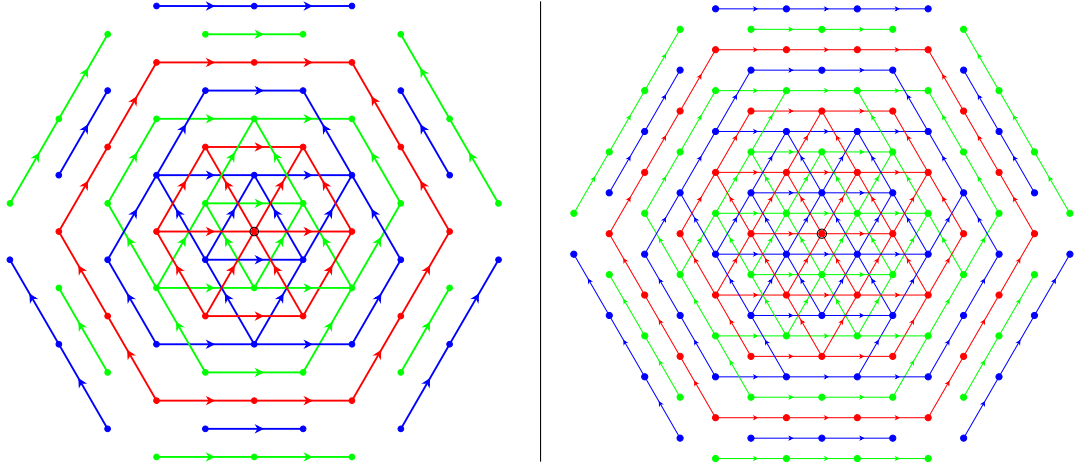

\scalebox{0.125}{%
}
\caption{The $k=1,2$ symmetric interval-firing processes for $\Phi=A_2$.} \label{fig:a2_k12}
\end{figure}

\subsection{Background on interval-firing} In this section we finally return to the study of the interval-firing processes introduced in Section~\ref{sec:intro}. We continue to fix an irreducible root system $\Phi$ as in Section~\ref{sec:lattice_pts}. In this subsection we review some definitions and results from~\cite{galashin2017rootfiring1}. We will define the interval-firing processes at a slightly greater level of generality than what was described in the introduction: now we allow our deformation parameter to be an element of~$\mathbb{N}[\Phi]^W$ rather than just~$\mathbb{Z}_{\geq 0}$. For~$\mathbf{k} \in \mathbb{N}[\Phi]^W$ the \emph{symmetric interval-firing process} is the binary relation on~$P$ defined by
\[ \lambda \xrightarrow[\mathrm{sym},\,\mathbf{k}]{} \lambda + \alpha \textrm{ whenever $\langle \lambda + \frac{\alpha}{2},\alpha^\vee\rangle \in \{-\mathbf{k}(\alpha),\ldots,\mathbf{k}(\alpha)\}$ for $\lambda \in P$, $\alpha \in \Phi^+$};\]
and the \emph{truncated interval-firing process} is the binary relation on $P$ defined by
\[ \lambda \xrightarrow[\mathrm{tr},\, \mathbf{k}]{} \lambda + \alpha \textrm{ whenever $\langle \lambda+\frac{\alpha}{2},\alpha^\vee\rangle \in \{-\mathbf{k}(\alpha)+1,\ldots,\mathbf{k}(\alpha)\}$ for $\lambda \in P$, $\alpha \in \Phi^+$}.\]

Let us take a moment to discuss the names for these interval-firing process. The symmetric interval-firing process is so named because the symmetric closure of the relation $ \xrightarrow[\mathrm{sym},\,\mathbf{k}]{}$ is invariant under the Weyl group; that is:

\begin{prop}[{See~\cite[\oldpapersymprop]{galashin2017rootfiring1}}] \label{prop:symmetry}
For $\mathbf{k}\in\mathbb{N}[\Phi]^W$, $w \in W$, and $\lambda,\mu \in P$, we have $(\mu \xrightarrow[\mathrm{sym},\,\mathbf{k}]{} \lambda \textrm{ or } \lambda \xrightarrow[\mathrm{sym},\,\mathbf{k}]{} \mu)$ if and only if $(w(\mu) \xrightarrow[\mathrm{sym},\,\mathbf{k}]{} w(\lambda) \textrm{ or } w(\lambda) \xrightarrow[\mathrm{sym},\,\mathbf{k}]{} w(\mu))$.
\end{prop}

The truncated process is so-named because the interval in its definition is truncated by one element on the left compared to the interval for the symmetric process. 

In this paper we will mostly be focused on the symmetric interval-firing process.

\begin{example}
Suppose that $\Phi=A_2$. Since $\Phi$ is simply laced, we have $\mathbf{k}=k$ is some constant. The $k=0$ symmetric interval-firing process for $\Phi$ is depicted in Figure~\ref{fig:a2_k0}. Of course in this figure we draw an arrow from $\mu$ to $\lambda$ if $\mu \xrightarrow[\mathrm{sym},0]{} \lambda$. Here the three different colors correspond to the three different cosets of~$P/Q$. In this figure we depict only the ``interesting'' portion of this relation near the origin. In Figure~\ref{fig:a2_k12} we depict the $k=1$ (left) and $k=2$ (right) symmetric interval-firing processes for $\Phi$. Observe how, as $k$ grows, the figures look the ``same,'' except that they get ``dilated.''
\end{example}

Now we will discuss confluence and stabilizations for the interval-firing processes. So let us review these notions, which apply to an arbitrary binary relation on a set. Let~$\rightarrow$ be a binary relation on a set~$S$. We use $\xrightarrow{*}$ to denote the reflexive, transitive closure of~$\rightarrow$. We say that $\rightarrow$ is \emph{confluent} if for every $x,y_1,y_2\in S$  with $x \xrightarrow{*}y_1$ and~$x\xrightarrow{*}y_2$ there exists~$y_3\in S$ with $y_1\xrightarrow{*} y_3$ and $y_2\xrightarrow{*} y_3$. We say that $\rightarrow$ is \emph{terminating} if there does not exits an infinite sequence $x_0\rightarrow x_1\rightarrow x_2 \rightarrow$ of related elements $x_0,x_1,x_2,\ldots \in S$. We say that $x\in S$ is~\emph{$\rightarrow$-stable} (or just \emph{stable} if the context is clear) if there does not exist $y \in S$ with~$x\rightarrow y$. Observe that if $\rightarrow$ is confluent and terminating, then for every $x\in S$ there exists a unique stable $y\in S$ with~$x\xrightarrow{*}y$ and we call this $y$ the $\rightarrow$-\emph{stabilization} (or just \emph{stabilization} if the context is clear) of $x$.

In~\cite{galashin2017rootfiring1} it was proved that the symmetric and truncated interval-firing processes are confluent. However, there is a slight caveat here when we use the more general parameters $\mathbf{k} \in \mathbb{N}[\Phi]^W$. Namely, we need to disallow some pathological choices of these parameters. For an irreducible root system~$\Phi$ there are at most two orbits of roots under the Weyl group. If $\Phi$ has only a single Weyl group orbit of roots, then we say that $\Phi$ is \emph{simply laced}. Thus if $\Phi$ is simply laced, then for any $\mathbf{k}\in\mathbb{N}[\Phi]^W$ there exists some~$k\in \mathbb{Z}_{\geq0}$ such that $\mathbf{k}=k$. In the non-simply laced case, the roots are divided into an orbit of \emph{long roots} (those which maximize the quantity $\langle \alpha,\alpha\rangle$ among $\alpha\in \Phi$) and and orbit of \emph{short roots} (those which minimize the quantity $\langle \alpha,\alpha\rangle$). Thus if $\Phi$ is not simply laced, then for any $\mathbf{k}\in\mathbb{N}[\Phi]^W$ there exist $k_l,k_s\in \mathbb{Z}_{\geq0}$  such that $\mathbf{k}(\alpha)=k_l$ if~$\alpha$ is long and $\mathbf{k}(\alpha)=k_s$ if $\alpha$ is short.

 \begin{definition}
Here we define what it means for $\mathbf{k} \in \mathbb{N}[\Phi]^W$ to be \emph{good}. If $\Phi$ is simply laced, then $\mathbf{k} \in \mathbb{N}[\Phi]^W$ is always good. If $\Phi$ is not simply laced, then $\mathbf{k} \in \mathbb{N}[\Phi]^W$  is good if $k_s=0\Rightarrow k_l=0$. Note in particular that if $\mathbf{k}=k$ is constant, then it is good.
 \end{definition}

\begin{thm}[{See~\cite[\oldpaperconfluence]{galashin2017rootfiring1}}]  \label{thm:confluence}
If $\mathbf{k}\in \mathbb{N}[\Phi]^W$ is good then both the symmetric and truncated interval-firing processes are confluent (and terminating).
\end{thm}

A key technical tool in the proof of Theorem~\ref{thm:confluence} is the ``permutohedron non-escaping lemma.'' We will also find the non-escaping lemma useful for our purposes; hence, let us state the following version of this lemma:

\begin{lemma}[{See~\cite[\oldpapernonescaping]{galashin2017rootfiring1}}]  \label{lem:perm_non_escaping}
Let $\mathbf{k} \in \mathbb{N}[\Phi]^W$ be good and let $\lambda \in P_{\geq 0}$. Let $\mu,\nu \in P$ with $\mu \in \Pi^Q(\lambda+\rho_{\mathbf{k}})$ and $\nu \notin  \Pi^Q(\lambda+\rho_{\mathbf{k}})$. Then it cannot be the case that~$\mu \xrightarrow[\mathrm{sym},\,\mathbf{k}]{} \nu$ or that~$\nu\xrightarrow[\mathrm{sym},\,\mathbf{k}]{}\mu$.
\end{lemma}

Thanks to Theorem~\ref{thm:confluence} we can ask about stabilizations for the symmetric and truncated interval-firing processes. The first thing we need to understand is what are the stable points. To describe these points we will need to define a certain map $\eta\colon P\to P$.

It is well-known that for any $I\subseteq[n]$, every coset of the parabolic subgroup $W_I$ has a unique element of minimal length (see e.g.~\cite[\S2.4]{bjorner2005coxeter}). We use $W^I$ to denote the minimal length coset representatives of $W_I$. The elements of $W^I$ are exactly those $w \in W$ for which $w(\Phi_I)\cap\Phi^+=w(\Phi_I\cap\Phi^+)$.

For a dominant $\lambda \in P_{\geq 0}$ we define $I^0_{\lambda} \coloneqq \{i\in [n]\colon \langle \lambda,\alpha_i^\vee\rangle= 0\}$. For any $\lambda \in P$, the set of $w\in W$ for which $w^{-1}(\lambda) = \lambda_{\mathrm{dom}}$ is a coset of $W_{I^0_{\lambda}}$. Hence it makes sense to define $w_{\lambda}$ to be the minimal length element of the Weyl group for which $w^{-1}(\lambda)$ is dominant. In fact, $W^{I^0_{\lambda}} = \{w_\mu\colon \mu \in W(\lambda)\}$ for a dominant~$\lambda \in P_{\geq 0}$.

For $\mathbf{k}\in \mathbb{N}[\Phi]^W$ we then define $\eta_{\mathbf{k}}\colon P\to P$ by $\eta_{\mathbf{k}}(\lambda) \coloneqq \lambda + w_{\lambda}(\rho_{\mathbf{k}})$. Some basic properties of $\eta_{\mathbf{k}}$ established in~\cite{galashin2017rootfiring1} are that it is always injective, and that $\eta_{\mathbf{a}}(\eta_{\mathbf{b}}) = \eta_{\mathbf{a}+\mathbf{b}}$ (hence if we set $\eta \coloneqq \eta_1$, then $\eta^k=\eta_k$). The stable points of the interval-firing processes are given in terms of $\eta$ as follows:

\begin{prop}[{See~\cite[\oldpapersinks]{galashin2017rootfiring1}}] \label{prop:sinks}
For any $\mathbf{k}\in\mathbb{N}[\Phi]^W$, the stable points of $\xrightarrow[\mathrm{sym},\,\mathbf{k}]{}$ are $\{\eta_{\mathbf{k}}(\lambda)\colon \lambda \in P, \langle \lambda,\alpha^\vee\rangle \neq -1 \textrm{ for all } \alpha\in\Phi^+\}$ and the stable points of $\xrightarrow[\mathrm{tr},\,\mathbf{k}]{}$ are $\{\eta_{\mathbf{k}}(\lambda)\colon \lambda \in P\}$.
\end{prop}

The condition $\langle \lambda,\alpha^\vee\rangle \neq -1$  for all $\alpha\in\Phi^+$ can also be described in terms of parabolic subgroups. For dominant $\lambda \in P_{\geq 0}$, let us define $I^{0,1}_{\lambda} \coloneqq \{i\in [n]\colon \langle \lambda,\alpha_i^\vee\rangle\in\{0,1\}\}$. And for an arbitrary weight $\lambda \in P$ we set $I^{0,1}_{\lambda} \coloneqq I^{0,1}_{\lambda_{\mathrm{dom}}}$. Then~$\lambda \in P$ has $\langle \lambda,\alpha^\vee\rangle \neq -1$ for all $\alpha\in\Phi^+$  if and only if $w_{\lambda}$ is of minimal length in its coset of $W_{I^{0,1}_{\lambda}}$. In fact, $W^{I^{0,1}_{\lambda}} = \{w_{\mu}\colon \mu \in W(\lambda), \langle \mu,\alpha^\vee\rangle \neq -1 \textrm{ for all } \alpha\in\Phi^+\}$ for a dominant $\lambda \in P_{\geq 0}$. 

Theorem~\ref{thm:confluence} and Proposition~\ref{prop:sinks} say that it makes sense to define for good $\mathbf{k}\in\mathbb{N}[\Phi]^W$ the \emph{stabilization maps} $s^{\mathrm{sym}}_{\mathbf{k}},s^{\mathrm{tr}}_{\mathbf{k}}\colon P\to P$ by
\begin{align*}
s^{\mathrm{sym}}_{\mathbf{k}}(\mu) = \lambda &\Leftrightarrow \textrm{ the $\xrightarrow[\mathrm{sym},\,\mathbf{k}]{}$-stabilization of $\mu$ is $\eta_{\mathbf{k}}(\lambda)$}; \\
s^{\mathrm{tr}}_{\mathbf{k}}(\mu) = \lambda &\Leftrightarrow \textrm{ the $\xrightarrow[\mathrm{tr},\,\mathbf{k}]{}$-stabilization of $\mu$ is $\eta_{\mathbf{k}}(\lambda)$}.
\end{align*}
For $\lambda\in P$ and good $\mathbf{k}\in\mathbb{N}[\Phi]^W$ we then define the quantities
\begin{align*}
L^{\mathrm{sym}}_{\lambda}(\mathbf{k}) &\coloneqq \#(s^{\mathrm{sym}}_\mathbf{k})^{-1}(\lambda); \\
L^{\mathrm{tr}}_{\lambda}(\mathbf{k}) &\coloneqq \#(s^{\mathrm{tr}}_\mathbf{k})^{-1}(\lambda),
\end{align*}
In other words, $L^{\mathrm{sym}}_{\lambda}(\mathbf{k})$ and $L^{\mathrm{tr}}_{\lambda}(\mathbf{k})$ count the number of weights with given stabilization as a function of $\mathbf{k}$.

\begin{example}
Suppose that $\Phi=A_2$, as in Figures~\ref{fig:a2_k0} and~\ref{fig:a2_k12}. Then we have
\begin{align*}
L_{0}^{\mathrm{sym}}(k) &= 3k^2+3k+1; \\
L_{\omega_1}^{\mathrm{sym}}(k) &= 3k^2+6k+3; \\
L_{\omega_2}^{\mathrm{sym}}(k) &= 3k^2+6k+3; \\
L_{\omega_1+\omega_2}^{\mathrm{sym}}(k) &= 6k+6.
\end{align*}
Here $L_{0}^{\mathrm{sym}}(k)$ counts the points in the inner red ``regular'' hexagon in Figures~\ref{fig:a2_k0} and~\ref{fig:a2_k12} as it grows (in Figure~\ref{fig:a2_k0} it is just a single point). And $L_{\omega_1+\omega_2}^{\mathrm{sym}}(k)$ counts the points on the boundary of the outer red ``regular'' hexagon. Meanwhile, $L_{\omega_1}^{\mathrm{sym}}(k)$ counts the points in the inner blue ``irregular'' hexagon, and similarly $L_{\omega_2}^{\mathrm{sym}}(k)$ counts the points in the inner green ``irregular'' hexagon (in Figure~\ref{fig:a2_k0} these are actually triangles).
\end{example}

\begin{conj}[{See~\cite[\oldpaperconjecture]{galashin2017rootfiring1}}] \label{conj:positivity}
For all good $\mathbf{k}\in\mathbb{N}[\Phi]^W$, both $L^{\mathrm{sym}}_{\lambda}(\mathbf{k})$ and $L^{\mathrm{tr}}_{\lambda}(\mathbf{k})$ are given by polynomials in $\mathbf{k}$ with nonnegative integer coefficients.
\end{conj}

By ``polynomial in $\mathbf{k}$,'' we mean: in case that $\Phi$ is simply laced, a univariate polynomial in $\mathbf{k}=k$; and in the case that $\Phi$ is not simply laced, a bivariate polynomial in~$k_l$ and~$k_s$. In light of Conjecture~\ref{conj:positivity}, we refer to $L^{\mathrm{sym}}_{\lambda}(\mathbf{k})$ and $L^{\mathrm{tr}}_{\lambda}(\mathbf{k})$ as the \emph{symmetric} and \emph{truncated Ehrhart-like polynomials}, respectively.

In~\cite{galashin2017rootfiring1} the polynomiality of $L^{\mathrm{sym}}_{\lambda}(\mathbf{k})$ was established for all root systems~$\Phi$, and the polynomiality of $L^{\mathrm{tr}}_{\lambda}(\mathbf{k})$ was established for simply laced~$\Phi$. In both cases it was shown that these polynomials have integer coefficients. But in neither case was it shown that the coefficients are nonnegative, except for special choices of $\lambda$ like $\lambda =0$ or $\lambda$ minuscule.

We will show in the next subsection that $L^{\mathrm{sym}}_{\lambda}(\mathbf{k})$ has positive coefficients by giving an explicit, combinatorial formula for these coefficients.

\subsection{Positive formula for symmetric Ehrhart-like polynomials}

The first step in our proof of the positivity of the coefficients of $L^{\mathrm{sym}}_{\lambda}(\mathbf{k})$ is to directly relate the fibers $\#(s^{\mathrm{sym}}_\mathbf{k})^{-1}(\lambda)$ to polytopes of the form $\Pi^Q(\lambda+\rho_{\mathbf{k}})$. In fact, the permutohedron non-escaping lemma already does this, as the following proposition explains:

\begin{prop}\label{prop:fibers}
For $\lambda \in P_{\geq 0}$ and good $\mathbf{k}\in\mathbb{N}[\Phi]^W$,
\[\bigcup_{w \in W^{I^{0,1}_{\lambda}}} (s^{\mathrm{sym}}_{\mathbf{k}})^{-1}(w\lambda) = \Pi^Q(\lambda+\rho_{\mathbf{k}}) \setminus \bigcup_{\substack{\mu \neq \lambda \in P_{\geq 0}, \\ \mu \leq \lambda}} \Pi^Q(\mu+\rho_{\mathbf{k}}). \]
\end{prop}
\begin{proof}
By Lemma~\ref{lem:perm_non_escaping}, no $\nu \in P$ with $\nu \notin \Pi^Q(\lambda+\rho_{\mathbf{k}})$ could possibly $\xrightarrow[\mathrm{sym},\,\mathbf{k}]{}$-stabilize to a weight in $\Pi^Q(\lambda+\rho_{\mathbf{k}})$, so certainly $\bigcup_{w \in W^{I^{0,1}_{\lambda}}} (s^{\mathrm{sym}}_{\mathbf{k}})^{-1}(w\lambda) \subseteq  \Pi^Q(\lambda+\rho_{\mathbf{k}})$. On the other hand, if~$\nu \in \Pi^Q(\mu+\rho_{\mathbf{k}})$ for some $\mu \leq \lambda \in P_{\geq 0}$ with $\nu \neq \lambda$, then, again by  Lemma~\ref{lem:perm_non_escaping}, the~$\lambda \xrightarrow[\mathrm{sym},\,\mathbf{k}]{}$-stabilization of~$\nu$ must still belong to $\Pi^Q(\mu+\rho_{\mathbf{k}})$ and so cannot be equal to $\eta_{\mathbf{k}}(w\lambda)$ for any $w\in W^{I^{0,1}_{\lambda}}$. Finally, it is an easy exercise to check (for instance using Proposition~\ref{prop:perm_containment}) that $\eta_{\mathbf{k}}(P)\cap\Pi^Q(\lambda+\rho_{\mathbf{k}}) = \eta_{\mathbf{k}}(\Pi^Q(\lambda))$. Hence, if~$\nu \in \Pi^Q(\lambda+\rho_{\mathbf{k}})$ and~$\nu \notin \Pi^Q(\mu+\rho_{\mathbf{k}})$ for any $\mu \leq \lambda \in P_{\geq 0}$ with $\mu \neq \lambda$, then, once more by Lemma~\ref{lem:perm_non_escaping}, the~$\lambda \xrightarrow[\mathrm{sym},\,\mathbf{k}]{}$-stabilization of $\nu$ cannot belong to $\Pi^Q(\mu+\rho_{\mathbf{k}})$ for any~$\mu \leq \lambda \in P_{\geq 0}$ with $\mu \neq \lambda$, but must belong to $\Pi^Q(\lambda+\rho_{\mathbf{k}})$, so the only possibility is that this stabilization is equal to $\eta_{\mathbf{k}}(w\lambda)$ for some $w \in W^{I^{0,1}_{\lambda}}$.
\end{proof}

\begin{remark} \label{remark:induction}
If $\lambda \in P_{\geq 0}$ satisfies $I^{0,1}_{\lambda} = [n]$, then Proposition~\ref{prop:fibers} says that for any good $\mathbf{k}\in\mathbb{N}[\Phi]^W$ we have
\begin{equation} \label{eqn:sym_components}
(s^{\mathrm{sym}}_{\mathbf{k}})^{-1}(\lambda) = \Pi^Q(\lambda+\rho_{\mathbf{k}}) \setminus \bigcup_{\substack{\mu \neq \lambda \in P_{\geq 0}, \\ \mu \leq \lambda}}\Pi^Q(\mu+\rho_{\mathbf{k}}).
\end{equation}
In~\cite{galashin2017rootfiring1} it is shown that for any $\lambda \in P$ with $\langle \lambda, \alpha^\vee\rangle \neq -1$ for all $\alpha \in \Phi^+$, the set~$(s^{\mathrm{sym}}_{\mathbf{k}})^{-1}(\lambda)$ belongs to the affine subspace $\lambda + \mathrm{Span}_{\mathbb{R}}(w_{\lambda}\Phi_{I^{0,1}_{\lambda}})$. Thus, for weights belonging to $(s^{\mathrm{sym}}_{\mathbf{k}})^{-1}(\lambda)$, symmetric interval-firing is the same as the corresponding process with respect to the sub-root system $w_{\lambda}\Phi_{I^{0,1}_{\lambda}}$. But $\langle \lambda,\beta^\vee \rangle \in \{0,1\}$ for any root~$\beta \in \Phi$ which is a simple root of $w_{\lambda}\Phi_{I^{0,1}_{\lambda}}$. In this way, every $(s^{\mathrm{sym}}_{\mathbf{k}})^{-1}(\lambda)$ can be written as a difference of permutohedra as in~\eqref{eqn:sym_components}, except that we may first need to project to a sub-root system. Hence by induction on the rank of our root system, to understand the $(s^{\mathrm{sym}}_{\mathbf{k}})^{-1}(\lambda)$ and $L^{\mathrm{sym}}_{\lambda}(\mathbf{k})$ for arbitrary $\lambda \in P$,  it is enough to just consider those $\lambda \in P_{\geq 0}$ with~$I^{0,1}_{\lambda} = [n]$. However, we will not invoke these kind of inductive arguments in this section because we can easily avoid them.
\end{remark}

With Proposition~\ref{prop:fibers} in hand, the strategy to understand $L^{\mathrm{sym}}_{\lambda}(\mathbf{k})$ is now just to use inclusion-exclusion on our formula for $\Pi^Q(\lambda+\rho_{\mathbf{k}})$ (Theorem~\ref{thm:perm}). The following series of propositions will prepare us for applying this inclusion-exclusion.

\begin{prop} \label{prop:intersections_perm}
Let $\lambda,\mu \in P_{\geq 0}$ with $\lambda-\mu \in Q$. Then,
\begin{itemize}
\item $\Pi(\lambda+\rho_{\mathbf{k}})\cap\Pi(\mu+\rho_{\mathbf{k}})=\Pi(\lambda\wedge\mu+\rho_{\mathbf{k}})$ for any $\mathbf{k}\in\mathbb{N}[\Phi]^W$;
\item $\Pi^Q(\lambda+\rho_{\mathbf{k}})\cap\Pi^Q(\mu+\rho_{\mathbf{k}})=\Pi^Q(\lambda\wedge\mu+\rho_{\mathbf{k}})$ for any $\mathbf{k}\in\mathbb{N}[\Phi]^W$.
\end{itemize}
\end{prop}
\begin{proof}
The first bulleted item follows from~Proposition~\ref{prop:perm_containment}. The second follows immediately from the first.
\end{proof}

\begin{prop} \label{prop:intersections_quot}
Let $\lambda,\mu \in P_{\geq 0}$ with $\lambda-\mu \in Q$. Then,
\begin{itemize}
\item $\mathrm{quot}_X(\Pi(\lambda))\cap\mathrm{quot}_X(\Pi(\mu)) = \mathrm{quot}_X(\Pi(\lambda\wedge\mu))$ for any $X\subseteq \Phi^+$;
\item $\mathrm{quot}_X(\Pi^Q(\lambda))\cap\mathrm{quot}_X(\Pi^Q(\mu)) = \mathrm{quot}_X(\Pi^Q(\lambda\wedge\mu))$ for any $X\subseteq \Phi^+$.
\end{itemize}
\end{prop}
\begin{proof}
We begin with the first bulleted item. Since by Proposition~\ref{prop:intersections_perm} we have that $\Pi(\lambda)\cap\Pi(\mu)=\Pi(\lambda\wedge\mu)$, it is clear that $ \mathrm{quot}_X(\Pi(\lambda\wedge\mu)) \subseteq \mathrm{quot}_X(\Pi(\lambda))\cap\mathrm{quot}_X(\Pi(\mu))$. Let us show the other containment. A point in $\mathrm{quot}_X(\Pi(\lambda))\cap\mathrm{quot}_X(\Pi(\mu))$ is an affine subspace of the form~$v+\mathrm{Span}_{\mathbb{R}}(X)$ for some $v\in V$ with $(v+\mathrm{Span}_{\mathbb{R}}(X))\cap\Pi(\lambda)\neq \varnothing$ and~$(v+\mathrm{Span}_{\mathbb{R}}(X))\cap\Pi(\mu)\neq \varnothing$. Let $v^0_X$ denote the unique point in $v+\mathrm{Span}_{\mathbb{R}}(X)\cap\Pi(\lambda)$ for which $\pi_X(v^0_X)=0$. As described in the beginning of the proof of Lemma~\ref{lem:slice}, the fact that $v+\mathrm{Span}_{\mathbb{R}}(X)\cap\Pi(\lambda)\neq \varnothing$ implies that $v^0_X \in \Pi(\lambda)$. Similarly, we have $v^0_X \in \Pi(\mu)$. Since, as mentioned, $\Pi(\lambda)\cap\Pi(\mu)=\Pi(\lambda\wedge\mu)$, we therefore get $v^0_X \in \Pi(\lambda \wedge \mu)$, and hence $(v+\mathrm{Span}_{\mathbb{R}}(X))\cap\Pi(\lambda \wedge \mu)\neq \varnothing$, meaning that $v+\mathrm{Span}_{\mathbb{R}}(X)\in \mathrm{quot}_X(\Pi(\lambda\wedge\mu))$.

The second bulleted item follows from the first by intersecting with $(Q+\lambda)$ and applying the integrality property of slices of permutohedra (Lemma~\ref{lem:slice}).
\end{proof}

\begin{figure}
\begin{tikzpicture}
\node[scale=\scl,draw,circle,fill=black] (1) at (0:1) {};
\node[scale=\scl,draw,circle,fill=black] (2) at (60:1) {};
\node[scale=\scl,draw,circle,fill=black] (3) at (120:1) {};
\node[scale=\scl,draw,circle,fill=black] (4) at (180:1) {};
\node[scale=\scl,draw,circle,fill=black] (5) at (240:1) {};
\node[scale=\scl,draw,circle,fill=black] (6) at (300:1) {};
\filldraw[fill=blue!25,draw=blue]  (1.center) -- (2.center) -- (3.center) -- (4.center) -- (5.center) -- (6.center)-- cycle;
\draw[dashed,thick,red] (120:2)--(300:2);
\node[scale=\scl,draw,circle,fill=black] (1) at (0:1) {};
\node[scale=\scl,draw,circle,fill=black] (2) at (60:1) {};
\node[scale=\scl,draw,circle,fill=black] (3) at (120:1) {};
\node[scale=\scl,draw,circle,fill=black] (4) at (180:1) {};
\node[scale=\scl,draw,circle,fill=black] (5) at (240:1) {};
\node[scale=\scl,draw,circle,fill=black] (6) at (300:1) {};
\draw[->,thick] (0,0) -- (3);
\draw[->,thick] (0,0) -- (1);
\node[scale=\scl,draw,circle,fill=black] (0) at (0,0) {};
\node[anchor=west] at (1) {$\alpha_2$};
\node[anchor=south west] at (2) {$\lambda=\alpha_1+\alpha_2$};
\node[anchor=south east] at (3) {$\alpha_1$};
\node[anchor=east] at (4) {$-\alpha_2$};
\node[anchor=north east] at (5) {$-\alpha_1-\alpha_2$};
\node[anchor=north west] at (6) {$-\alpha_1$};
\node[anchor = north east] at (0,0)  {$0$};
\end{tikzpicture}
\caption{Example from the proof of Proposition~\ref{prop:quots_minuscule}.} \label{fig:quots_minuscule}
\end{figure}
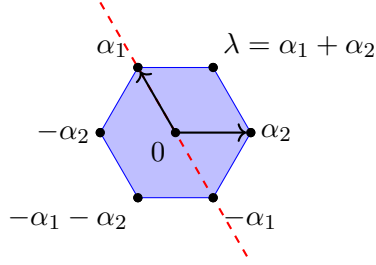

\begin{prop} \label{prop:quots_minuscule}
For $\lambda \in P_{\geq 0}$ and $X\subseteq \Phi^+$,
\[\# \left( \hspace{-0.1cm} \mathrm{quot}_X(\Pi^Q(\lambda)) \setminus   \hspace{-0.4cm} \bigcup_{\substack{\mu \neq \lambda \in P_{\geq 0}, \\ \mu \leq \lambda}}  \hspace{-0.4cm} \mathrm{quot}_X(\Pi^Q(\mu)) \hspace{-0.1cm}  \right)  = \#\left\{\mu \in W(\lambda)\colon \parbox{1.55in}{\begin{center}$\langle\mu,\alpha^\vee\rangle \in \{0,1\}$ for \\ all $\alpha \in \Phi^+\cap \mathrm{Span}_{\mathbb{R}}(X)$\end{center}}\right\} .\]
\end{prop}

\begin{proof}
In this proof we will have as a running example $\Phi=A_2$, $\lambda=\alpha_1+\alpha_2$ and $X = \{\alpha_1\}$: this is depicted in Figure~\ref{fig:quots_minuscule}, where $\Pi(\lambda)$ is shaded in blue and $\mathrm{Span}_{\mathbb{R}}(X)$ is drawn as a dashed red line. Note that $\Pi^Q(\lambda)=\{0,\pm\alpha_1,\pm\alpha_2,\pm(\alpha_1+\alpha_2)\}$.

First observe that 
\[ \mathrm{quot}_X(\Pi^Q(\lambda)) \setminus   \bigcup_{\substack{\mu \neq \lambda \in P_{\geq 0}, \\ \mu \leq \lambda}}  \mathrm{quot}_X(\Pi^Q(\mu))  = \mathrm{quot}_X(W(\lambda)) \setminus  \bigcup_{\substack{\mu \neq \lambda \in P_{\geq 0}, \\ \mu \leq \lambda}} \mathrm{quot}_X(\Pi^Q(\mu)).\]
Indeed, suppose that $\nu \in \Pi^Q(\lambda)$ and $\nu \notin W(\lambda)$. (In our running example this happens with $\nu=0$.) Then $\nu_{\mathrm{dom}}\leq \lambda$ (by Propositions~\ref{prop:perm_containment}) but~$\nu_{\mathrm{dom}} \neq \lambda$. And of course $\nu_{\mathrm{dom}} \in \Pi^Q(\lambda)$. Therefore $\mathrm{quot}_X(\nu)$ does not belong to the set we are interested in counting.

So from now on suppose $\nu \in W(\lambda)$. Let $W' \subseteq W$ denote the Weyl group of the sub-root system $\Phi' \coloneqq \Phi\cap\mathrm{Span}_{\mathbb{R}}(X)$. Let $\nu'$ be the unique element in $W'(\nu)$ for which $\pi_X(\nu')$ is a dominant weight of $\Phi'$.

First suppose that $\pi_X(\nu')$ is not zero or a minuscule weight of $\Phi'$, i.e., that we have $\langle\nu',\alpha^\vee\rangle \notin \{0,1\}$ for some $\alpha \in \Phi^+\cap \mathrm{Span}_{\mathbb{R}}(X)$. (In our running example this happens with $\nu=\alpha_1$ and $\nu=-\alpha_1$.) Then  by applying Proposition~\ref{prop:minuscule_alternate} to $\Phi'$ we get $\mathrm{ConvexHull} \, W'(\pi_X(\nu')) \cap (\mathrm{Span}_{\mathbb{Z}}(\Phi')+\pi_X(\nu')) \neq W'(\pi_X(\nu'))$. Moreover, it is clear that $\nu - \pi_X(\nu) + \mathrm{ConvexHull} \, W'(\pi_X(\nu')) \subseteq (\nu + \mathrm{Span}_{\mathbb{R}}(X)) \cap \Pi(\lambda)$. So we see in this case that there is some $\mu\in(\nu+\mathrm{Span}_{\mathbb{R}}(X))\cap\Pi^Q(\lambda)$ with $\mu_{\mathrm{dom}}\neq \lambda$, and thus (by Propositions~\ref{prop:perm_containment}) $\mathrm{quot}_X(\nu)$ does not belong to the set we are interested in counting. 

Now suppose $\pi_X(\nu')$ is zero or a minuscule weight of $\Phi'$, i.e., that $\langle\nu',\alpha^\vee\rangle \in \{0,1\}$ for all $\alpha \in \Phi^+\cap \mathrm{Span}_{\mathbb{R}}(X)$. (In our running example this happens with $\nu=\alpha_1+\alpha_2$, $\nu=-\alpha_1-\alpha_2$, $\nu=\alpha_2$, and $\nu=-\alpha_2$.) Let $\mu \in (\nu+\mathrm{Span}_{\mathbb{R}}(X))\cap \Pi^Q(\lambda)$. We claim that $\mu_{\mathrm{dom}}=\lambda$. Let $\mu'$ be the unique element of $\mathrm{Span}_{\mathbb{Z}}(\Phi') + \mu$ for which $\pi_X(\mu')$ is a zero-or-minuscule weight of $\Phi'$. Note that $\pi_X(\mu') \in Q+\pi_X(\nu')$. As explained in the proof of Proposition~\ref{prop:parabolic_subspace}, we can choose a set of simple roots $\beta_1,\ldots,\beta_\ell$ of $\Phi'$ which can be extended to a set of simple roots $\beta_1,\ldots,\beta_n$ of $\Phi$. And then by writing $\pi_X(\mu') =\pi_X(\nu') + \sum_{i=1}^{n}b_i\beta_i$ for $b_i\in\mathbb{Z}$, we see that $\pi_X(\mu') \in \mathrm{Span}_{\mathbb{Z}}(\Phi') + \pi_X(\nu')$. But since $\pi_X(\mu')$ and $\pi_X(\nu')$ are both zero-or-minuscule weights of $\Phi'$, this means (by Proposition~\ref{prop:minuscule}) that $\pi_X(\mu')=\pi_X(\nu')$ and hence that $\mu'=\nu'$. Moreover, by applying Propositions~\ref{prop:perm_containment},~\ref{prop:dominant_maximal}, and~\ref{prop:minuscule} to~$\Phi'$, we see that~$\mu' \in  \mathrm{ConvexHull} \, W'(\mu) \subseteq \Pi(\mu)$; on the other hand, $\mu\in \Pi(\lambda)=\Pi(\nu')=\Pi(\mu')$; this is only possible (again by Proposition~\ref{prop:perm_containment}) if $\mu_{\mathrm{dom}}=\lambda$. Thus in this case $\mathrm{quot}_X(\nu)$ does in fact belong to the set we are interested in counting. But we would overcount if we counted two different elements of $W(\lambda)$ which become equal after quotienting by $X$. Therefore, we only count a given $\nu$ if $\nu=\nu'$, i.e., if $\langle\nu,\alpha^\vee\rangle \in \{0,1\}$ for all~$\alpha \in \Phi^+\cap \mathrm{Span}_{\mathbb{R}}(X)$. (In our running example this happens with $\nu=\alpha_1+\alpha_2$ and $\nu=-\alpha_2$.) In this way we obtain the claimed formula.
\end{proof}

We now apply inclusion-exclusion on the formula for $\Pi^Q(\lambda+\rho_{\mathbf{k}})$ (Theorem~\ref{thm:perm}).

\begin{cor} \label{cor:main}
For $\lambda \in P_{\geq 0}$ and good $\mathbf{k}\in\mathbb{N}[\Phi]^W$,
\[\sum_{w \in W^{I^{0,1}_{\lambda}}} L^{\mathrm{sym}}_{w\lambda}(\mathbf{k})=  \hspace{-0.4cm}\sum_{\substack{X\subseteq \Phi^+ \\ \textrm{ $X$ is linearly} \\ \textrm{independent}}}  \hspace{-0.4cm}\#\left\{\mu \in W(\lambda)\colon \parbox{1.55in}{\begin{center}$\langle\mu,\alpha^\vee\rangle \in \{0,1\}$ for \\ all $\alpha \in \Phi^+\cap \mathrm{Span}_{\mathbb{R}}(X)$\end{center}}\right\} \cdot \mathrm{rVol}_Q(X) \, \mathbf{k}^{X}.\]
\end{cor}

\begin{proof}
First note that $\sum_{w \in W^{I^{0,1}_{\lambda}}} L^{\mathrm{sym}}_{w\lambda}(\mathbf{k})=\#\left(\bigcup_{w \in W^{I^{0,1}_{\lambda}}} (s^{\mathrm{sym}}_{\mathbf{k}})^{-1}(w\lambda) \right)$ because all the fibers of $s^{\mathrm{sym}}_{\mathbf{k}}$ are disjoint. Thus, by Proposition~\ref{prop:fibers} we have
\begin{equation} \label{eq:sum_sym_polys}
\sum_{w \in W^{I^{0,1}_{\lambda}}} L^{\mathrm{sym}}_{w\lambda}(\mathbf{k})=\#\left(\Pi^Q(\lambda+\rho_{\mathbf{k}}) \setminus \bigcup_{\substack{\mu \neq \lambda \in P_{\geq 0}, \\ \mu \leq \lambda}} \Pi^Q(\mu+\rho_{\mathbf{k}})\right)
\end{equation}

Let $(\mathcal{L},\leq)$ be a meet semi-lattice, and let $F\colon\mathcal{L}\to 2^S$ be a function which associates to every $p\in\mathcal{L}$ some finite subset $F(p)$ of a set $S$ such that $F(p)\cap F(q)=F(p\wedge q)$, where $\wedge$ is the meet operation of $\mathcal{L}$. Then a simple application of the M\"{o}bius inversion formula (see e.g.~\cite[\S3.7]{stanley2012ec1}) says that
\[\#\left(F(p)\setminus\bigcup_{q\leq p, q\neq p}F(q)\right) = \sum_{q\leq p}\mu_{\mathcal{L}}(q,p) \cdot \#F(q),\]
where $\mu_{\mathcal{L}}(q,p)$ is the \emph{M\"{o}bius function} of~$\mathcal{L}$. (Do not confuse this M\"{o}bius function with a weight $\mu\in P$.)

Hence, by Proposition~\ref{prop:intersections_perm} we have
\begin{equation} \label{eq:perm_minus_mobius}
\#\left(\Pi^Q(\lambda+\rho_{\mathbf{k}}) \setminus \bigcup_{\substack{\mu \neq \lambda \in P_{\geq 0}, \\ \mu \leq \lambda}} \Pi^Q(\mu+\rho_{\mathbf{k}})\right) = \sum_{\nu \leq \lambda \in P_{\geq 0}} \mu_{(P_{\geq 0},\leq)} (\nu,\lambda) \cdot \#\Pi^Q(\nu+\rho_{\mathbf{k}}),
\end{equation}
where $ \mu_{(P_{\geq 0},\leq)}$ is the M\"{o}bius function of the poset $(P_{\geq 0},\leq)$ of dominant weights with respect to root order (here we are using Stembridge's result, stated as Proposition~\ref{prop:meet_weights} above, that each connected component of this poset is a meet semi-lattice).

Then by Theorem~\ref{thm:perm}, we have
\begin{align} \label{eq:quot_mobius}
\nonumber \textrm{RHS of~\eqref{eq:perm_minus_mobius}} &= \sum_{\nu \leq \lambda \in P_{\geq 0}} \mu_{(P_{\geq 0},\leq)} (\nu,\lambda) \cdot  \left(\sum_{\substack{X\subseteq \Phi^+ \\ \textrm{ $X$ is linearly} \\ \textrm{independent}}} \# \mathrm{quot}_X(\Pi^Q(\nu))\cdot \mathrm{rVol}_Q(X) \, \mathbf{k}^{X}\right) \\
\nonumber &= \sum_{\substack{X\subseteq \Phi^+ \\ \textrm{ $X$ is linearly} \\ \textrm{independent}}} \left(  \sum_{\nu \leq \lambda \in P_{\geq 0}} \mu_{(P_{\geq 0},\leq)} (\nu,\lambda) \cdot  \#\mathrm{quot}_X(\Pi^Q(\nu)) \right) \cdot  \mathrm{rVol}_Q(X) \, \mathbf{k}^{X} \\
\nonumber &= \sum_{\substack{X\subseteq \Phi^+ \\ \textrm{ $X$ is linearly} \\ \textrm{independent}}} \# \left( \hspace{-0.1cm} \mathrm{quot}_X(\Pi^Q(\lambda)) \setminus   \hspace{-0.4cm} \bigcup_{\substack{\mu \neq \lambda \in P_{\geq 0}, \\ \mu \leq \lambda}}  \hspace{-0.4cm} \mathrm{quot}_X(\Pi^Q(\mu)) \hspace{-0.1cm}  \right)   \cdot  \mathrm{rVol}_Q(X) \, \mathbf{k}^{X}, \\
&=\sum_{\substack{X\subseteq \Phi^+ \\ \textrm{ $X$ is linearly} \\ \textrm{independent}}} \#\left\{\mu \in W(\lambda)\colon \parbox{1.55in}{\begin{center}$\langle\mu,\alpha^\vee\rangle \in \{0,1\}$ for \\ all $\alpha \in \Phi^+\cap \mathrm{Span}_{\mathbb{R}}(X)$\end{center}}\right\}   \cdot  \mathrm{rVol}_Q(X) \, \mathbf{k}^{X}
\end{align}
where in the third line we applied Proposition~\ref{prop:intersections_quot} together with M\"{o}bius inversion, and in the last line we applied Proposition~\ref{prop:quots_minuscule}. Putting together equations~\eqref{eq:sum_sym_polys},~\eqref{eq:perm_minus_mobius}, and~\eqref{eq:quot_mobius} proves the corollary.
\end{proof}

\begin{prop} \label{prop:poly_symmetry}
Let $\lambda \in P_{\geq 0}$. Let $w \in W^{I^{0,1}_{\lambda}}$. Let $\mathbf{k}\in\mathbb{N}[\Phi]^W$ be good. Then we have~$L^{\mathrm{sym}}_{w\lambda}(\mathbf{k})=L^{\mathrm{sym}}_{\lambda}(\mathbf{k})$. Consequently, $L^{\mathrm{sym}}_{w\lambda}(\mathbf{k}) = \frac{1}{[W:W_{I^{0,1}_{\lambda}}]}\sum_{w' \in W^{I^{0,1}_{\lambda}}} L^{\mathrm{sym}}_{w'\lambda}(\mathbf{k})$.
\end{prop}
\begin{proof}
This is an immediate consequence of the $W$-symmetry of the symmetric interval-firing process (Proposition~\ref{prop:symmetry}).
\end{proof}

\begin{prop} \label{prop:orbit_symmetry}
Let $\lambda \in P_{\geq 0}$ and $w \in W^{I^{0,1}_{\lambda}}$.  Then for any $X\subseteq \Phi^+$, the quantity
\[\#\left\{\mu \in wW_{I^{0,1}_{\lambda}}(\lambda)\colon \parbox{1.55in}{\begin{center}$\langle\mu,\alpha^\vee\rangle \in \{0,1\}$ for \\ all $\alpha \in \Phi^+\cap \mathrm{Span}_{\mathbb{R}}(X)$\end{center}}\right\}\]
is nonzero only if $w^{-1}(X)\subseteq \Phi_{I^{0,1}_{\lambda}}$. Consequently,
\begin{gather*}
\sum_{\substack{X\subseteq \Phi^+ \\ \textrm{ $X$ is linearly} \\ \textrm{independent}}}  \hspace{-0.4cm}\#\left\{\mu \in wW_{I^{0,1}_{\lambda}}(\lambda)\colon \parbox{1.55in}{\begin{center}$\langle\mu,\alpha^\vee\rangle \in \{0,1\}$ for \\ all $\alpha \in \Phi^+\cap \mathrm{Span}_{\mathbb{R}}(X)$\end{center}}\right\} \cdot \mathrm{rVol}_Q(X) \, \mathbf{k}^{X} \\ 
=\frac{1}{[W:W_{I^{0,1}_{\lambda}}]}\sum_{\substack{X\subseteq \Phi^+ \\ \textrm{ $X$ is linearly} \\ \textrm{independent}}}  \hspace{-0.4cm}\#\left\{\mu \in W(\lambda)\colon \parbox{1.55in}{\begin{center}$\langle\mu,\alpha^\vee\rangle \in \{0,1\}$ for \\ all $\alpha \in \Phi^+\cap \mathrm{Span}_{\mathbb{R}}(X)$\end{center}}\right\} \cdot \mathrm{rVol}_Q(X) \, \mathbf{k}^{X}.
\end{gather*}
\end{prop}
\begin{proof}
For the first claim: let $X\subseteq \Phi^+$ for which  $w^{-1}(X)\not\subseteq \Phi_{I^{0,1}_{\lambda}}$. This means there is some $\alpha \in X$ and $i \notin I^{0,1}_{\lambda}$ for which $w^{-1}(\alpha)^\vee$ has a nonzero coefficient in front of~$\alpha_i^\vee$ in its expansion in terms of simple coroots. By definition of $ I^{0,1}_{\lambda}$ we have $\langle \lambda,\alpha_i^\vee\rangle \geq 2$. Since the coefficients expressing $w^{-1}(\alpha)^\vee$ in terms of simple coroots are either all nonnegative integers or all nonpositive integers, we have~$|\langle \lambda,w^{-1}(\alpha)^\vee\rangle| \geq 2$. Thus~$|\langle \lambda,\alpha^\vee\rangle| \geq 2$, which means that the claimed quantity is indeed zero.

For the ``consequently,'' statement: observe that the prior statement implies that in the sum
\begin{gather} \label{eq:orbit_sum}
\sum_{\substack{X\subseteq \Phi^+ \\ \textrm{ $X$ is linearly} \\ \textrm{independent}}}  \hspace{-0.4cm}\#\left\{\mu \in wW_{I^{0,1}_{\lambda}}(\lambda)\colon \parbox{1.55in}{\begin{center}$\langle\mu,\alpha^\vee\rangle \in \{0,1\}$ for \\ all $\alpha \in \Phi^+\cap \mathrm{Span}_{\mathbb{R}}(X)$\end{center}}\right\} \cdot \mathrm{rVol}_Q(X) \, \mathbf{k}^{X} 
\end{gather}
we only get nonzero terms for the $X$ with $X\subseteq w(\Phi_{I^{0,1}_{\lambda}})\cap \Phi^+$. Furthermore, recall that we have $w(\Phi_{I^{0,1}_{\lambda}})\cap \Phi^+=w(\Phi_{I^{0,1}_{\lambda}}\cap \Phi^+)$ because $w \in W^{I^{0,1}_{\lambda}}$. That is to say, the expression~\eqref{eq:orbit_sum} is equal to
\begin{gather} \label{eq:orbit_sum2}
\sum_{\substack{X\subseteq w(\Phi_{I^{0,1}_{\lambda}}\cap \Phi^+) \\ \textrm{ $X$ is linearly} \\ \textrm{independent}}}  \hspace{-0.4cm}\#\left\{\mu \in wW_{I^{0,1}_{\lambda}}(\lambda)\colon \parbox{1.55in}{\begin{center}$\langle\mu,\alpha^\vee\rangle \in \{0,1\}$ for \\ all $\alpha \in \Phi^+\cap \mathrm{Span}_{\mathbb{R}}(X)$\end{center}}\right\} \cdot \mathrm{rVol}_Q(X) \, \mathbf{k}^{X} 
\end{gather}
Then making the substitutions $X\mapsto w(X)$ and $\mu \mapsto w(\mu)$ in~\eqref{eq:orbit_sum2}, and observing again that $\Phi^+\cap \mathrm{Span}_{\mathbb{R}}(w(X))=w(\Phi^+\cap \mathrm{Span}_{\mathbb{R}}(X))$ for any $X\subseteq w(\Phi_{I^{0,1}_{\lambda}}\cap \Phi^+)$ because $w \in W^{I^{0,1}_{\lambda}}$, we see that the expression~\eqref{eq:orbit_sum} is in fact equal to
\begin{gather} \label{eq:orbit_sum_indep}
\sum_{\substack{X\subseteq \Phi_{I^{0,1}_{\lambda}}\cap \Phi^+ \\ \textrm{ $X$ is linearly} \\ \textrm{independent}}}  \hspace{-0.4cm}\#\left\{\mu \in W_{I^{0,1}_{\lambda}}(\lambda)\colon \parbox{1.55in}{\begin{center}$\langle\mu,\alpha^\vee\rangle \in \{0,1\}$ for \\ all $\alpha \in \Phi^+\cap \mathrm{Span}_{\mathbb{R}}(X)$\end{center}}\right\} \cdot \mathrm{rVol}_Q(X) \, \mathbf{k}^{X}. 
\end{gather}
But note that $\{wW_{I^{0,1}_{\lambda}}(\lambda)\colon w \in W^{I^{0,1}_{\lambda}}\}$ is a partition of $W(\lambda)$ into $[W:W_{I^{0,1}_{\lambda}}]$ disjoint sets. So by summing the expression~\eqref{eq:orbit_sum} over all $w\in W^{I^{0,1}_{\lambda}}$, and observing that~\eqref{eq:orbit_sum} does not depend on the choice of $w \in W^{I^{0,1}_{\lambda}}$ since it is equal to~\eqref{eq:orbit_sum_indep}, we obtain the claimed formula.
\end{proof}

\begin{thm} \label{thm:main}
Let $\lambda \in P$. Let $\mathbf{k}\in\mathbb{N}[\Phi]^W$ be good. Then $L^{\mathrm{sym}}_{\lambda}(\mathbf{k})=0$ if $\langle \lambda,\alpha^\vee\rangle = -1$ for some positive root~$\alpha \in \Phi^+$, and otherwise
\[L^{\mathrm{sym}}_{\lambda}(\mathbf{k}) = \hspace{-0.5cm} \sum_{\substack{X\subseteq \Phi^+, \\ \textrm{$X$ is linearly} \\ \textrm{independent}}} \hspace{-0.5cm} \#\left\{\mu \in w_{\lambda}W_{I^{0,1}_{\lambda}} (\lambda_{\mathrm{dom}})\colon \parbox{1.6in}{\begin{center}$\langle\mu,\alpha^\vee\rangle \in \{0,1\}$ for \\ all $\alpha \in \Phi^+\cap \mathrm{Span}_{\mathbb{R}}(X)$\end{center}}\right\} \cdot \mathrm{rVol}_Q(X) \,\, \mathbf{k}^{X}.\]
\end{thm}
\begin{proof}
This follows immediately by combining Corollary~\ref{cor:main} with Propositions~\ref{prop:poly_symmetry} and~\ref{prop:orbit_symmetry}.
\end{proof}

\begin{remark}
If we only cared about proving the positivity of the coefficients of the symmetric Ehrhart-like polynomials, we could have actually avoided the use of the subtle integrality property of slices of permutohedra. This is because the same inclusion-exclusion strategy as above but invoking only Theorem~\ref{thm:poly_plus_zono} and the first bulleted item in Proposition~\ref{prop:intersections_quot} would yield the formula
\begin{equation}\label{eq:bad_formula}
L^{\mathrm{sym}}_{\lambda}(\mathbf{k}) \hspace{-0.1cm} = \hspace{-0.6cm} \sum_{\substack{X\subseteq \Phi^+, \\ \textrm{$X$ is linearly} \\ \textrm{independent}}} \hspace{-0.5cm} \# \left( \left( \mathrm{quot}_X(\Pi(\lambda))\setminus \hspace{-0.5cm} \bigcup_{\substack{\mu \neq \lambda \in P_{\geq 0}, \\ \mu \leq \lambda}}  \hspace{-0.5cm}\mathrm{quot}_X(\Pi(\mu)) \right) \cap \mathrm{quot}_X(Q+\lambda)  \right) \mathrm{rVol}_Q(X) \,\, \mathbf{k}^{X},
\end{equation}
for $\lambda \in P_{\geq 0}$ with $I^{0,1}_{\lambda} = [n]$. As mentioned in Remark~\ref{remark:induction}, by induction on the rank of the root system it is enough to consider $\lambda$ of this form. However, the formula in~\eqref{eq:bad_formula} is not ideal from a combinatorial perspective because the coefficients potentially involve checking every rational point in $\Pi(\lambda)$. The formula in Theorem~\ref{thm:main} is much more combinatorial, and makes clear the significance of minuscule weights.
\end{remark}

\section{Truncated Ehrhart-like polynomials and other future directions} \label{sec:future}

In this section we discuss some open questions and future directions, starting with the truncated Ehrhart-like polynomials.

\subsection{Truncated Ehrhart-like polynomials} 
One might hope that the formula for the symmetric Ehrhart-like polynomials could suggest a formula for the truncated polynomials. Indeed, in~\cite{galashin2017rootfiring1} it was shown that for any good $\mathbf{k}\in\mathbb{N}[\Phi]^W$ and any $\lambda \in P$ with~$\langle \lambda,\alpha^\vee\rangle \neq -1$ for all $\alpha\in \Phi^+$ we have
\[ (s^{\mathrm{sym}}_{\mathbf{k}})^{-1}(\lambda) = \hspace{-0.2cm} \bigcup_{\mu \in w_\lambda W_{I^{0,1}_{\lambda}} (\lambda_{\mathrm{dom}})}  \hspace{-0.2cm} (s^{\mathrm{tr}}_{\mathbf{k}})^{-1}(\mu),\]
or at the level of Ehrhart-like polynomials,
\[ L^{\mathrm{sym}}_{\lambda}(\mathbf{k}) =  \hspace{-0.2cm} \sum_{\mu \in w_\lambda W_{I^{0,1}_{\lambda}} (\lambda_{\mathrm{dom}})}  \hspace{-0.2cm} L^{\mathrm{tr}}_{\mu}(\mathbf{k}).\]
Hence, the formula in Theorem~\ref{thm:main} very naturally suggests the following conjecture:
\begin{conj} \label{conj:main}
Let $\lambda \in P$ be any weight. Then for any good $\mathbf{k}\in\mathbb{N}[\Phi]^W$ we have
\[L^{\mathrm{tr}}_{\lambda}(\mathbf{k}) = \sum_{X} \mathrm{rVol}_Q(X) \, \mathbf{k}^{X},\]
where the sum is over all $X\subseteq \Phi^+$ such that:
\begin{itemize}
\item $X$ is linearly independent;
\item $\langle \lambda,\alpha^\vee\rangle \in \{0,1\}$ for all $\alpha \in \Phi^+\cap\mathrm{Span}_{\mathbb{R}}(X)$.
\end{itemize}
\end{conj}
However, in fact Conjecture~\ref{conj:main} is {\bf false in general}! The smallest counterexample is in Type~$G_2$. Table~\ref{tab:counterexamples_1} records, for $\Phi=A_1$, $A_2$, $B_2$, $G_2$, $A_3$, $B_3$, $C_3$, $A_4$, and~$D_4$, all counterexamples to Conjecture~\ref{conj:main} among those~$\lambda \in P_{\geq 0}$ with $I^{0,1}_{\lambda}=[n]$ (recall that, as in Remark~\ref{remark:induction}, for other $\lambda$ we can understand the Ehrhart-like polynomials by projecting to a smaller sub-root system). These counterexamples were found using Sage~\cite{sagemath, Sage-Combinat}. There are counterexamples in $G_2$, $C_3$, and~$D_4$. Observe that whenever there is a counterexample $\lambda$ to Conjecture~\ref{conj:main} there must be also be another $\mu\in W(\lambda)$ which is a counterexample because we know, thanks to Theorem~\ref{thm:main}, that if we sum the left- and right-hand sides of the formula in Conjecture~\ref{conj:main} along Weyl group orbits they agree. Furthermore, although for larger root systems $\Phi$ we could not carry out an exhaustive search for counterexamples, for~$\Phi=B_4$, $C_4$, and~$A_5$ we were able to check whether the left- and right-hand sides of Conjecture~\ref{conj:main} agree when we plug in $\mathbf{k}=1$; this information is recored in Table~\ref{tab:counterexamples_2}.

\begin{table}
\renewcommand*{\arraystretch}{1.3}
\setlength\tabcolsep{3pt}
\begin{tabular}{c|c|c}
$\Phi$ & $\#\{\lambda \colon I^{0,1}_{\lambda}\hspace{-0.2cm} = [n]\}$ & Counterexamples $\lambda$ to Conjecture~\ref{conj:main} \\[-0.3cm]
& & \\ \hline
$A_1$ & $3$ & (None) \\ \hline
$A_2$ & $13$ & (None) \\ \hline
$B_2$ & $17$ & (None) \\ \hline
$G_2$ & $25$ & \begin{tabular}{r l}$\omega_1$: & \parbox{1.7in}{\begin{center}LHS $= \textcolor{red}{4k_l}+2k_s+1$ \\ RHS $=\textcolor{red}{3k_l}+2k_s+1$\end{center}} \\ \hline $-\omega_1+\omega_2$: &\parbox{1.7in}{\begin{center}LHS $= \textcolor{red}{2k_l}+k_s+1$ \\ RHS $=\textcolor{red}{3k_l}+k_s+1$\end{center}} \end{tabular} \\ \hline
$A_3$ & $75$ & (None) \\ \hline
$B_3$ & $147$ & (None) \\ \hline
$C_3$ & $147$ & \begin{tabular}{r l}$-\omega_1+\omega_2$: & \parbox{2.9in}{\begin{center}LHS $= 4k_l^2 + \textcolor{red}{14k_lk_s} + 8k_s^2 + 3k_l+5k_s+1$ \\ RHS $= 4k_l^2 + \textcolor{red}{13k_lk_s} + 8k_s^2 + 3k_l+5k_s+1$\end{center}} \\ \hline $-\omega_2+\omega_3$: &\parbox{2.9in}{\begin{center}LHS $= 2k_l^2 + \textcolor{red}{7k_lk_s} + 4k_s^s + 3k_l+4k_s+k_s+1$ \\ RHS $= 2k_l^2 + \textcolor{red}{8k_lk_s} + 4k_s^s + 3k_l+4k_s+k_s+1$\end{center}} \\ \hline \hline $\omega_2$: & \parbox{2.9in}{\begin{center}LHS $= 4k_l^2 + \textcolor{red}{14k_lk_s} + 8k_s^2 + 3k_l+5k_s+1$ \\ RHS $= 4k_l^2 + \textcolor{red}{13k_lk_s} + 8k_s^2 + 3k_l+5k_s+1$\end{center}} \\ \hline $\omega_1-\omega_2+\omega_3$: &\parbox{2.9in}{\begin{center}LHS $= 2k_l^2 + \textcolor{red}{7k_lk_s} + 4k_s^s + 3k_l+4k_s+k_s+1$ \\ RHS $= 2k_l^2 + \textcolor{red}{8k_lk_s} + 4k_s^s + 3k_l+4k_s+k_s+1$\end{center}} \end{tabular} \\ \hline 
$A_4$ & $541$ & (None) \\ \hline
$D_4$ & $865$ & \begin{tabular}{r l}$-\omega_1+\omega_2$: & \parbox{2in}{\begin{center}LHS $= \textcolor{red}{106k^3}+51k^2+11k+1$ \\ RHS $= \textcolor{red}{105k^3} + 51k^2 + 11k +1$\end{center}} \\ \hline $-\omega_2+\omega_3+\omega_4$: &\parbox{2in}{\begin{center}LHS $= \textcolor{red}{53k^3} + 39k^2 + 10k + 1$ \\ RHS $= \textcolor{red}{54k^3} + 39k^2 + 10k + 1$\end{center}} \\ \hline \hline $\omega_2$: & \parbox{2in}{\begin{center}LHS $= \textcolor{red}{106k^3}+51k^2+11k+1$ \\ RHS $= \textcolor{red}{105k^3} + 51k^2 + 11k +1$\end{center}} \\ \hline $\omega_1-\omega_2+\omega_3+\omega_4$: &\parbox{2in}{\begin{center}LHS $= \textcolor{red}{53k^3} + 39k^2 + 10k + 1$ \\ RHS $= \textcolor{red}{54k^3} + 39k^2 + 10k + 1$\end{center}} \\ \hline \hline $\omega_2-\omega_3$: & \parbox{2in}{\begin{center}LHS $= \textcolor{red}{106k^3}+51k^2+11k+1$ \\ RHS $= \textcolor{red}{105k^3} + 51k^2 + 11k +1$\end{center}} \\ \hline $\omega_1-\omega_2+\omega_4$: &\parbox{2in}{\begin{center}LHS $= \textcolor{red}{53k^3} + 39k^2 + 10k + 1$ \\ RHS $= \textcolor{red}{54k^3} + 39k^2 + 10k + 1$\end{center}} \\ \hline \hline $\omega_2-\omega_4$: & \parbox{2in}{\begin{center}LHS $= \textcolor{red}{106k^3}+51k^2+11k+1$ \\ RHS $= \textcolor{red}{105k^3} + 51k^2 + 11k +1$\end{center}} \\ \hline $\omega_1-\omega_2+\omega_3$: &\parbox{2in}{\begin{center}LHS $= \textcolor{red}{53k^3} + 39k^2 + 10k + 1$ \\ RHS $= \textcolor{red}{54k^3} + 39k^2 + 10k + 1$\end{center}} \end{tabular} \\ \hline
\end{tabular} \medskip
\caption{Counterexamples to Conjecture~\ref{conj:main} for small $\Phi$.} \label{tab:counterexamples_1}
\end{table}

\begin{table}
\renewcommand*{\arraystretch}{1.3}
\begin{tabular}{c|c|c}
$\Phi$ & $\#\{\lambda \colon I^{0,1}_{\lambda} = [n]\}$ & \parbox{3in}{\begin{center}Number of $\lambda$ for which the LHS and RHS of Conjecture~\ref{conj:main} disagree when $\mathbf{k}=1$\end{center}}  \\ \hline
$B_4$ & $1697$ & $0$ \\ \hline
$C_4$ & $1697$ & $60$ \\ \hline
$A_5$ & $4683$ & $0$ \\ \hline
\end{tabular} \medskip
\caption{Counterexamples to Conjecture~\ref{conj:main} with $\mathbf{k}=1$ for small $\Phi$.} \label{tab:counterexamples_2}
\end{table}

Another remark about the wrong formula in Conjecture~\ref{conj:main}: it at least has the symmetry that we expect. Namely, there is a copy of the abelian group $P/Q$ inside of~$W$ which acts in a natural way on (the symmetric closure of) the truncated interval-firing process. We believe that the truncated polynomials should be invariant under this action of $P/Q$ (see~\cite[\oldpapertruncatedsymremark]{galashin2017rootfiring1}). This action of $P/Q$ preserves the $\Phi^\vee$-Shi arrangement, and hence the right-hand side of the formula appearing in Conjecture~\ref{conj:main} is indeed invariant under this action of $P/Q$.

Looking at Tables~\ref{tab:counterexamples_1} and~\ref{tab:counterexamples_2}, we see that no counterexamples to Conjecture~\ref{conj:main} are known when $\Phi$ is of either Type~A or Type~B. Therefore we are prompted to ask:

\begin{question} \label{question:conjecture_ab}
Is Conjecture~\ref{conj:main} true when $\Phi$ is of either Type~A or Type~B?
\end{question}

Note that in~\cite{galashin2017rootfiring1} it was also shown that if $\Phi$ is simply laced, then for any $k\in\mathbb{Z}_{\geq 0}$ and any $\lambda \in P$ we have 
\[ (s^{\mathrm{tr}}_{k})^{-1}(\lambda) = \bigcup_{\mu \in (s^{\mathrm{tr}}_{1})^{-1}(\lambda)} (s^{\mathrm{sym}}_{k-1})^{-1}(\mu),\]
or at the level of Ehrhart-like polynomials,
\begin{equation} \label{eq:tr_decompose}
L^{\mathrm{tr}}_{\lambda}(k) = \sum_{\mu \in (s^{\mathrm{tr}}_{1})^{-1}(\lambda)} L^{\mathrm{sym}}_{\mu}(k-1).
\end{equation}
Indeed, equation~\eqref{eq:tr_decompose} is precisely how the polynomiality of $L^{\mathrm{tr}}_{\lambda}(k)$ was established. But the fact that $k-1$ appears on the right-hand side of~\eqref{eq:tr_decompose} means that it is totally unclear how to deduce positivity for the truncated polynomials from the positivity for the symmetric polynomials.

The truncated Ehrhart-like polynomials remain largely mysterious, but still seem very much worthy of further investigation.

\subsection{Lattice point formulas via tilings}

The way we obtained the formula for the symmetric Ehrhart-like polynomials was via a miraculous transfer from inclusion-exclusion at the level of polynomials to inclusion-exclusion at the level of coefficients. We could ask for a more geometric proof via tilings which better ``explains'' why these polynomials have the form they do. Let us describe what we have in mind.

For a linearly independent set $X=\{v_1,\ldots,v_m\}$ of lattice vectors $v_1,\ldots,v_m\in\mathbb{Z}^n$, a \emph{half-open parallelepiped with edge set $X$} is a convex set $\mathcal{Z}^{h.o.}_X$ of the form
\[\mathcal{Z}^{h.o.}_X \coloneqq \sum_{i=1}^{m} \begin{cases} [0,v_i) &\textrm{ if $\epsilon_i=1$}; \\ (0,v_i] &\textrm{if $\epsilon_i=-1$},\end{cases} \]
for some choice of signs $(\epsilon_1,\epsilon_2,\ldots,\epsilon_m)\in\{-1,1\}^m$.  For such a half-open parallelepiped, we always have that $\#(\mathcal{Z}^{h.o.}_X \cap \mathbb{Z}^n)=\mathrm{rVol}_{\mathbb{Z}^n}(X)$ (see e.g.~\cite[Lemma 9.8]{beck2015computing}). As mentioned in the proof of Theorem~\ref{thm:poly_plus_zono}, it is well-known that the zonotope $\mathcal{Z} \coloneqq \sum_{i=1}^{m} [0,v_i]$ can be decomposed into pieces which are (up to translation) of the form $\mathcal{Z}^{h.o.}_X$ for linearly independent subsets $X\subseteq\{v_1,\ldots,v_m\}$, with each such subset $X$ contributing exactly one piece. In fact, we can decompose the $k$th dilate $k\mathcal{Z}$ into pieces of the form $\mathcal{Z}^{h.o.}_{kX}$ in a manner consistent across all $k\in\mathbb{Z}_{\geq 0}$ (here we use the notation $kX \coloneqq \{kv\colon v\in X\}$).

Given the form of the formula in Theorem~\ref{thm:main}, we can ask whether a similar decomposition into half-open parallelepipeds exists for the symmetric Ehrhart-like polynomials. Here for $X\subseteq\Phi^+$ and $\mathbf{k}\in\mathbb{N}[\Phi]^W$ we use the notation $\mathbf{k}X\coloneqq\{\mathbf{k}(\alpha)\alpha\colon \alpha\in X\}$.

\begin{question}
For good $\mathbf{k}\in\mathbb{N}[\Phi]^W$ and $\lambda \in P_{\geq 0}$ with $I^{0,1}_{\lambda}=[n]$, can we decompose $(s^{\mathrm{sym}}_{\mathbf{k}})^{-1}(\lambda)$ into pieces which are (up to translation) of the form $\mathcal{Z}^{h.o.}_{\mathbf{k}X} \cap Q$ for linearly independent $X\subseteq\Phi^+$, with each such subset $X$ contributing 
\[\#\{\mu\in W(\lambda)\colon \langle \mu,\alpha^\vee\rangle \in \{0,1\} \textrm{ for all $\alpha\in\mathrm{Span}_{\mathbb{R}}(X)\cap\Phi^+$}\}\] 
many pieces? (Of course we also want the decomposition to be consistent across all $\mathbf{k}$.)
\end{question}

In light of Question~\ref{question:conjecture_ab} above, we can even ambitiously ask whether the same could be done for truncated Ehrhart-like polynomials in Types~A and~B.

\begin{question}
Suppose $\Phi$ is of Type~A or~B. For good $\mathbf{k}\in\mathbb{N}[\Phi]^W$ and $\lambda \in P$, can we decompose $(s^{\mathrm{tr}}_{\mathbf{k}})^{-1}(\lambda)$ into pieces which are (up to translation) of the form $\mathcal{Z}^{h.o.}_{\mathbf{k}X} \cap Q$ for linearly independent $X\subseteq\Phi^+$ which satisfy $\langle \lambda,\alpha^\vee\rangle \in \{0,1\}$ for all $\alpha \in \mathrm{Span}_{\mathbb{R}}(X)\cap\Phi^+$, with each such subset $X$ contributing exactly one piece? (Of course we also want the decomposition to be consistent across all $\mathbf{k}$.)
\end{question}

\begin{remark}
Rather than decompositions into parallelepipeds, we could also consider decompositions into simplices. An interesting decomposition of dilates of the fundamental parallelepiped of a root system~$\Phi$ into partially closed simplices is studied in~\cite{yoshinaga2018worpitzky}. This decomposition is used to resolve some Ehrhart-style questions related to the Catalan and Shi hyperplane arrangements, which as mentioned in the introduction seem to have a strong spiritual connection to our interval-firing processes.
\end{remark}

\subsection{Reciprocity for Ehrhart-like polynomials}

The Ehrhart polynomial $L_{\mathcal{P}}$ of a $d$-dimensional lattice polytope $\mathcal{P}$ satisfies a \emph{reciprocity theorem} which says that the evaluation $L_{\mathcal{P}}(-k)$ of this polynomial at a negative integer $-k$ is $(-1)^d$ times the number of lattice points in the \emph{(relative) interior} of the $k$th dilate $k\mathcal{P}$ of the polytope (see~\cite{macdonald1971polynomials} or~\cite[Theorem 4.6.9]{stanley2012ec1}). It is thus reasonable to ask if the Ehrhart-like polynomials also satisfy any kind of reciprocity. Unfortunately, we have not been able to discover anything interesting along these lines. In fact, we actually have some negative examples: Table~\ref{tab:reciprocity} shows the evaluation at $k=-1$ for some symmetric and truncated Ehrhart-like polynomials in the case $\Phi=B_3$ (of course these should be bivariate polynomials because $B_3$ is not simply laced, but we reduced to the univariate case $k_l=k_s=k$ for simplicity). There is no discernable pattern to the sign of this evaluation. In particular it does not just depend on the rank of $\Phi$ and the degree of the polynomial. We have no conjectures about reciprocity at the moment.

\begin{table}
\renewcommand*{\arraystretch}{1.5}
\begin{tabular}{c|c|c|c|c}
Weight $\lambda$ & $L_{\lambda}^{\mathrm{sym}}(k)$ & $L_{\lambda}^{\mathrm{sym}}(-1)$ & $L_{\lambda}^{\mathrm{tr}}(k)$ & $L_{\lambda}^{\mathrm{tr}}(-1)$ \\ \hline
$0$ & $87k^3 + 39k^2 + 9k + 1$ & $-56$ & $87k^3 + 39k^2 + 9k + 1$ & $-56$ \\ \hline
$\omega_1$ & $78k^2 + 36k + 6$ & $48$ &  $23k^2 + 8k + 1$ & $16$ \\ \hline
$\omega_2$ & $36k^2 + 48k + 12k$ & $0$ & $7k^2 + 6k + 1$ & $2$ \\ \hline
$\omega_3$ & $87k^3 + 108k^2 + 48k + 8$ & $-19$ & $87k^3 + 39k^2 + 9k + 1$ & $-56$ \\ \hline
$\omega_1+\omega_2$ & $12k^2 + 60k + 24$ & $-24$ &  $k^2 + 4k + 1$ & $-2$ \\ \hline
$\omega_1+\omega_3$  & $78k^2 + 84k + 24$ & $18$ & $12k^2 + 6k + 1$ & $7$ \\ \hline
$\omega_2+\omega_3$ & $36k^2 + 60k + 24$ & $0$ & $4k^2 + 4k + 1$ & $1$ \\ \hline
$\omega_1+\omega_2+\omega_3$ & $12k^2 + 72k + 48$ & $-12$ & $k^2 + 3k + 1$ & $-1$ \\ \hline
\end{tabular}
\medskip
\caption{Evaluation of symmetric and truncated polynomials at $-1$ for $\Phi=B_3$.} \label{tab:reciprocity}
\end{table}

\subsection{\texorpdfstring{$h^*$}{h star}-polynomials}
The $h^*$-polynomial $h^*_{\mathcal{P}}(z)$ of a $d$-dimensional lattice polytope $\mathcal{P}$ is defined by 
\[ \sum_{k \geq 0} L_{\mathcal{P}}(k)z^k = \frac{h^*_{\mathcal{P}}(z)}{(1-z)^{d+1}}. \]
By a celebrated result of Stanley~\cite{stanley1980decompositions}, the coefficients of $h^*_{\mathcal{P}}(z)$ are nonnegative. Hence, one might wonder whether something similar holds for the Ehrhart-like polynomials. Unfortunately, there are small counterexamples. Let $\Phi=A_3$ and $\lambda \coloneqq \omega_1+\omega_2+\omega_3$. Then $L_{\lambda}^{\mathrm{sym}}(k) = 6k^2+36k+24$ and
\[ \sum_{k \geq 0} (6k^2+36k+24) \, z^k = \frac{-6z^2-6z+24}{(1-z)^{3}}. \]
Similarly, $L_{\lambda}^{\mathrm{tr}}(k) = k^2+3k+1$ and
\[ \sum_{k \geq 0} (k^2+3k+1) \, z^k = \frac{-z^2+2z+1}{(1-z)^{3}}. \]
We have no conjectures about $h^*$-polynomials at the moment.

\subsection{Uniform proof of integrality property} 

We would obviously prefer to have a completely uniform proof of the subtle integrality property of slices of permutohedra (Lemma~\ref{lem:slice}). A uniform explanation of the projection-dilation property of root polytopes (Lemma~\ref{lem:root_projection}) would yield such a proof. Another place to start looking for such a proof might be representation theory. The most direct connection to representation theory is the following: let $\mathfrak{g}$ be a simple Lie algebra over the complex numbers whose corresponding root system is $\Phi$; then for a dominant weight $\lambda\in P_{\geq 0}$, the discrete permutohedron $\Pi^Q(\lambda)$ consists of those weights~$\mu \in P$ appearing with nonzero multiplicity in the irreducible representation~$V^{\lambda}$ of $\mathfrak{g}$ with highest weight $\lambda$ (see e.g.~\cite{stembridge1998partial}). Hence it is not unreasonable to think that the slices of permutohedra, and the corresponding integrality property, could have some representation-theoretic meaning.

\subsection{Slices of permutohedra}

It would be interesting to further study the slices of permutohedra which appear in the integrality lemma (Lemma~\ref{lem:slice}). By slices we mean polytopes of the form $(\mu+\mathrm{Span}_{\mathbb{R}}(X))\cap \Pi(\lambda)$ for $\lambda \in P_{\geq 0}$, $\mu \in Q+\lambda$, and $X\subseteq \Phi^+$. Note that these polytopes are certainly invariant under the Weyl group~$W'\subseteq W$ of the sub-root system $\mathrm{Span}_{\mathbb{R}}(X)\cap \Phi$, but they need not be $W'$-permutohedra. As mentioned in Section~\ref{subsec:type_a}, the \emph{quotients} $\mathrm{quot}_X(\Pi(\lambda))$ of permutohedra are generalized permutohedra (at least in Type~A; but we believe the corresponding statement should be true in all types). However, these slices $(\mu+\mathrm{Span}_{\mathbb{R}}(X))\cap \Pi(\lambda)$ need not even be generalized permutohedra: for instance, in $\Phi=A_3$ the slices can have more vertices than would be possible for a generalized permutohedron.

In general these slices do not necessarily have ``integer vertices;'' that is to say, their vertices do not necessarily belong to $Q+\lambda$, as can be seen in Figure~\ref{fig:slice}. But let us quickly explain how in Type~A these slices actually \emph{do} have vertices in $Q+\lambda$. So assume that~$\Phi=A_n$. First observe that any vertex of~$(\mu+\mathrm{Span}_{\mathbb{R}}(X))\cap \Pi(\lambda)$  is the intersection of $\mu+\mathrm{Span}_{\mathbb{R}}(X)$ with a face $F$ of $\Pi(\lambda)$ of codimension $\ell \coloneqq \mathrm{dim}(\mathrm{Span}_{\mathbb{R}}(X))$. Thus we can find $\beta_1,\ldots,\beta_\ell \in \Phi$ which span $\mathrm{Span}_{\mathbb{R}}(X)$ and $\beta_{\ell+1},\ldots,\beta_n \in \Phi$ which affinely span~$F$ such that $\beta_1,\ldots,\beta_n$ is a basis of $V$. But then because the collection of roots in Type A is totally unimodular, any subset of roots which forms a basis of $V$ generates the root lattice~$Q$. This means that we have~$\mu = \lambda + \sum_{i=1}^{n}b_i\beta_i$ for $b_i \in \mathbb{Z}$. By acting by the Weyl group we are free to assume that $F$ contains $\lambda$. Hence the unique point in~$(\mu+\mathrm{Span}_{\mathbb{R}}(X))\cap F$ is~$\lambda+\sum_{i=\ell+1}^{n}b_i\beta_i$, which indeed belongs to $Q+\lambda$ because the~$b_i$ are integers. This argument gives a much simpler proof of  Lemma~\ref{lem:slice} for Type~A, but we have been unable to adapt it to other types where the collection of roots is not totally unimodular.

\subsection{The projection-dilation constant for centrally symmetric sets}

Let $\mathcal{S}\subseteq V$ be a centrally symmetric, bounded set. Then we can define the \emph{projection-dilation constant} of $\mathcal{S}$ to be the be minimal $\kappa \geq 1$ such that for all nonzero subspaces $\{0\}\neq U \subseteq V$ spanned by elements of $\mathcal{S}$, we have $\pi_U(\mathrm{ConvexHull}(\mathcal{S})) \subseteq \kappa\cdot \mathrm{ConvexHull}(\mathcal{S}\cap U)$. Lemma~\ref{lem:root_projection} asserts that the projection-dilation constant of $\Phi^\vee$ is strictly less than $2$. It might be interesting to study this quantity for other choices of~$\mathcal{S}$. For example, if we let $\mathcal{S} \coloneqq \{(\pm 1,\pm 1,\ldots, \pm 1)\} \subseteq \mathbb{R}^n$ be the set of vertices of the standard $n$-hypercube, then it can be shown that the projection-dilation constant of $\mathcal{S}$ is at least on the order of $\sqrt{n}$: to see this we can take $U$ to be the orthogonal complement to the vector $(m,1,1,\ldots,1)$ where~$m$ is an integer close to $\sqrt{n}$. More generally, we might consider $\mathcal{S}_{n,k}$, the set of vectors in $\mathbb{R}^n$ with $k$ nonzero entries which are all $\pm 1$. So $\mathcal{S}_{n,n}$ is the set of vertices of the $n$-hypercube, and has unbounded projection-dilation constant. On the other hand, $\mathcal{S}_{n,2}$ is the root system $\Phi=D_n$, and has projection-dilation constant bounded by $2$.

\appendix

\section{The projection-dilation property of root polytopes} \label{sec:root_polytopes}

In this appendix, we prove Lemma~\ref{lem:root_projection}, concerning dilations of projections of the root polytope $\mathcal{P}_{\Phi^\vee}\coloneqq \mathrm{ConvexHull}(\Phi^\vee)$. Unfortunately our proof will consist of a case-by-case analysis, invoking the classification of root systems. Hence, we now go over this classification.

The \emph{Cartan matrix} of $\Phi$ is the $n\times n$ matrix $\mathbf{C} = (\mathbf{C}_{ij})$ with entries $\mathbf{C}_{ij} = \langle \alpha_i,\alpha_j^\vee\rangle$ for all $i,j\in[n]$. In other words, the rows of the Cartan matrix are the coefficients expressing the simple roots in the basis of fundamental weights. The matrix~$\mathbf{C}$ determines~$\Phi$ up to isomorphism. The \emph{Dynkin diagram} of $\Phi$ is another way of encoding the same information as the Cartan matrix in the form of a decorated graph. The Dynkin diagram, which has $[n]$ as its set of nodes, is obtained as follows: first for all $i \neq j\in[n]$ we draw $\langle\alpha_i,\alpha_j^\vee\rangle\langle\alpha_j,\alpha_i^\vee\rangle$ edges between~$i$ and~$j$; then, if $\langle\alpha_i,\alpha_j^\vee\rangle \neq \langle\alpha_j,\alpha_i^\vee\rangle$ for some $i\neq j\in[n]$ we draw an arrow on top of the edges between $i$ and $j$, with the arrow going from $i$ to $j$ if~$\langle\alpha_i,\alpha_i\rangle \geq \langle\alpha_j,\alpha_j\rangle$. That~$\Phi$ is irreducible is equivalent to this Dynkin diagram being connected. There are no arrows in the Dynkin diagram of $\Phi$ if and only if $\Phi$ is simply laced. Figure~\ref{fig:dynkin_classification} depicts the Dynkin diagrams for all the irreducible root systems: these are the \emph{classical types} $A_n$ for $n\geq 1$, $B_n$ for~$n\geq 2$, $C_n$ for $n\geq 3$, $D_n$ for~$n\geq 4$, as well as the \emph{exceptional types} $G_2$, $F_4$, $E_6$, $E_7$, and $E_8$. The subscript in the name of the type denotes the number of nodes of the Dynkin diagram, i.e., the number of simple roots. Our numbering of the simple roots here is consistent with Bourbaki~\cite{bourbaki2002lie}. The \emph{dual Dynkin diagram} of $\Phi$, i.e., the Dynkin diagram of $\Phi^\vee$, is obtained from that of $\Phi$ by reversing the direction of the arrows.

\begin{figure}
\begin{tikzpicture}
\node (A) at (0,0) {\begin{tikzpicture}
\node[scale=\scl,draw,circle,fill=black] (1) at (-1,0) {};
\draw (-1,0) circle (0.1);
\node[scale=\scl,draw,circle,fill=black] (2) at (0,0) {};
\draw (0,0) circle (0.1);
\node[scale=\scl,draw,circle,fill=black] (3) at (1,0) {};
\draw (1,0) circle (0.1);
\node[scale=\scl,draw,circle,fill=black] (4) at (2,0) {};
\draw (2,0) circle (0.1);
\node[anchor=north] at (1.south) {$1$};
\node[anchor=north] at (2.south) {$2$};
\node[anchor=north] at (3.south) {${n-1}$};
\node[anchor=north] at (4.south) {$n$};
\draw[thick] (1)--(2);
\draw[thick,dashed] (2)--(3);
\draw[thick] (3)--(4);
\end{tikzpicture}};
\node at (A.south) {$A_n$};
\node (B) at (0,-1.5) {\begin{tikzpicture}[decoration={markings,mark=at position 0.7 with {\arrow{>}}}]
\node[scale=\scl,draw,circle] (1) at (-1,0) {};
\node[scale=\scl,draw,circle] (2) at (0,0) {};
\node[scale=\scl,draw,circle] (3) at (1,0) {};
\node[scale=\scl,draw,circle,fill=black] (4) at (2,0) {};
\draw (2,0) circle (0.1);
\node[anchor=north] at (1.south) {$1$};
\node[anchor=north] at (2.south) {$2$};
\node[anchor=north] at (3.south) {${n-1}$};
\node[anchor=north] at (4.south) {$n$};
\draw[thick] (1)--(2);
\draw[thick,dashed] (2)--(3);
\draw[thick,double,postaction={decorate}] (3)--(4);
\end{tikzpicture}};
\node at (B.south) {$B_n$};
\node (C) at (0,-3) {\begin{tikzpicture}[decoration={markings,mark=at position 0.7 with {\arrow{>}}}]
\node[scale=\scl,draw,circle,fill=black] (1) at (-1,0) {};
\node[scale=\scl,draw,circle] (2) at (0,0) {};
\node[scale=\scl,draw,circle] (3) at (1,0) {};
\node[scale=\scl,draw,circle] (4) at (2,0) {};
\node[anchor=north] at (1.south) {$1$};
\node[anchor=north] at (2.south) {$2$};
\node[anchor=north] at (3.south) {${n-1}$};
\node[anchor=north] at (4.south) {$n$};
\draw (-1,0) circle (0.1);
\draw (2,0) circle (0.1);
\draw[thick] (1)--(2);
\draw[thick,dashed] (2)--(3);
\draw[thick,double,postaction={decorate}] (4)--(3);
\end{tikzpicture}};
\node at (C.south) {$C_n$};
\node (D) at (0,-4.5) {\begin{tikzpicture}
\node[scale=\scl,draw,circle,fill=black] (1) at (-1,0) {};
\node[scale=\scl,draw,circle] (2) at (0,0) {};
\node[scale=\scl,draw,circle] (3) at (1,0) {};
\node[scale=\scl,draw,circle,fill=black] (4) at (2,-0.3) {};
\node[scale=\scl,draw,circle,fill=black] (5) at (2,0.3) {};
\node[anchor=north] at (1.south) {$1$};
\node[anchor=north] at (2.south) {$2$};
\node[anchor=north] at (3.south) {${n-2}$};
\node[anchor=north] at (4.south) {${n-1}$};
\node[anchor=south] at (5.north) {$n$};
\draw (-1,0) circle (0.1);
\draw (2,-0.3) circle (0.1);
\draw (2,0.3) circle (0.1);
\draw[thick] (1)--(2);
\draw[thick,dashed] (2)--(3);
\draw[thick] (3)--(4);
\draw[thick] (3)--(5);
\end{tikzpicture}};
\node at (D.south) {$D_n$};
\node (G) at (6,1) {\begin{tikzpicture}[decoration={markings,mark=at position 0.7 with {\arrow{>}}}]
\node[scale=\scl,draw,circle] (1) at (-1,0) {};
\node[scale=\scl,draw,circle] (2) at (0,0) {};
\draw (-1,0) circle (0.1);
\node[anchor=north] at (1.south) {$1$};
\node[anchor=north] at (2.south) {$2$};
\draw[thick,double distance=2pt,postaction={decorate}] (2)--(1);
\draw[thick] (1)--(2);
\end{tikzpicture}};
\node at (G.south) {$G_2$};
\node (F) at (6,-0.5) {\begin{tikzpicture}[decoration={markings,mark=at position 0.7 with {\arrow{>}}}]
\node[scale=\scl,draw,circle] (1) at (-1,0) {};
\node[scale=\scl,draw,circle] (2) at (0,0) {};
\node[scale=\scl,draw,circle] (3) at (1,0) {};
\node[scale=\scl,draw,circle] (4) at (2,0) {};
\draw (2,0) circle (0.1);
\node[anchor=north] at (1.south) {$1$};
\node[anchor=north] at (2.south) {$2$};
\node[anchor=north] at (3.south) {$3$};
\node[anchor=north] at (4.south) {$4$};
\draw[thick] (1)--(2);
\draw[thick,double,postaction={decorate}] (2)--(3);
\draw[thick] (3)--(4);
\end{tikzpicture}};
\node at (F.south) {$F_4$};
\node (E6) at (6,-2) {\begin{tikzpicture}
\node[scale=\scl,draw,circle,fill=black] (1) at (-1,0) {};
\node[scale=\scl,draw,circle] (2) at (1,0.5) {};
\node[scale=\scl,draw,circle] (3) at (0,0) {};
\node[scale=\scl,draw,circle] (4) at (1,0) {};
\node[scale=\scl,draw,circle] (5) at (2,0) {};
\node[scale=\scl,draw,circle,fill=black] (6) at (3,0) {};
\draw (-1,0) circle (0.1);
\draw (3,0) circle (0.1);
\node[anchor=north] at (1.south) {$1$};
\node[anchor=east] at (2.west) {$2$};
\node[anchor=north] at (3.south) {$3$};
\node[anchor=north] at (4.south) {$4$};
\node[anchor=north] at (5.south) {$5$};
\node[anchor=north] at (6.south) {$6$};
\draw[thick] (1)--(3)--(4)--(5)--(6);
\draw[thick] (2)--(4);
\end{tikzpicture}};
\node at (E6.south) {$E_6$};
\node (E7) at (6,-3.5) {\begin{tikzpicture}
\node[scale=\scl,draw,circle] (1) at (-1,0) {};
\node[scale=\scl,draw,circle] (2) at (1,0.5) {};
\node[scale=\scl,draw,circle] (3) at (0,0) {};
\node[scale=\scl,draw,circle] (4) at (1,0) {};
\node[scale=\scl,draw,circle] (5) at (2,0) {};
\node[scale=\scl,draw,circle] (6) at (3,0) {};
\node[scale=\scl,draw,circle,fill=black] (7) at (4,0) {};
\draw (1,0.5) circle (0.1);
\draw (4,0) circle (0.1);
\node[anchor=north] at (1.south) {$1$};
\node[anchor=east] at (2.west) {$2$};
\node[anchor=north] at (3.south) {$3$};
\node[anchor=north] at (4.south) {$4$};
\node[anchor=north] at (5.south) {$5$};
\node[anchor=north] at (6.south) {$6$};
\node[anchor=north] at (7.south) {$7$};
\draw[thick] (1)--(3)--(4)--(5)--(6)--(7);
\draw[thick] (2)--(4);
\end{tikzpicture}};
\node at (E7.south) {$E_7$};
\node (E8) at (6,-5) {\begin{tikzpicture}
\node[scale=\scl,draw,circle] (1) at (-1,0) {};
\node[scale=\scl,draw,circle] (2) at (1,0.5) {};
\node[scale=\scl,draw,circle] (3) at (0,0) {};
\node[scale=\scl,draw,circle] (4) at (1,0) {};
\node[scale=\scl,draw,circle] (5) at (2,0) {};
\node[scale=\scl,draw,circle] (6) at (3,0) {};
\node[scale=\scl,draw,circle] (7) at (4,0) {};
\node[scale=\scl,draw,circle] (8) at (5,0) {};
\draw (1,0.5) circle (0.1);
\draw (-1,0) circle (0.1);
\node[anchor=north] at (1.south) {$1$};
\node[anchor=east] at (2.west) {$2$};
\node[anchor=north] at (3.south) {$3$};
\node[anchor=north] at (4.south) {$4$};
\node[anchor=north] at (5.south) {$5$};
\node[anchor=north] at (6.south) {$6$};
\node[anchor=north] at (7.south) {$7$};
\node[anchor=north] at (8.south) {$8$};
\draw[thick] (1)--(3)--(4)--(5)--(6)--(7)--(8);
\draw[thick] (2)--(4);
\end{tikzpicture}};
\node at (E8.south) {$E_8$};
\end{tikzpicture}
\caption{Dynkin diagrams of all irreducible root systems. The nodes corresponding to minuscule weights are filled in black. The nodes whose removal does not disconnect the dual extended Dynkin diagram are circled.}  \label{fig:dynkin_classification}
\end{figure}
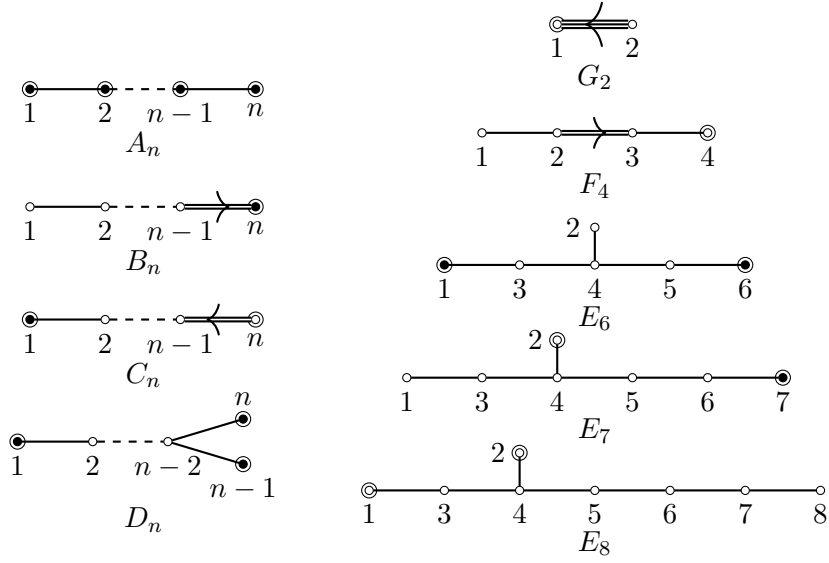

One can observe that in Figure~\ref{fig:dynkin_classification} we added some extra decorations to the Dynkin diagrams: we filled in black some nodes and we circled some nodes. Let us now explain what these extra decorations mean. In an irreducible root system $\Phi$, there exists a unique positive root that is maximal with respect to root order. We denote this root by $\theta \in \Phi$ and we call it the \emph{highest root}. We use $\widehat{\theta}$ to denote the root of $\Phi$ for which~$\widehat{\theta}^\vee$ is the highest root of~$\Phi^\vee$ ($\widehat{\theta}$ is called the \emph{highest short root}). The existence of a highest root implies that the minuscule weights of $\Phi$ are a subset of the fundamental weights. In Figure~\ref{fig:dynkin_classification}, we filled in black the nodes $i\in[n]$ of the Dynkin diagram for which $\omega_i$ is a minuscule weight.  And in Figure~\ref{fig:dynkin_classification}, we circled the nodes $i\in[n]$ of the Dynkin diagram for which $\{\alpha_1^\vee,\alpha_2^\vee,\ldots,\overline{\alpha_i^\vee},\ldots,\alpha_n^\vee,-\widehat{\theta}^\vee\}$ (where the overline denotes omission) forms the set of simple roots of an \emph{irreducible} root system. We also refer to these nodes as ``the nodes whose removal does not disconnect the dual extended Dynkin diagram,'' although we will not explain precisely what an extended Dynkin diagram is. Note that the nodes corresponding to minuscule weights are always a subset of this collection (this can be established uniformly: e.g., there is a copy of the abelian group~$P/Q$ inside of~$W$ which transitively permutes the set $\{\alpha^\vee_i\colon\omega_i \textrm{ is minuscule}\}\cup\{-\widehat{\theta}^\vee\}$, as described in~\cite{lam2012alcoved}). 

By now we have seen the significance of the minuscule weights. What is the significance of the nodes whose removal does not disconnect the dual extended Dynkin diagram? They arise in a description of the facets of the root polytope due to Cellini and Marietti~\cite{cellini2015root}.

\begin{thm}[{Cellini and Marietti~\cite{cellini2015root}}] \label{thm:root_polytope_facets}
We have
\[\mathcal{P}_{\Phi^\vee} = \left\{v\in V\colon \langle v,w(\omega_i)\rangle \leq a_i, \parbox{2.5in}{ \begin{center} for all $w\in W$, and all nodes $i\in[n]$ \\ of the Dynkin diagram of $\Phi$  \\ whose removal does not disconnect \\ the dual extended Dynkin diagram\end{center}}\right\}, \]
where the coefficients $a_i$ are determined by writing $\widehat{\theta}^\vee = a_1\alpha^\vee_1+a_2\alpha^\vee_2+\cdots+\alpha^\vee_n$.
\end{thm}

Note that $\omega_i$ is a minuscule weight of $\Phi$ if and only if the corresponding coefficient~$a_i$ in the expansion $\widehat{\theta}^\vee = a_1\alpha^\vee_1+a_2\alpha^\vee_2+\cdots+\alpha^\vee_n$ satisfies $a_i=1$.

Now let us return to our discussion of Lemma~\ref{lem:root_projection}. Lemma~\ref{lem:root_projection} asserts that for any nonzero subspace $\{0\}\neq U\subseteq V$ spanned by a subset of~$\Phi^\vee$, there is some $1\leq \kappa <2$ so that $\pi_U(\mathcal{P}_{\Phi^\vee}) \subseteq \kappa\cdot \mathcal{P}_{\Phi_U^\vee}$ (where we recall the notation $\Phi^\vee_U\coloneqq \Phi^\vee\cap U$). Recall that in Example~\ref{ex:d4} we gave an example showing what these projections look like, and explaining why this assertion is nontrivial.

\begin{table}
\renewcommand*{\arraystretch}{2.5}
\begin{tabular}{c|c}
$\Phi$ & Dominant vertices of $\mathcal{P}^{*}_{\Phi^\vee}$ \\ \hline
$A_n$ & $\omega_j = \sum_{i=1}^{n} \min \left(\frac{i(n+1-j)}{n+1},\frac{j(n+1-i)}{n+1}\right) \alpha_i$, for all $j\in [n]$. \\ \hline
$B_n$ & $\omega_n = \sum_{i=1}^{n} \frac{i}{2} \alpha_i$. \\ \hline
$C_n$ & \parbox{4in}{\begin{center}  $\omega_1 = \left(\sum_{i=1}^{n-1} \alpha_i\right) + \frac{1}{2} \alpha_n$;  \\  $\frac{1}{2}\omega_n = \left(\sum_{i=1}^{n-1} \frac{i}{2}\alpha_i\right) + \frac{n}{4} \alpha_n$.\end{center}} \\ \hline
$D_n$ & \parbox{4in}{\begin{center} $\omega_1 =  \left( \sum_{i=1}^{n-2} \alpha_i \right) + \frac{1}{2}\alpha_{n-1}+ \frac{1}{2}\alpha_n$; \\ $\omega_{n-1} =  \left( \sum_{i=1}^{n-2} \frac{i}{2} \alpha_i \right) + \frac{n}{4}\alpha_{n-1}+ \frac{n-2}{4}\alpha_n$;  \\ $\omega_{n} =  \left( \sum_{i=1}^{n-2} \frac{i}{2} \alpha_i \right) + \frac{n-2}{4}\alpha_{n-1}+ \frac{n}{4}\alpha_n$.\end{center}} \\ \hline
\end{tabular}
\medskip
\caption{Formulas for the simple root coordinates of the dominant vertices of~$\mathcal{P}^{*}_{\Phi^\vee}$.} \label{tab:minuscule_coeffs}
\end{table}

We find it more convenient to work with the polar dual of the root polytope rather than the root polytope itself: this polar dual is $\mathcal{P}^{*}_{\Phi^\vee}=\{v\in V\colon \langle v,\alpha^\vee \rangle \leq 1\}$. Theorem~\ref{thm:root_polytope_facets} says that the vertices of~$\mathcal{P}^{*}_{\Phi^\vee}$ lying in $P_{\geq 0}^{\mathbb{R}}$ are $\frac{1}{a_i}\omega_i$ for nodes $i\in [n]$ whose removal does not disconnect the dual extended Dynkin diagram. Let us refer to these as the \emph{dominant vertices} of $\mathcal{P}^{*}_{\Phi^\vee}$. The minuscule weights of $\Phi$ are always dominant vertices of $\mathcal{P}^{*}_{\Phi^\vee}$, but there may be more dominant vertices than this (although there are not too many more, as can be seen in Figure~\ref{fig:dynkin_classification}). Formulas for the simple root coordinates of the dominant vertices of~$\mathcal{P}^{*}_{\Phi^\vee}$ for the classical types are recorded in Table~\ref{tab:minuscule_coeffs}: these can be obtained by computing the inverse of a Cartan matrix; see, e.g.~\cite[\S13, Table~1]{humphreys1972lie}. We have so far assumed that $\Phi$ is irreducible; but note that if $\Phi$ is reducible with $\Phi = \Phi' \oplus \Phi''$ then $\mathcal{P}^{*}_{\Phi^\vee} = \mathcal{P}^{*}_{(\Phi')^\vee} \times \mathcal{P}^{*}_{(\Phi'')^\vee}$ (and so the dominant vertices of $\mathcal{P}^{*}_{\Phi^\vee}$ are then the sums of the dominant vertices of~$\mathcal{P}^{*}_{(\Phi')^\vee}$ and of~$\mathcal{P}^{*}_{(\Phi'')^\vee}$).

Lemma~\ref{lem:root_projection} amounts to the assertion for any nonzero subspace $\{0\}\neq U\subseteq V$ spanned by a subset of~$\Phi^\vee$, we have $\langle v,\alpha^\vee \rangle < 2$ for all vertices $v$ of~$\mathcal{P}^{*}_{\Phi_U^\vee}$ and all~$\alpha^\vee \in \Phi^\vee$. First let us observe that we can reduce to the case where $\mathrm{dim}(U)=  \mathrm{dim}(V)-1$. Indeed, note that if there were $U\subseteq V$ and $\alpha^\vee \in \Phi^\vee$ which provided a counterexample, then this counterexample would also occur for the root system~$\Phi \cap (U+\mathrm{Span}_{\mathbb{R}}\{\alpha\})$. Moreover, by Proposition~\ref{prop:parabolic_subspace}, we can also assume that $U$ is a parabolic subspace. Thus, we may assume that~$U$ is a \emph{maximal parabolic subspace} of $\Phi$, i.e., a subspace spanned by all but one of the simple roots. For $i\in[n]$, we use the notation $\Phi_i \coloneqq \Phi_{[n]\setminus \{i\}}$ and~$W_i \coloneqq W_{[n]\setminus \{i\}}$.

So to prove Lemma~\ref{lem:root_projection} we need only to show that for all $i\in[n]$, we have $\langle v,\alpha^\vee \rangle < 2$ for all vertices $v$ of~$\mathcal{P}^{*}_{\Phi_i^\vee}$ and all~$\alpha^\vee \in \Phi^\vee$. By the $W_i$-invariance of $\mathcal{P}^{*}_{\Phi_i^\vee}$, it is enough to prove this for $\Phi^\vee_i$-dominant vertices $v$ of~$\mathcal{P}^{*}_{\Phi_i^\vee}$, and for $\alpha^\vee \in \Phi^\vee$ which are $\Phi^\vee_i$-dominant in the sense that $\langle \alpha_j,\alpha^\vee\rangle \geq0$ for all $j\neq i\in [n]$. Furthermore, clearly we need only consider $\alpha^\vee \in \Phi^\vee\setminus \Phi^\vee_i$. We have a good understanding of the $\Phi^\vee_i$-dominant vertices of~$\mathcal{P}^{*}_{\Phi_i^\vee}$ thanks to Theorem~\ref{thm:root_polytope_facets}. To understand the $\Phi^\vee_i$-dominant coroots $\alpha^\vee \in \Phi^\vee\setminus \Phi_i^\vee$, we can appeal to Oshima's lemma:

\begin{lemma}[{``Oshima's lemma''~\cite{oshima2006classification, dyer2018parabolic}}] \label{lem:oshima}
For any $i \in [n]$, any $c\in\mathbb{Z}$, and any $\ell \in \mathbb{R}_{\geq 0}$, there is at most one coroot~$\alpha^\vee \in \Phi^\vee\setminus \Phi_i^\vee$ with $\alpha^\vee_i$ coordinate $c$ (in the basis of simple coroots) and $\langle \alpha^\vee,\alpha^\vee\rangle = \ell$ which is $\Phi^\vee_i$-dominant.
\end{lemma}

Actually Oshima's lemma is more general in that it applies to all parabolic subgroups, not just maximal parabolic subgroups. But Lemma~\ref{lem:oshima} is all we will need. Using Lemma~\ref{lem:oshima} (together with knowledge of the simple root coordinates of the highest roots as recorded for instance in~\cite[\S12, Table~2]{humphreys1972lie}) it is easy to determine all the $\Phi^\vee_i$-dominant representatives for $W_i$-orbtis of~$\Phi^\vee \setminus \Phi^\vee_i$. For the classical types, this information is recorded in Table~\ref{tab:parabolic_orbit_reps}. 

By combining the information contained in Tables~\ref{tab:minuscule_coeffs} and~\ref{tab:parabolic_orbit_reps} (together with knowledge of what the Dynkin diagram of $\Phi$ looks like), we can compute for any $i\in[n]$ the maximum value of~$\langle v,\alpha^\vee \rangle$ over $v\in\mathcal{P}^{*}_{\Phi_i^\vee}$ and~$\alpha^\vee \in \Phi^\vee\setminus \Phi^\vee_i$ for the classical types: this information is recorded in Table~\ref{tab:classical_maxes}. For the exceptional types, we have computed these maximums via computer: this information is recorded in Table~\ref{tab:exceptional_maxes}. Sage~\cite{sagemath, Sage-Combinat} code used to generate Table~\ref{tab:exceptional_maxes} is available upon request from the first author. 

An inspection of Table~\ref{tab:classical_maxes} and Table~\ref{tab:exceptional_maxes} completes the verification of Lemma~\ref{lem:root_projection}.

\begin{table}
\renewcommand*{\arraystretch}{1.8}
\begin{tabular}{c|c|c}
$\Phi$ & $i$ & $\Phi^\vee_i$-dominant representatives for $W_i$-orbtis of $\Phi^\vee \setminus \Phi^\vee_i$ \\ \hline
$A_n$ & $i=1,\ldots,n$ & \parbox{4in}{\begin{center} $\alpha^\vee_1+\alpha^\vee_2+\cdots+\alpha^\vee_n $; \\ 
$-\alpha^\vee_i$.
\end{center}} \\ \hline \hline
$B_n$ & $i=1$ & \parbox{4in}{\begin{center} $2\alpha^\vee_1+2\alpha^\vee_2+\cdots+2\alpha^\vee_{n-1}+\alpha^\vee_n$ (long); \\ 
$-2\alpha^\vee_1-2\alpha^\vee_2-\cdots-2\alpha^\vee_{n-1}-\alpha^\vee_n$ (long); \\ 
$\alpha^\vee_1+2\alpha^\vee_2+2\alpha^\vee_3+\cdots+2\alpha^\vee_{n-1}+\alpha^\vee_n$ (short); \\ 
$-\alpha^\vee_1$ (short).
 \end{center}}  \\ \hline
$B_n$ & $i=2,\ldots,n-2$ &  \parbox{4in}{\begin{center} $2\alpha^\vee_1+2\alpha^\vee_2+\cdots+2\alpha^\vee_{n-1}+\alpha^\vee_n$ (long); \\ 
$-2\alpha^\vee_i-2\alpha^\vee_{i-1}-\cdots-2\alpha^\vee_{n-1}-\alpha^\vee_n$ (long); \\
$\alpha^\vee_1+2\alpha^\vee_2+2\alpha^\vee_3+\cdots+2\alpha^\vee_{n-1}+\alpha^\vee_n$ (short); \\ 
$-\alpha^\vee_{i-1}-2\alpha^\vee_{i}-2\alpha^\vee_{i+1}-\cdots-2\alpha^\vee_{n-1}-\alpha^\vee_n$ (short); \\ 
$\alpha^\vee_1+\cdots+\alpha^\vee_i+2\alpha^\vee_{i+1}+\cdots+2\alpha^\vee_{n-1}+\alpha^\vee_n$ (short); \\
$-\alpha^\vee_i$ (short).
 \end{center}}\\ \hline
$B_n$ & $i=n$ & \parbox{4in}{\begin{center} $2\alpha^\vee_1+2\alpha^\vee_2+\cdots+2\alpha^\vee_{n-1}+\alpha^\vee_n$ (long); \\ 
$-\alpha^\vee_n$ (long); \\
$\alpha^\vee_1+2\alpha^\vee_2+2\alpha^\vee_3+\cdots+2\alpha^\vee_{n-1}+\alpha^\vee_n$ (short); \\ 
$-\alpha^\vee_{n-1}-\alpha^\vee_n$ (short).
\end{center}} \\ \hline \hline
$C_n$ & $i=1$ & \parbox{4in}{\begin{center} $\alpha^\vee_1+2\alpha^\vee_2+2\alpha^\vee_3+\cdots+2\alpha^\vee_n$ (long); \\ 
$-\alpha_1^\vee$ (long); \\ 
$\alpha_1^\vee+\alpha_2^\vee+\cdots+\alpha_n^\vee$ (short); \\ 
$-\alpha_1^\vee-\alpha_2^\vee-\cdots-\alpha_n^\vee$ (short).
\end{center}}  \\ \hline
$C_n$ & $i=2,\ldots,n-2$ &  \parbox{4in}{\begin{center} $\alpha^\vee_1+2\alpha^\vee_2+2\alpha^\vee_3+\cdots+2\alpha^\vee_n$ (long); \\ 
$-\alpha^\vee_{i-1}-2\alpha^\vee_{i}-2\alpha^\vee_{i+1}-\cdots-2\alpha^\vee_n$ (long); \\ 
$\alpha^\vee_i+2\alpha^\vee_{i+1}+2\alpha^\vee_{i+2}+\cdots+2\alpha^\vee_{n}$ (long); \\ 
$-\alpha^\vee_i$ (long); \\
$\alpha_1^\vee+\alpha_2^\vee+\cdots+\alpha_n^\vee$ (short); \\ 
$-\alpha^\vee_{i}-\alpha^\vee_{i+2}-\cdots-\alpha^\vee_n$ (short). \\ 
 \end{center}}\\ \hline
 $C_n$ & $i=n$ & \parbox{4in}{\begin{center} $\alpha^\vee_1+2\alpha^\vee_2+2\alpha^\vee_3+\cdots+2\alpha^\vee_n$ (long); \\ 
$-\alpha^\vee_{n-1}-2\alpha^\vee_n$ (long); \\ 
$\alpha_1^\vee+\alpha_2^\vee+\cdots+\alpha_n^\vee$ (short); \\ 
$-\alpha^\vee_n$ (short). 
\end{center}} \\ \hline \hline
$D_n$ & $i\in\{1,n-1,n\}$ & \parbox{4in}{\begin{center} $\alpha^\vee_1+\alpha^\vee_2+\cdots+\alpha^\vee_n $; \\  
$-\alpha^\vee_i$. 
\end{center}} \\ \hline
$D_n$ & $i=2,\ldots,n-2$ & \parbox{4in}{\begin{center} $\alpha^\vee_1+2\alpha^\vee_2+2\alpha^\vee_3+\cdots+2\alpha^\vee_{n-2}+\alpha^\vee_{n-1}+\alpha^\vee_n$; \\ 
$-\alpha^\vee_{i-1}-2\alpha^\vee_{i}-2\alpha^\vee_{i}-\cdots-2\alpha^\vee_{n-2}-\alpha^\vee_{n-1}-\alpha^\vee_{n}$; \\
$\alpha^\vee_1+\cdots+\alpha^\vee_i+2\alpha^\vee_{i+1}+\cdots+2\alpha^\vee_{n-2}+\alpha^\vee_{n-1}+\alpha^\vee_{n-2}$; \\ 
$-\alpha^\vee_i$. 
\end{center}} \\ \hline
\end{tabular}
\medskip
\caption{Orbit representatives for the roots under the action of maximal parabolic subgroups for the classical types.}  \label{tab:parabolic_orbit_reps}
\end{table}

\begin{table}
\renewcommand*{\arraystretch}{2.5}
\begin{tabular}{c|c|c|c}
$\Phi$ & $i$ & max~$\langle v,\alpha^\vee\rangle$ & $v \in \mathcal{P}^{*}_{\Phi_i^\vee}$; $\alpha^\vee \in \Phi^\vee\setminus \Phi^\vee_i$  attaining max\\ \hline
\rowcolor[gray]{.9} $A_n$ & $i=1,\ldots,n$ & \parbox{1.25in}{\begin{center} $\frac{i-1}{i} + \frac{n-i}{n-i+1}$; \\(maximum at \\  $i=\lfloor \frac{n-1}{2}\rfloor+1$:\\ $\frac{\lfloor \frac{n-1}{2}\rfloor}{\lfloor \frac{n-1}{2}\rfloor+1} + \frac{\lceil \frac{n-1}{2}\rceil}{\lceil \frac{n-1}{2}\rceil+1}$) \end{center}} & $v=\omega'_{i-1}+\omega'_{i+1}$; $\alpha^\vee = -\alpha_i^\vee$ \\ \hline \hline
$B_n$ & $i=1$ & $\frac{1}{2}$  & $v=\omega'_{n}$; $\alpha^\vee =-\alpha_i^\vee$ \\ \hline
$B_n$ & $i=2,\ldots,n-1$ & $\frac{2i-2}{i}$  & \parbox{2.5in}{\begin{center}$v=\omega_{i-1}'+\omega'_{n}$; \\$\alpha^\vee =-2\alpha^\vee_i-\cdots-2\alpha^\vee_{n-1}-\alpha^\vee_n$ \end{center}}\\ \hline
\rowcolor[gray]{.9} $B_n$ & $i=n$ & $\frac{2n-2}{n}$ & $v = \omega'_{n-1}$; $\alpha^\vee =-\alpha^\vee_n$ \\ \hline \hline
$C_n$ & $i=1$ & $1$ & $v=\omega'_2$; $\alpha^\vee =-\alpha^\vee_1$ \\ \hline
\rowcolor[gray]{.9} $C_n$ & $i=2,\ldots,n-1$ & \parbox{1.25in}{\begin{center}$\frac{2i-1}{i}$; \\ (maximum at \\ $i=n-1$: $\frac{2n-3}{n-1}$) \end{center}} & $v=\omega'_{i-1}+\omega'_{i+1}$; $\alpha^\vee =-\alpha^\vee_i$ \\ \hline
$C_n$ & $i=n$ & $\frac{2n-4}{n}$ & \parbox{2.5in}{\begin{center} $v=\omega'_{2}$; \\ $\alpha^\vee =\alpha^\vee_1+2\alpha^\vee_2+2\alpha^\vee_3+\cdots+2\alpha^\vee_n$\end{center}} \\ \hline \hline
$D_n$ & $i=1$ & $1$ & $v=\omega'_2$; $\alpha^\vee =-\alpha^\vee_1$ \\ \hline
$D_n$ & $i=2,\ldots,n-3$ & $\frac{2i-1}{i}$ & $v=\omega'_{i-1}+\omega'_{i+1}$; $\alpha^\vee =-\alpha^\vee_1$ \\ \hline
\rowcolor[gray]{.9} $D_n$ & $i=n-2$ & $\frac{2n-5}{n-2}$ & $v=\omega'_{n-3}+\omega'_{n-1}+\omega'_{n}$; $\alpha^\vee =-\alpha^\vee_1$ \\ \hline
$D_n$ & $i\in\{n-1,n\}$ & $\frac{2n-4}{n}$ & $v=\omega'_{n-2}$; $\alpha^\vee =-\alpha^\vee_i$ \\ \hline
\end{tabular}
\medskip
\caption{Maximum values of~$\langle v,\alpha^\vee \rangle$ for $v\in \mathcal{P}^{*}_{\Phi_i^\vee}$ and $\alpha^\vee \in \Phi^\vee\setminus \Phi^\vee_i$ for the classical types. We use $\{\omega'_j\colon j\neq i\in[n]\}$ to denote the fundamental weights of $\Phi_i$ (with the convention $\omega'_0 \coloneqq \omega'_{n+1} \coloneqq 0$). The overall maximum for each $\Phi$ is highlighted in gray.} \label{tab:classical_maxes}
\end{table}

\begin{table}
\begin{tabular}{c|c|c}
$\Phi$ & $i$ & max~$\langle v,\alpha^\vee\rangle$ \\ \hline
\rowcolor[gray]{.9} $G_2$ & $i=1$ & $3/2$ \\ \hline
$G_2$ & $i=2$ & $1/2$ \\ \hline \hline
$F_4$ & $i=1$ & $1$ \\ \hline
$F_4$ & $i=2$ & $5/3$ \\ \hline
\rowcolor[gray]{.9} $F_4$ & $i=3$ & $11/6$ \\ \hline
$F_4$ & $i=4$ & $3/2$ \\ \hline \hline
$E_6$ & $i=1$ & $5/4$ \\ \hline
$E_6$ & $i=2$ & $3/2$ \\ \hline
$E_6$ & $i=3$ & $17/10$ \\ \hline
\rowcolor[gray]{.9} $E_6$ & $i=4$ & $11/6$ \\ \hline
$E_6$ & $i=5$ & $17/10$ \\ \hline
$E_6$ & $i=6$ & $5/4$ \\ \hline
\end{tabular} \qquad
\begin{tabular}{c|c|c}
$\Phi$ & $i$ & max~$\langle v,\alpha^\vee\rangle$ \\ \hline
$E_7$ & $i=1$ & $3/2$ \\ \hline
$E_7$ & $i=2$ & $12/7$ \\ \hline
$E_7$ & $i=3$ & $11/6$ \\ \hline
\rowcolor[gray]{.9} $E_7$ & $i=4$ & $23/12$ \\ \hline
$E_7$ & $i=5$ & $28/15$ \\ \hline
$E_7$ & $i=6$ & $7/4$ \\ \hline
$E_7$ & $i=7$ & $4/3$ \\ \hline \hline
$E_8$ & $i=1$ & $7/4$ \\ \hline
$E_8$ & $i=2$ & $15/8$ \\ \hline
$E_8$ & $i=3$ & $27/14$ \\ \hline
\rowcolor[gray]{.9} $E_8$ & $i=4$ & $59/30$ \\ \hline
$E_8$ & $i=5$ & $39/20$ \\ \hline
$E_8$ & $i=6$ & $23/12$ \\ \hline
$E_8$ & $i=7$ & $11/6$ \\ \hline
$E_8$ & $i=8$ & $3/2$ \\ \hline
\end{tabular}
\medskip
\caption{Maximum values of~$\langle v,\alpha^\vee \rangle$ for $v\in \mathcal{P}^{*}_{\Phi_i^\vee}$ and $\alpha^\vee \in \Phi^\vee\setminus \Phi^\vee_i$ for the exceptional types. The overall maximum for each $\Phi$ is highlighted in gray.} \label{tab:exceptional_maxes}
\end{table}

\begin{remark}
Let $\kappa\geq 1$ be minimal such that $\pi_U(\mathcal{P}_{\Phi^\vee}) \subseteq \kappa\cdot \mathcal{P}_{\Phi_U^\vee}$ for all nonzero subspaces $\{0\}\neq U\subseteq V$ spanned by a subset of~$\Phi^\vee$. We just proved that $\kappa < 2$. How close to~$2$ does it get? Looking at Table~\ref{tab:classical_maxes}, we see that it gets arbitrarily close to $2$. Moreover, $(2-\kappa)$ appears to be on the order of $n^{-1}$. Therefore, one might wonder how the quantity $(\textrm{rank of $\Phi$})\times (2-\kappa)$ behaves for various~$\Phi$. This information is recorded in Table~\ref{tab:rank_kappa}. The root system which minimizes the quantity $(\textrm{rank of $\Phi$})\times (2-\kappa)$ (i.e., the root system which ``comes closest'' to breaking Lemma~\ref{lem:root_projection}, at least relative to its rank) is $\Phi=E_8$. Hence we have added to the long list of ways in which $E_8$ is the ``most exceptional'' root system (see, e.g.~\cite{garibaldi2016e8}).
\end{remark}

\begin{table}
\renewcommand*{\arraystretch}{2.5}
\begin{tabular}{c|c}
$\Phi$ & $(\textrm{rank of $\Phi$})\times(2-\kappa)$ \\ \hline
$A_n$ & always $\geq 2$ (for $n\geq 2$), and $\approx 4$ as $n\to \infty$ \\ \hline
$B_n$ & $2$ \\ \hline
$C_n$ & always $>1$, but $\approx 1$ as $n \to \infty$ \\ \hline
$D_n$ & always $>1$, but $\approx 1$ as $n \to \infty$ \\ \hline
$G_2$ & $1$ \\ \hline
$F_4$ & $\frac{2}{3}$ \\ \hline
$E_6$ & $1$ \\ \hline
$E_7$ & $\frac{7}{12}$ \\ \hline
$E_8$ & $\frac{8}{30}$ \\ \hline
\end{tabular}
\medskip
\caption{Rank times $(2-\kappa)$ for all the irreducible root systems, where~$\kappa \geq 1$ is minimal such that $\mathcal{P}^*_{\Phi_U^\vee} \subseteq \kappa \cdot \mathcal{P}^*_{\Phi^\vee}$ for all nonzero subspaces $\{0\}\neq U\subseteq V$ spanned by a subset of~$\Phi^\vee$.} \label{tab:rank_kappa}
\end{table}

\bibliography{ehrhart_formula}{}
\bibliographystyle{plain}

\end{document}